\documentclass[a4paper,twoside,10pt]{article}
\usepackage[T1]{fontenc}

\usepackage{amsfonts}
\usepackage{amsmath}
\usepackage{amssymb}
\usepackage{geometry}
\usepackage{sectsty}
\usepackage{lineno}
\usepackage{theorem}
\usepackage{fancyhdr}

\usepackage{graphicx,subfigure}
\usepackage{multirow,multicol}
\usepackage{url}

\usepackage{array}

\usepackage[table]{xcolor}
\newcolumntype{L}{>{$}l<{$}}
\newcolumntype{C}{>{$}c<{$}}
\definecolor{lgray}{gray}{0.8}
\providecommand{\keywords}[1]{{\textit{Key words and phrases---}} #1}
\providecommand{\subjclass}[1]{{\textit{AMS 2000 subject classifications---}} #1}

\theoremstyle{plain}
\newtheorem{thrm}{Theorem}[section]
\newtheorem{lmm}[thrm]{Lemma}
\newtheorem{crllr}[thrm]{Corollary}
\newtheorem{prpstn}[thrm]{Proposition}

\newtheorem{rmrk}[thrm]{Remark}

\newtheorem{rmk}{Remark}[section]

\newenvironment{proof}[1][Proof]{\noindent\textbf{#1.} }{\ \rule{0.5em}{0.5em}}
\usepackage {geometry}
\usepackage[normalem]{ulem}
\usepackage{hyperref}
\hypersetup{
    colorlinks,%
    citecolor=blue,%
    filecolor=black,%
    linkcolor=black,%
    urlcolor=blue
}
\usepackage[numbers]{natbib}
\usepackage[normalem]{ulem}

\DeclareMathOperator{\pen}{pen}
\DeclareMathOperator{\crit}{crit}
\DeclareMathOperator{\leb}{Leb}
\DeclareMathOperator{\card}{Card}
\DeclareMathOperator{\var}{Var}

\DeclareMathOperator{\Span}{Span}

\DeclareMathOperator{\VFCV}{VFCV}
\DeclareMathOperator{\p}{p}
\DeclareMathOperator{\emp}{emp}
\DeclareMathOperator*{\supp}{supp}

\title{\textbf{\scshape Slope heuristics and V-Fold model selection\\ in heteroscedastic regression\\ using strongly localized bases}}
\author{{\scshape Fabien Navarro and Adrien Saumard}\\\\
{CREST-ENSAI, BRUZ, FRANCE}\\
}
\date{}

\usepackage{fancyhdr}
\begin{document}
\maketitle

\begin{abstract}
We investigate the optimality for model selection of the so-called
slope heuristics, $V$-fold cross-validation and $V$-fold penalization in a heteroscedatic with random design
regression context. We consider a new class of linear models that we call
strongly localized bases and that generalize histograms, piecewise
polynomials and compactly supported wavelets.
We derive sharp oracle inequalities that prove the asymptotic optimality of
the slope heuristics---when the optimal penalty shape is known---and $V$%
-fold penalization. Furthermore, $V$-fold cross-validation seems to be
suboptimal for a fixed value of $V$ since it recovers asymptotically the
oracle learned from a sample size equal to $1-V^{-1}$ of the original amount
of data.
Our results are based on genuine concentration inequalities for the true and
empirical excess risks that are of independent interest.
We show in our experiments the good behavior of the slope heuristics
for the selection of linear wavelet models. Furthermore, $V$-fold
cross-validation and $V$-fold penalization have comparable efficiency.
\end{abstract}
%
%\begin{resume} ... \end{resume}
%

\noindent\keywords{Nonparametric regression, heteroscedastic noise, random design, model selection,  cross-validation, wavelets}
\maketitle
\tableofcontents

\section*{Introduction}
The main goal of this paper is to substantially extend the study, in a
heteroscedastic regression with random design context, of the \textit{%
optimality} of two general model selection devices: the so-called slope
heuristics and V-fold resampling strategies. More precisely, we consider
projection estimators on some general linear models and investigate from a
theoretical perspective the possibility to derive optimal oracle
inequalities for the considered model selection procedures. We also
experiment and compare the procedures for the selection of linear wavelet
models.

The slope heuristics \cite{BirMas:07} is a recent calibration method of
penalization procedures in model selection: from the knowledge of a (good)
penalty shape it allows to calibrate a penalty that performs an accurate
model selection. It is based on the existence of a minimal penalty, around
which there is a drastic change in the behavior of the model selection
procedure. Moreover, the optimal penalty is simply linked to the minimal one
by a factor two. The slope heuristics is thus a general method for the
selection of M-estimators \cite{ArlotMassart:09} and it has been
successfully applied in various methodological studies surveyed in \cite%
{BauMauMich:12}.

However, there is a gap between the wide range of applicability of the slope
heuristics and its theoretical justification. Indeed, there are only a few
studies, in quite restrictive frameworks, that theoretically describe the
optimality of this penalty calibration procedure. First, Birg\'{e} and
Massart \cite{BirMas:07} have shown the validity of the slope heuristics in
a generalized linear Gaussian model setting, including the case of
homoscedastic regression with fixed design. Then, Arlot and Massart \cite%
{ArlotMassart:09} validated the slope heuristics in a heteroscedastic with
random design regression framework, for the selection of linear models of
histograms. These result has been extended to the case of piecewise
polynomial functions in \cite{saum:13}. Lerasle 
%TCIMACRO{\TeXButton{\cite{Ler:11,Ler:12}}{\cite{Ler:11,Ler:12}} }%
%BeginExpansion
\cite{Ler:11,Ler:12}
%EndExpansion
has shown the optimality of the slope heuristics in least-squares density
estimation for the selection of some linear models for both independent and
dependent data. It has also been shown in \cite{Saum:10c}---refining
previous partial results of \cite{Castellan:99}---that the slope heuristics
is valid for the selection of histograms in maximum likelihood density
estimation. On the negative side, Arlot and Bach \cite{Arl_Bac:2009} proved
that the constant two between the minimal penalty and the optimal one is not always valid for the selection of linear estimators in least-squares
regression with fixed design. For instance, kernel ridge regression leads to a ratio between the optimal penalty and the minimal one that takes values between $1$ and $2$. The existence of a minimal penalty---that can be estimated in practice---seems to be general however, even for the selection of linear estimators.

If the noise is homoscedastic, then the shape of the ideal penalty is known
and is linear in the dimension of the models as in the case of Mallows' $%
C_{p}$. However, if the noise is heteroscedastic, then Arlot \cite{Arl:2010}
showed that the ideal penalty is not in general a function of the linear
dimension of the models. Hence, it is likely that finding a good penalty
shape in order to use the slope heuristics will be hard and another approach
would be needed. Probably, the most commonly used method to select an
hyperparameter---such as the linear dimension of the models in our problem---in practice is the $V$-fold cross-validation (VFCV) procedure \cite{Geisser:75},
with $V$ classically taken to be equal to $5$ or $10$.

Despite its wide success in practice, there is still quite few theoretical
results concerning VFCV, that are surveyed in \cite%
{ArlotCelisse:10}. Some asymptotic results are described in \cite%
{GLKKWbook:02}. Some papers more specifically address the efficiency of VFCV as a model selection tool by deriving oracle
inequalities. But most results, such as in \cite{LecuMitchell:2012} in a
general learning context or in \cite{Wegkamp:03} for least-squares
regression, do not allow to tackle the question of the \textit{optimality}
of the procedure as a model selection tool, since they prove oracle
inequalities with unknown or suboptimal leading constant. A notable
exception is \cite{Arl:2008a}, which proves that VFCV
for a fixed $V$ is indeed asymptotically suboptimal for the selection of
regressograms. This is simply explained by the fact that VFCV gives a biased estimation of the risk, as emphasized
earlier by Burman \cite{Burman:89}, who proposed to remove this bias.

Building on ideas of \cite{Burman:89}, Arlot \cite{Arl:2008a} defined the
so-called $V$-fold penalization and proved its asymptotic optimality, even
for fixed $V$, for the selection of histograms. In particular, the procedure adapts to the
heteroscedasticity of the noise, a property of $V$-fold techniques also putted on
emphasis in \cite{ArlotCelisse:11}\ in the context of change-point detection.
The idea is that $V$-fold penalization gives an unbiased estimate of the
risk by adding to the empirical risk a cross-validated estimate of the ideal
penalty. However, in practice $V$-fold penalization and cross-validation
roughly give the same accuracy, since the over-penalization performed by
cross-validation can actually be an advantage when the sample size is small
to moderate. Concerning the choice of $V$ in either VFCV or penalization, Arlot and Lerasle \cite%
{Arl_Ler:2012:penVF:JMLR}\ recently justified in a least-squares density
estimation context that the choice of $V=5$ or $10$ is a reasonable choice.

The theoretical investigation of optimality of either the slope heuristics
or $V$-fold strategies will be based, among other things, on sharp results
that describe the\textit{\ concentration} of the true and the empirical
excess risks when the model is fixed---but with dimension allowed to depend
on the sample size. Since the excess risk of an empirical risk minimizer is
a central object of the theory of statistical learning, such concentration
result and subsequent optimal upper and lower bounds for the excess risk of
least-squares estimators are of independent interest. Moreover, excess
risk's concentration around a single deterministic point is an exciting new
direction of research that refines more classical excess risk bounds. It recently gained interest after the work of Chatterjee \cite{chatterjee2014},
proving concentration inequalities for excess risk in least-squares
regression under convex constraint and deducing universal admissibility of
least-squares estimation in this context.

It is worth noting that one of the main arguments developed in \cite%
{chatterjee2014} and leading to excess risk's concentration is a formula
expressing the excess risk as the maximizer of a functional related to local
suprema of a Gaussian process. In fact, such a \textit{representation} of
the excess risk of a general M-estimator in terms of an empirical process
appeared earlier in Saumard \cite{saum:12}---see Remark 1 of Section 3
therein---and was also used to prove concentration inequalities for
the excess risk of a projection estimator in least-squares regression.
Building on \cite{chatterjee2014}, Muro and van de Geer \cite%
{MurovandeGeer:15} recently proved concentration inequalities for the excess
risk in regularized least-squares regression and van de Geer and Wainwright 
\cite{vandeGeerWain:16} proposed a generic framework of regularized
M-estimation allowing to derive excess risk's concentration. These studies
are also both based on excess risk's representation in terms of either a
Gaussian or an empirical process.

Let us now detail our contributions:

\begin{itemize}
\item We propose a new analytical property, allowing to deal with a lot of
functional bases, that we call \textit{strongly localized basis}. We show
that it is a refinement on the classical concept of localized basis \cite%
{BirgeMassart:98}, that encompasses the cases of histograms, piecewise
polynomials and compactly supported wavelets. We prove better results for
strongly localized bases than for localized bases, while
all known examples of localized bases are in fact strongly localized.
Therefore, the concept of strongly localized basis is a way to describe some
functional bases that is of independent interest and that could be used in
many other nonparametric settings.

\item We substantially extend the theoretical analysis of the slope
heuristics, generalizing the results of 
%TCIMACRO{%
%\TeXButton{\cite{ArlotMassart:09,saum:13}}{\cite{ArlotMassart:09,saum:13}} }%
%BeginExpansion
\cite{ArlotMassart:09,saum:13}
%EndExpansion
to the case of strongly localized bases.

\item We prove sharp oracle inequalities for the $V$-fold cross validation
with fixed $V$, showing that it asymptotically recovers an oracle model
learned with a fraction equal to $1-V^{-1}$ of the original amount of data.
Then we improve on these bounds by considering $V$-fold penalization, which
satisfies optimal oracle inequalities. By proving such a result, we generalize
a previous study of Arlot \cite{Arl:2008a}, from the case of histograms to
the case of strongly localized bases.

\item We prove concentration bounds for the excess risk of projection
estimators, that are of independent interest. These results are
based on previous work \cite{saum:12} and on a new approach to sup-norm
consistency. We indeed generalize previous representation formulas in terms
of empirical process for the excess risk of a (regularized) M-estimator
obtained in 
%TCIMACRO{%
%\TeXButton{\cite{saum:12,vandeGeerWain:16}}{\cite{saum:12,vandeGeerWain:16}} }%
%BeginExpansion
\cite{saum:12,vandeGeerWain:16}
%EndExpansion
to \textit{any functional} of a M-estimator and use it to obtain bounds in
sup-norm for projection estimators on strongly localized bases. These new
representation formulas are also of independent interest, since they are
totally general in M-estimation.

\item We show in our experiments the good behavior of the slope
heuristics for the selection of linear wavelet models. Indeed, it often
compares favorably to VFCV and penalization. In addition, Mallows' $C_p$ seems to be also efficient. We also recover in our more general framework some previous
observations of Arlot \cite{Arl:2008a}: even if the $V$-fold penalization
has better theoretical guarantees than the $V$-fold cross validation, it has
only comparable efficiency in practice.
\end{itemize}

The paper is organized as follows. In Section \ref{section_framework_reg
copy(1)}, we describe the statistical framework. The concept of strongly
localized basis is presented in Section \ref{section_strong_loc_bas}. The
slope heuristics is validated in Section \ref{section_slope_heuristics}, and 
$V$-fold strategies are considered in Section \ref%
{section_Vfold_model_selection}. Then we expose our results for a fixed
model, that are of independent interest, in Section \ref%
{section_excess_risk_concentration}. Numerical experiments are detailed in
Section \ref{section_experiments}. The proofs are postponed to Section \ref%
{section_proof_slope_reg}.

\section{Statistical framework\label{section_framework_reg copy(1)}}

We consider $n$ independent observations $\xi _{i}=\left( X_{i},Y_{i}\right)
\in \mathcal{X\times }\mathbb{R}$ with common distribution $P$, as well as a
generic random variable $\xi =\left( X,Y\right) $, independent of the sample 
$\left( \xi _{1},\ldots,\xi _{n}\right) $, following the same distribution $P$.
The \textit{feature space} $\mathcal{X}$ is a subset of $\mathbb{R}^{d}$, $%
d\geq 1$. The marginal distribution of $X_{i}$ is denoted $P^{X}$. We assume
that the following relation holds,%
\begin{equation*}
Y=s_{\ast }\left( X\right) +\sigma \left( X\right) \varepsilon \text{ },
%\label{regression_model_2}
\end{equation*}%
where $s_{\ast }\in L_{2}\left( P^{X}\right) $ is the \textit{regression
function }of $Y$ with respect to $X$ to be estimated. Conditionally to $X$,
the residual $\varepsilon $ is normalized, i.e. it has mean zero and
variance one. The function $\sigma :\mathcal{X\rightarrow }\mathbb{R}_{+}$
is the unknown \textit{heteroscedastic noise level}.

To estimate $s_{\ast }$, we consider a finite \textit{collection of models} $%
\mathcal{M}_{n}$, with cardinality depending on the sample size $n$. Each
model $m\in \mathcal{M}_{n}$ will be a finite-dimensional vector space of 
\textit{linear dimension} $D_{m}$. The models that we consider in this paper
are more precisely defined in Section \ref{section_strong_loc_bas} below.

We write $\left\Vert s\right\Vert _{2}=\left( \int_{\mathcal{X}%
}s^{2}dP^{X}\right) ^{1/2}$ the quadratic norm in $L_{2}\left( P^{X}\right) $
and $s_{m}$ the\textit{\ orthogonal projection} of $s_{\ast }$ onto $m$ in
the Hilbert space $\left( L^{2}\left( P^{X}\right) ,\left\Vert \cdot
\right\Vert _{2}\right) $. For a function $f\in L_{1}\left( P\right) $, we
write $P(f)=Pf=\mathbb{E}\left[ f\left( \xi \right) \right] $. By setting $%
\gamma :L_{2}\left( P^{X}\right) \rightarrow L_{1}\left( P\right) $ the 
\textit{least-squares contrast}, defined by 
\begin{equation*}
\gamma \left( s\right) :\left( x,y\right) \mapsto \left( y-s\left( x\right)
\right) ^{2}\text{ , \ \ \ \ }s\in L_{2}\left( P^{X}\right) ,
%\label{def_contrast}
\end{equation*}%
the regression function $s_{\ast }$ is characterized by the following
relation, 
\begin{equation*}
s_{\ast }=\arg \min_{s\in L_{2}\left( P^{X}\right) }P\left( \gamma \left(
s\right) \right) .  
%\label{def_target}
\end{equation*}%
The projections $s_{m}$ also satisfy,%
\begin{equation*}
s_{m}=\arg \min_{s\in m}P\left( \gamma \left( s\right) \right) .
%\label{def_projection}
\end{equation*}%
For each model $m\in \mathcal{M}_{n}$, we consider a \textit{least-squares
estimator} $\widehat{s}_{m}$ (possibly non unique), satisfying 
\begin{align*}
\widehat{s}_{m}& \in \arg \min_{s\in m}\left\{ P_{n}\left( \gamma \left(
s\right) \right) \right\} \\
& =\arg \min_{s\in m}\left\{ \frac{1}{n}\sum_{i=1}^{n}\left( Y_{i}-s\left(
X_{i}\right) \right) ^{2}\right\} ,
\end{align*}%
where $P_{n}=n^{-1}\sum_{i=1}^{n}\delta _{\xi _{i}}$ is the \textit{%
empirical measure} built from the data.

The performance of the least-squares estimators is tackled through their 
\textit{excess loss},%
\begin{equation*}
\ell \left( s_{\ast },\widehat{s}_{m}\right) :=P\left( \gamma \left( \widehat{s}%
_{m}\right) -\gamma \left( s_{\ast }\right) \right) =\left\Vert \widehat{s}%
_{m}-s_{\ast }\right\Vert _{2}^{2}.
\end{equation*}%
We split the excess risk into a sum of two terms,%
\begin{equation*}
\ell \left( s_{\ast },\widehat{s}_{m}\right) =\ell \left( s_{\ast },s_{m}\right)
+\ell \left( s_{m},\widehat{s}_{m}\right) ,
\end{equation*}%
where%
\begin{equation*}
\ell \left( s_{\ast },s_{m}\right) :=P\left( \gamma \left( s_{m}\right)
-\gamma \left( s_{\ast }\right) \right) =\left\Vert s_{m}-s_{\ast
}\right\Vert _{2}^{2}\text{ \ \ \ and \ \ \ }\ell \left( s_{m},\widehat{s}%
_{m}\right) :=P\left( \gamma \left( \widehat{s}_{m}\right) -\gamma \left(
s_{m}\right) \right) \geq 0.
\end{equation*}%
The quantity $\ell \left( s_{\ast },s_{m}\right) $ is a deterministic term
called the \textit{bias} of the model $m$, while $\ell \left( s_{m},\widehat{s}%
_{m}\right) $ is a random variable that we call the \textit{excess risk }of
the least-squares estimator $\widehat{s}_{m}$ \textit{on the model} $m$. Notice
that by the Pythagorean theorem, it holds%
\begin{equation*}
\ell \left( s_{m},\widehat{s}_{m}\right) =\left\Vert \widehat{s}_{m}-s_{m}\right%
\Vert _{2}^{2}.
\end{equation*}

Having at hand the collection of models $\mathcal{M}_{n}$, we want to
construct an estimator whose excess risk is as close as possible to the
excess risk of an \textit{oracle model} $m_{\ast }$,%
\begin{equation}
m_{\ast }\in \arg \min_{m\in \mathcal{M}_{n}}\left\{ \ell \left( s_{\ast },%
\widehat{s}_{m}\right) \right\} .  \label{oracle_model}
\end{equation}%
We propose to perform this task \textit{via} a penalization procedure: given
some penalty $%
%TCIMACRO{\TeXButton{pen}{\pen}}%
%BeginExpansion
\pen%
%EndExpansion
$, that is a function from $\mathcal{M}_{n}$ to $\mathbb{R}^+$, we consider
the following \textit{selected model},%
\begin{equation}
\widehat{m}\in \arg \min_{m\in \mathcal{M}_{n}}\left\{ P_{n}\left( \gamma
\left( \widehat{s}_{m}\right) \right) +%
%TCIMACRO{\TeXButton{pen}{\pen}}%
%BeginExpansion
\pen%
%EndExpansion
\left( m\right) \right\} \text{ }.  \label{def_proc_2_reg}
\end{equation}%
The goal is then to find a good penalty, such that the selected model $%
\widehat{m}$ satisfies an \textit{oracle inequality} of the form%
\begin{equation}
\ell \left( s_{\ast },\widehat{s}_{\widehat{m}}\right) \leq C\times \inf_{m\in 
\mathcal{M}_{n}}\ell \left( s_{\ast },\widehat{s}_{m}\right) ,
\label{def_oracle_ineq}
\end{equation}%
with probability close to one and with some constant $C\geq 1$, as close to
one as possible.

\section{Strongly localized bases\label{section_strong_loc_bas}}

We define here the analytic constraints that we need to put on the models in
order to derive our model selection results. We also provide various
examples of such models.

\subsection{Definition}%\label{ssection_definition}

Let us take a finite-dimensional model $m$ with linear dimension $D_{m}$ and
an orthonormal basis $\left( \varphi _{k}\right) _{k=1}^{D_{m}}$. The family 
$\left(\varphi _{k}\right) _{k=1}^{D_{m}}$ is called a \textit{strongly
localized basis} (with respect to the probability measure $P^{X}$) if the
following assumption is satisfied:

\begin{description}
\item[(\textbf{Aslb})] there exist $r_{m}>0$, $b_m\in \mathbb{N}_{\ast }$, a
partition $\left( \Pi _{i}\right) _{i=1}^{b_m}$ of $\left\{
1,\ldots,D_{m}\right\} $, positive constants $\left( A_{i}\right) _{i=1}^{b_m}$
and an orthonormal basis $\left(\varphi _{k}\right) _{k=1}^{D_{m}}$ of $%
\left( m,\left\Vert 
%TCIMACRO{\TeXButton{.}{\cdot}}%
%BeginExpansion
\cdot%
%EndExpansion
\right\Vert _{2}\right) $ such that $1\leq A_{1}\leq A_{2}\leq \ldots\leq
A_{b_m}<+\infty $,%
\begin{equation}
\sum_{i=1}^{b_m}\sqrt{A_{i}}\leq r_{m}\sqrt{D_{m}},  \label{def_Ai}
\end{equation}%
and%
\begin{equation}
\text{for all }i\in \left\{ 1,\ldots,b_m\right\} \text{, for all }k\in \Pi _{i}%
\text{, }\left\Vert \varphi _{k}\right\Vert _{\infty }\leq r_{m}\sqrt{A_{i}}%
.  \label{loc_plus}
\end{equation}%
Moreover, for every $\left( i,j\right) \in \left\{ 1,\ldots,b_m\right\} ^{2}$ and 
$k\in \Pi _{i}$, we set 
\begin{equation*}
\Pi _{j\left\vert k\right. }=\left\{ l\in \Pi _{j} ; \supp%
\left( \varphi _{k}\right) \bigcap \supp\left( \varphi _{l}\right)
\neq \emptyset \right\}
\end{equation*}%
and we assume that there exists a positive constant $A_{c}$ such that for
all $j\in \left\{ 1,\ldots,b_m\right\} $,%
\begin{equation}
\max_{k\in \Pi _{i}}%
%TCIMACRO{\TeXButton{Card}{\card}}%
%BeginExpansion
\card%
%EndExpansion
\left( \Pi _{j\left\vert k\right. }\right) \leq A_{c}\left(
A_{j}A_{i}^{-1}\vee 1\right) .  \label{def_card_Pi_i_k}
\end{equation}
\end{description}

Up to our knowledge, the concept of strongly localized basis is new. In (\ref%
{loc_plus}), we ask for a control in sup-norm of each element of the
considered basis. We also require in (\ref{def_card_Pi_i_k}) a control of
the number of intersections between the supports of the elements of the
considered orthonormal basis.

As shown in Section \ref{ssection_examples} below, the property of strongly
localized basis allows to unify the treatment of some models of histograms,
piecewise polynomials and compactly supported wavelets. From this point of view, we may interpret the parameter $b_m$ as the number "scales" in the basis, which in particular equals one for histograms and piecewise polynomials. It is also equal to the number of resolutions in the multi-resolution analysis associated to wavelet models. See Section \ref{ssection_examples} below for details about these examples.

The classical concept of localized basis (Birg\'{e} and Massart \cite%
{BirgeMassart:98}) also covers the previous examples. More precisely, recall
that an orthonormal basis $\left( \varphi _{k}\right) _{k=1}^{D_{m}}$ of $%
\left( m,\left\Vert 
%TCIMACRO{\TeXButton{.}{\cdot}}%
%BeginExpansion
\cdot%
%EndExpansion
\right\Vert _{2}\right) $ is a\textit{\ localized basis} if there exists $%
r_{\varphi }>0$ such that%
\begin{equation*}
\text{for all }\beta =\left( \beta _{k}\right) _{k=1}^{D_{m}}\in \mathbb{R}%
^{D_{m}}\text{, }\left\Vert \sum_{k=1}^{D_{m}}\beta _{k}\varphi
_{k}\right\Vert _{\infty }\leq r_{\varphi }\sqrt{D_{m}}\max_{k\in \left\{
1,\ldots,D_{m}\right\} }\left\vert \beta _{k}\right\vert .
\label{def_localized_basis}
\end{equation*}%
In fact, we show in the next proposition that strongly localized bases are
localized in the classical sense. The interest of strongly localized bases
over localized bases then comes from the fact that it allows to derive
concentration bounds for the excess risks for models with dimension much
larger than what we can prove with localized bases (from $D_{m}\ll n^{1/3}$
for localized bases to $D_{m}\ll n$ for strongly localized ones). This point
is detailed in Section \ref{ssection_strong_loc_bas}.

\begin{prpstn}
\label{prop_strong_loc_loc}If an orthonormal basis $\left( \varphi
_{k}\right) _{k=1}^{D_{m}}$ is strongly localized, then it is localized.
More precisely, if $\left( \varphi _{k}\right) _{k=1}^{D_{m}}$ satisfies (%
\textbf{Aslb}), then for every $\beta =\left( \beta _{k}\right)
_{k=1}^{D_{m}}\in \mathbb{R}^{D_{m}}$,%
\begin{eqnarray*}
\left\Vert \sum_{k=1}^{D_{m}}\beta _{k}\varphi _{k}\right\Vert _{\infty }
&\leq &A_{c}r_{m}\sum_{i=1}^{b_m}\sqrt{A_{i}}\max_{l\in \Pi _{i}}\left\vert
\beta _{l}\right\vert \\
&\leq &A_{c}r_{m}^{2}\sqrt{D_{m}}\max_{k\in \left\{ 1,\ldots,D_{m}\right\}
}\left\vert \beta _{k}\right\vert .
\end{eqnarray*}%
Reciprocally, if $\left( \varphi _{k}\right) _{k=1}^{D_{m}}$ is a localized
basis as in (\ref{def_localized_basis}), then it achieves (\ref{def_Ai}) and (\ref{loc_plus}) above with $b=1$, $A_{1}=D_{m}$ and $r_m=\max\{r_\varphi,1\}$.
\end{prpstn}

Proposition \ref{prop_strong_loc_loc} shows that the parameter $r_m$ appearing in the definition of a strongly localized basis is closely related to the parameter $r_\varphi$ defining a localized basis.\\
The proof of Proposition \ref{prop_strong_loc_loc} can be found in Section %
\ref{ssection_proofs_loc_bas}. 

\subsection{Examples\label{ssection_examples}}

We investigate here the scope of the concept of strongly localized basis by
providing some examples of linear models achieving this condition.

\subsubsection{Histograms and piecewise polynomials}

\noindent It is proved in \cite{saum:12} that linear models of histograms
and more general piecewise polynomials with bounded degree are localized
bases in $L_{2}\left( P^{X}\right) $ if the underlying partition $\mathcal{P}
$ of $\mathcal{X}$ is lower-regular in the sense that there exists a
constant $c_{m}$ such that 
\begin{equation*}
0<c_{m}<\sqrt{\left\vert \mathcal{P}\right\vert \inf_{I\in \mathcal{P}%
}P^{X}\left( I\right) }\text{ }.
\end{equation*}%
More precisely, if $r\in \mathbb{N}$ is the maximal degree of the piecewise
polynomials---$r=0$\ in the case of histograms---then any orthonormal basis 
$\left\{ \varphi _{I,j},\text{ }I\in \mathcal{P},\text{ }j\in \left\{
0,\ldots,r\right\} \right\} $ of $\left( m,\left\Vert 
%TCIMACRO{\TeXButton{.}{\cdot}}%
%BeginExpansion
\cdot%
%EndExpansion
\right\Vert _{2}\right) $ such that for all $j\in \left\{ 0,\ldots,r\right\} $, 
$\varphi _{I,j}$ is supported by the element $I$ of $\mathcal{P}$, is
localized. Hence, by Proposition \ref{prop_strong_loc_loc}, it achieves
Inequalities (\ref{def_Ai}) and (\ref{loc_plus}) of the definition of
strongly localized basis with $b=1$ and $A_{1}=D_{m}=(r+1)$Card$\left( 
\mathcal{P}\right) $. It is also immediately seen that such basis achieves
in this case (\ref{def_card_Pi_i_k}) with $A_{c}=r+1$. Furthermore, $r_m=c_m^{-1}$ for histograms (Lemma 4, \cite{saum:12}) and it has a more complicated expression for piecewise polynomials (Lemma 7, \cite{saum:12}). As a result, models
of histograms and piecewise polynomials with bounded degree and underlying
lower-regular partition $\mathcal{P}$ of $\mathcal{X}$ are endowed with a
strongly localized structure.

\subsubsection{Compactly supported wavelet expansions\label{section_Haar}}

\noindent We assume here that $\mathcal{X=}\left[ 0,1\right] $ and take $%
b_m\in \mathbb{N}^\star$. For details about wavelets and interactions with
Statistics, we refer to \cite{HKPT:98}. Set $\phi _{0}$ the father wavelet
and $\psi _{0}$ the mother wavelet. For every integers $j\geq 0$, $1\leq
k\leq 2^{j}$, define%
\begin{equation*}
\psi _{j,k}:x\mapsto 2^{j/2}\psi _{0}\left( 2^{j}x-k+1\right) .
\end{equation*}%
As explained in \cite{CohenDaubVial:93}, there are many ways to consider
wavelets on the interval. We will consider here one of the most classical
solution, that consists of using "periodized" wavelets. To this aim, we
associate to a function $\psi$ on $\mathbb{R}$, the $1$-periodic function%
\begin{equation*}
\psi^{\text{per}}\left( x\right) =\sum_{p\in \mathbb{Z}}\psi \left(
x+p\right) .
\end{equation*}%
Notice that if $\psi$ has a compact support, then the sum at the right-hand
side of the latter inequality is finite for any $x$.

We set for every integers $i,j,l\geq 0$, satisfying $i\leq j$ and $1\leq
l\leq 2^{i}$,%
\begin{eqnarray*}
\Lambda \left( j\right)  &=&\left\{ \left( j,k\right) \text{ };\text{ }1\leq
k\leq 2^{j}\right\} ,  \\%\label{def_gamma_j} 
\Lambda \left( j,i,l\right)  &=&\left\{ \left( j,k\right) \text{ };\text{ }%
2^{j-i}\left( l-1\right) +1\leq k\leq 2^{j-i}l\right\} .
%\label{def_gamma_j_i_l}
\end{eqnarray*}%
Moreover, we set $\psi _{-1,k}\left( x\right) =\phi _{0}\left( x-k+1\right) $%
, 
\begin{equation*}
\Lambda \left( -1\right) =\left\{ \left( -1,k\right) ; \supp\left(
\psi _{-1,k}\right) \cap \left[ 0,1\right] \not=\emptyset \right\} \text{ \
and \ }\Lambda _{b_m}=\bigcup\limits_{j=-1}^{b_m}\Lambda \left( j\right) \text{ .%
}
\end{equation*}%
Notice that for every integers $i,j\geq 0$ such that $i\leq j$, $\left\{
\Lambda \left( j,i,l\right) \text{ };\text{ }1\leq l\leq 2^{i}\right\} $ is
a partition of $\Lambda \left( j\right) $, which means that%
\begin{equation*}
\Lambda \left( j\right) =\bigcup\limits_{l=1}^{2^{i}}\Lambda \left(
j,i,l\right) \text{ and for all }1\leq l,h\leq 2^{i}\text{, }\Lambda \left(
j,i,l\right) \bigcap \Lambda \left( j,i,h\right) =\emptyset .
\end{equation*}%
We consider the model%
\begin{equation}
m=%
%TCIMACRO{\TeXButton{Span}{\Span}}%
%BeginExpansion
\Span%
%EndExpansion
\left\{ \psi _{\lambda }^{\text{per}}\text{ };\text{ }\lambda \in \Lambda
_{b_m}\right\} .  \label{model_wav}
\end{equation}%
Notice that the linear dimension $D_{m}$ of $m$ satisfies $D_{m}=2^{b_m+1}$.

\begin{prpstn}
\label{prop_wav}With the notations above, if $\phi _{0}$ and $\psi _{0}$ are
compactly supported, then $\left\{ \psi _{\lambda }^{\text{per}}\text{ };%
\text{ }\lambda \in \Lambda _{b_m}\right\} $ is a strongly localized basis on $%
\left( \left[ 0,1\right] ,%
%TCIMACRO{\TeXButton{Leb}{\leb}}%
%BeginExpansion
\leb%
%EndExpansion
\right) $, with parameters $b_m$ as defined above,
$A_{j}=2^{j}$ for $j\geq 0$ and $A_{-1}=1$ (an explicit value of $r_m$ is also given in the proof, but is more complicated).
\end{prpstn}

The proof of Proposition \ref{prop_wav} can be found in Section \ref%
{ssection_proofs_loc_bas}. Proposition \ref{prop_wav} proves that periodized
compactly supported wavelets on the unit interval form a localized basis for
the Lebesgue measure.

Considering the Haar basis, we can avoid the use of periodization and
consider more general measures than the Lebesgue one.

\begin{prpstn}
\label{prop_Haar}Let us take $\phi _{0}=\mathbf{1}_{\left[ 0,1\right] }$ and 
$\psi _{0}=\mathbf{1}_{\left[ 0,1/2\right] }-\mathbf{1}_{\left( 1/2,1\right]
}$ and consider the model $m$ given in (\ref{model_wav}). Set for every
integers $j\geq 0$, $1\leq k\leq 2^{j}$,%
\begin{equation*}
p_{j,k,-}=P^{X}\left( \left[ 2^{-j}\left( k-1\right) ,2^{-j}\left(
k-\frac{1}{2}\right) \right] \right) \text{ },\text{ }p_{j,k,+}=P^{X}\left( \left(
2^{-j}\left( k-\frac{1}{2}\right) ,2^{-j}k\right] \right) 
\end{equation*}%
\begin{equation*}
\psi _{j,k}:x\in \left[ 0,1\right] \mapsto \frac{1}{\sqrt{%
p_{j,k,+}^{2}p_{j,k,-}+p_{j,k,-}^{2}p_{j,k,+}}}\left( p_{j,k,+}\mathbf{1}_{%
\left[ 2^{-j}\left( k-1\right) ,2^{-j}\left( k-1/2\right) \right] }-p_{j,k,-}%
\mathbf{1}_{\left( 2^{-j}\left( k-\frac{1}{2}\right) ,2^{-j}k\right]
}\right) .  %\label{def_haar_basis}
\end{equation*}%
Moreover we set $\psi _{-1}=\phi _{0}$. Assume that $P^{X}$ has a density $f$
with respect to $%
%TCIMACRO{\TeXButton{Leb}{\leb}}%
%BeginExpansion
\leb%
%EndExpansion
$ on $\left[ 0,1\right] $ and that there exists $c_{\min }>0$ such that for
all $x\in \left[ 0,1\right] $,%
\begin{equation*}
f\left( x\right) \geq c_{\min }>0.
\end{equation*}%
Then $\left\{ \psi _{\lambda }\text{ };\text{ }\lambda \in \Lambda
_{b_m}\right\} $ is a strongly localized orthonormal basis of $\left(
m,\left\Vert 
%TCIMACRO{\TeXButton{.}{\cdot}}%
%BeginExpansion
\cdot%
%EndExpansion
\right\Vert _{2}\right) $. Indeed, by setting $A_{-1}=1$ and $A_{j}=2^{j}$, $%
j\geq 0$, we have for every integers $j\geq 0$, $1\leq k\leq 2^{j}$,%
\begin{equation*}
\left\Vert \psi _{j,k}\right\Vert _{\infty }\leq \sqrt{\frac{2}{c_{\min }}%
A_{j}}  %\label{upper_uni_Haar_basis_reg_sup}
\end{equation*}%
and%
\begin{equation*}
\sum_{j=-1}^{b_m}\sqrt{A_{j}}\leq \left( \sqrt{2}+1\right) \sqrt{D_{m}}\text{ .%
}  %\label{Haar_Ai}
\end{equation*}%
Finally, if $\Lambda _{j\left\vert \mu \right. }=\left\{ \lambda \in \Lambda
_{j};\supp\left( \varphi _{\mu }\right) \bigcap \supp%
\left( \varphi _{\lambda }\right) \neq \emptyset \right\} $ for $\mu \in
\Lambda _{b_m}$ and $j\in \left\{ -1,0,1,...,b_m\right\} $,%
\begin{equation*}
\max_{\mu \in \Lambda _{i}}%
%TCIMACRO{\TeXButton{Card}{\card}}%
%BeginExpansion
\card%
%EndExpansion
\left( \Lambda _{j\left\vert \mu \right. }\right) \leq A_{j}A_{i}^{-1}\vee 1%
.  %\label{haar_support}
\end{equation*}
\end{prpstn}

\noindent Proposition \ref{prop_Haar}, which proof is straightforward and
left to the reader, ensures that if $P^{X}$ has a density which is uniformly
bounded away from zero on $\mathcal{X}$, then the Haar basis is a strongly
localized orthonormal basis for the $L_{2}\left( P^{X}\right) $-norm. More
precisely, with notations of (\textbf{Aslb}), $r_{m}=\max \left\{ \sqrt{2}+1,%
\sqrt{2c_{\min }^{-1}}\right\} $ and $A_{c}=1$ are convenient.

\section{The slope heuristics\label{section_slope_heuristics}}

\subsection{Principles\label{ssection_principles}}

The slope heuristics is a conjunction of general facts about penalization
techniques in model selection, that lead in practice to an efficient penalty
calibration procedure. Let us briefly recall the main ideas underlying the
slope heuristics.

Consider the model selection problem described in (\ref{def_proc_2_reg}).
First, there exists a \textit{minimal penalty}, denoted $%
%TCIMACRO{\TeXButton{pen}{\pen}}%
%BeginExpansion
\pen%
%EndExpansion
_{\min }$, such that if $%
%TCIMACRO{\TeXButton{pen}{\pen}}%
%BeginExpansion
\pen%
%EndExpansion
\left( m_{1}\right) <%
%TCIMACRO{\TeXButton{pen}{\pen}}%
%BeginExpansion
\pen%
%EndExpansion
_{\min }\left( m_{1}\right) $ where $m_{1}$ is one of the largest models in $%
\mathcal{M}_{n}$, then the procedure defined in (\ref{def_proc_2_reg}) totally
misbehaves in the sense that the dimension of the selected model is one of
the largest of the collection, $D_{\widehat{m}}\gtrsim D_{m_{1}}$, and the
excess risk of the selected model explodes compared to the excess risk of
the oracle.

Furthermore, if $%
%TCIMACRO{\TeXButton{pen}{\pen}}%
%BeginExpansion
\pen%
%EndExpansion
>%
%TCIMACRO{\TeXButton{pen}{\pen}}%
%BeginExpansion
\pen%
%EndExpansion
_{\min }$ uniformly over the collection of models, then the selected model
is of reasonable dimension and achieves an oracle inequality as in (\ref%
{def_oracle_ineq}).

Arlot and Massart \cite{ArlotMassart:09} conjectured the validity in a large
M-estimation context of the following candidate for the minimal penalty,%
\begin{equation}
%TCIMACRO{\TeXButton{pen}{\pen}}%
%BeginExpansion
\pen%
%EndExpansion
_{\min }\left( m\right) =\mathbb{E}\left[ \ell _{%
%TCIMACRO{\TeXButton{emp}{\emp}}%
%BeginExpansion
\emp%
%EndExpansion
}\left( \widehat{s}_{m},s_{m}\right) \right] ,  \label{def_pen_min}
\end{equation}%
where $\ell _{%
%TCIMACRO{\TeXButton{emp}{\emp}}%
%BeginExpansion
\emp%
%EndExpansion
}\left( \widehat{s}_{m},s_{m}\right) $ is the empirical excess risk on the
model $m\in \mathcal{M}_{n}$, defined to be 
\begin{equation}
\ell _{%
%TCIMACRO{\TeXButton{emp}{\emp}}%
%BeginExpansion
\emp%
%EndExpansion
}\left( \widehat{s}_{m},s_{m}\right) =P_{n}\left( \gamma \left( s_{m}\right)
-\gamma \left( \widehat{s}_{m}\right) \right) \geq 0.
\label{def_emp_risk}
\end{equation}%
Finally, if the penalty satisfies $%
%TCIMACRO{\TeXButton{pen}{\pen}}%
%BeginExpansion
\pen%
%EndExpansion
=2\times 
%TCIMACRO{\TeXButton{pen}{\pen}}%
%BeginExpansion
\pen%
%EndExpansion
_{\min }$ then it is optimal in the sense that the excess risk of the
selected model converges to the excess risk of the oracle when the amount of
data tends to infinity,%
\begin{equation*}
\frac{\ell \left( s_{\ast },\widehat{s}_{\widehat{m}}\right) }{\inf_{m\in 
\mathcal{M}_{n}}\left\{ \ell \left( s_{\ast },\widehat{s}_{m}\right)
\right\} }\longrightarrow _{n\rightarrow +\infty }1.
\end{equation*}%
From the previous facts, two algorithms have been built in order to
optimally calibrate a penalty shape. Both are based on the estimation of the
minimal penalty. One takes advantage of the dimension jump of the selected
model occurring around the minimal penalty (see Figure~\ref{fig:dimjump}) and the other is based on formula
(\ref{def_pen_min}), performing a robust regression of the empirical risk
with respect to the penalty shape. We refer to the survey paper \cite%
{BauMauMich:12}\ for further details about the algorithmic and theoretical
works existing on the slope heuristics.

 \begin{figure}[t]
 \centering
 \includegraphics[width=0.4\textwidth]{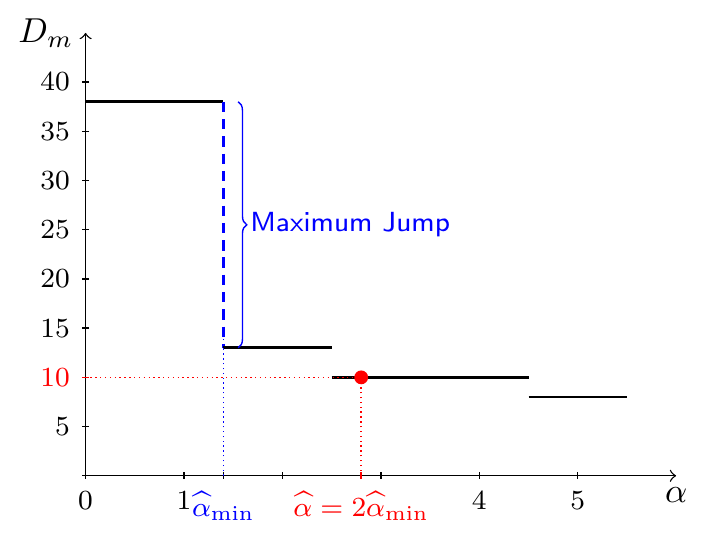}
 \caption{If $\mathrm{pen}_{\min} = \alpha_{\min}\cdot \mathrm{pen}_{\mathrm{shape}}$, for a known penalty shape $\mathrm{pen}_{\mathrm{shape}}$, then one can estimate $\alpha_{\min}$ by using the dimension jump. This gives $\widehat{\alpha}_{\min}$ and $\mathrm{pen}_{\mathrm{opt}}=2\widehat{\alpha}_{\min}\cdot \mathrm{pen}_{\mathrm{shape}}$ is then an optimal penalty according to the slope heuristics.}
 \label{fig:dimjump}
 \end{figure}

\subsection{Assumptions and comments\label{section_main_assumptions}}

\smallskip

\noindent \textbf{Set of assumptions : (SA)}

\begin{description}
\item[(\textbf{P1})] Polynomial complexity of $\mathcal{M}_{n}$: there exist
some constants $c_{\mathcal{M}},$ $\alpha _{\mathcal{M}}>0$ such that $%
%TCIMACRO{\TeXButton{Card}{\card}}%
%BeginExpansion
\card%
%EndExpansion
\left( \mathcal{M}_{n}\right) \leq c_{\mathcal{M}}n^{\alpha _{\mathcal{M}}}$ 
$.$

\item[(Auslb)] Existence of strongly localized bases: there exist $r_{%
\mathcal{M}},A_{c}>0$ such that for every $m\in \mathcal{M}_{n}$, there
exist $b_{m}\in \mathbb{N}_{\ast }$, a partition $\left( \Pi _{i}\right)
_{i=1}^{b_{m}}$ of $\left\{ 1,\ldots,D_{m}\right\} $, positive constants $%
\left( A_{i}\right) _{i=1}^{b_{m}}$ and an orthonormal basis $\left( \varphi
_{k}\right) _{k=1}^{D_{m}}$ of $\left( m,\left\Vert 
%TCIMACRO{\TeXButton{.}{\cdot}}%
%BeginExpansion
\cdot%
%EndExpansion
\right\Vert _{2}\right) $ such that $0<A_{1}\leq A_{2}\leq \ldots\leq
A_{b_{m}}<+\infty $,%
\begin{equation*}
\sum_{i=1}^{b_{m}}\sqrt{A_{i}}\leq r_{\mathcal{M}}\sqrt{D_{m}},
\end{equation*}%
and%
\begin{equation*}
\text{for all }i\in \left\{ 1,\ldots,b_{m}\right\} \text{, for all }k\in \Pi
_{i}\text{, }\left\Vert \varphi _{k}\right\Vert _{\infty }\leq r_{\mathcal{M}%
}\sqrt{A_{i}}.
\end{equation*}%
Moreover, for every $\left( i,j\right) \in \left\{ 1,\ldots,b_{m}\right\} ^{2}$
and $k\in \Pi _{i}$, we set 
\begin{equation*}
\Pi _{j\left\vert k\right. }=\left\{ l\in \Pi _{j}\text{ };\supp%
\left( \varphi _{k}\right) \bigcap \supp\left( \varphi _{l}\right)
\neq \emptyset \right\}
\end{equation*}%
and we assume that for all $j\in \left\{ 1,\ldots,b_{m}\right\} $,%
\begin{equation*}
\max_{k\in \Pi _{i}}%
%TCIMACRO{\TeXButton{Card}{\card}}%
%BeginExpansion
\card%
%EndExpansion
\left( \Pi _{j\left\vert k\right. }\right) \leq A_{c}\left(
A_{j}A_{i}^{-1}\vee 1\right) .
\end{equation*}

\item[(\textbf{P2})] Upper bound on dimensions of models in $\mathcal{M}_{n}$%
: there exists a positive constant $A_{\mathcal{M},+}$ such that for every $%
m\in \mathcal{M}_{n},$ $1\leq D_{m}\leq \max \left\{
D_{m},b_{m}^{2}A_{b_{m}}\right\} \leq A_{\mathcal{M},+}n\left( \ln n\right)
^{-2}$ $.$

\item[(\textbf{P3})] Richness of $\mathcal{M}_{n}$: there exist $%
m_{0},m_{1}\in \mathcal{M}_{n}$ and some constants $c_{\mathrm{\mathrm{rich}}},$ $A_{\mathrm{\mathrm{rich}}}>0$
such that $D_{m_{0}}\in \left[ \sqrt{n},c_{\mathrm{rich}}\sqrt{n}\right] $ and $%
D_{m_{1}}\geq A_{\mathrm{rich}}n\left( \ln n\right) ^{-2}$ $.$

\item[(\textbf{Ab})] A positive constant $A$ exists, that bounds the data
and the projections $s_{m}$ of the target $s_{\ast }$ over the models $m$ of
the collection $\mathcal{M}_{n}$: $\left\vert Y_{i}\right\vert \leq A<\infty
,$ $\left\Vert s_{m}\right\Vert _{\infty }\leq A<\infty $ for all $m\in 
\mathcal{M}_{n}.$

\item[(\textbf{An})] Uniform lower-bound on the noise level: $\sigma \left(
X_{i}\right) \geq \sigma _{\min }>0$ $a.s.$

\item[(\textbf{Ap}$_{u}$)] The bias decreases as a power of $D_{m}$: there
exist $\beta _{+}>0$ and $C_{+}>0$ such that%
\begin{equation*}
\ell \left( s_{\ast },s_{m}\right) \leq C_{+}D_{m}^{-\beta _{+}}\text{ }.
\end{equation*}
\end{description}

\smallskip

The set of Assumptions (\textbf{SA}) is very similar---and actually extends---the set of assumptions used in \cite{ArlotMassart:09} and \cite{saum:13}
to prove the validity of the slope heuristics in heteroscedastic
least-squares regression, respectively for models of histograms and
piecewise polynomials.

The main features in this set of Assumptions (\textbf{SA}) are as follows.
Assumption (\textbf{P1}) amounts to say that we select a model among a
"small" collection, as opposed to large collection of models whose cardinal
is exponential with respect to the amount of data $n$. Roughly speaking,
this assumption allows to neglect the deviations of the excess risks on each
model around their mean, since concentration inequalities shown in Section %
\ref{section_excess_risk_concentration} below for these quantities are
exponential.

Then Assumptions (\textbf{Auslb}), (\textbf{P2}), (\textbf{Ab}) and (\textbf{%
An}) enable to apply the desired concentration inequalities for the excess
risks established in Section \ref{section_excess_risk_concentration}. As
shown in Section \ref{section_Haar}, Assumption (\textbf{Auslb}) allows in
particular to encompass the case of compactly supported wavelet expansions
on the interval.

For further and more detailed comments on the above assumptions, we refer to 
\cite{ArlotMassart:09} and \cite{saum:13}.

\subsection{Statement of the theorems}%\label{section_main_theorems_reg}

Let us now state our results validating the slope heuristics for the
selection of uniformly strongly localized bases. The first theorem exhibits the
empirical excess risk defined in (\ref{def_emp_risk}) as a (majorant of the)
minimal penalty, as conjectured by Arlot and Massart \cite{ArlotMassart:09}.

\begin{thrm}
\label{theorem_min_pen_reg_pp}Take a positive penalty: for all $m\in 
\mathcal{M}_{n}$, $%
%TCIMACRO{\TeXButton{pen}{\pen}}%
%BeginExpansion
\pen%
%EndExpansion
\left( m\right) \geq 0$. Suppose that the assumptions (\textbf{SA}) of
Section \ref{section_main_assumptions} hold, and furthermore suppose that
for $A_{%
%TCIMACRO{\TeXButton{pen}{\pen}}%
%BeginExpansion
\pen%
%EndExpansion
}\in \left[ 0,1\right) $ and $A_{p}>0$ the model $m_{1}$ of assumption (%
\textbf{P3}) satisfies 
\begin{equation*}
\text{ }0\leq 
%TCIMACRO{\TeXButton{pen}{\pen}}%
%BeginExpansion
\pen%
%EndExpansion
\left( m_{1}\right) \leq A_{%
%TCIMACRO{\TeXButton{pen}{\pen}}%
%BeginExpansion
\pen%
%EndExpansion
}\mathbb{E}\left[ \ell _{%
%TCIMACRO{\TeXButton{emp}{\emp}}%
%BeginExpansion
\emp%
%EndExpansion
}\left( \widehat{s}_{m_{1}},s_{m_{1}}\right) \right] \text{ },
%\label{majo_pen_pp}
\end{equation*}%
with probability at least $1-A_{p}n^{-2}$. Then there exist a constant $%
L_{1}>0$ only depending on constants in \textit{(\textbf{SA})}, as well as
an integer $n_{0}$ and a positive constant $L_{2}$ only depending on $A_{%
%TCIMACRO{\TeXButton{pen}{\pen}}%
%BeginExpansion
\pen%
%EndExpansion
}$ and on constants in \textit{(\textbf{SA})} such that, for all $n\geq
n_{0} $, it holds with probability at least $1-L_{1}n^{-2}$,%
\begin{equation*}
D_{\widehat{m}}\geq L_{2}n\ln \left( n\right) ^{-2}
\end{equation*}%
and%
\begin{equation*}
\ell \left( s_{\ast },\widehat{s}_{\widehat{m}}\right) \geq \frac{n^{\beta
_{+}/\left( 1+\beta _{+}\right) }}{\left( \ln n\right) ^{3}}\inf_{m\in 
\mathcal{M}_{n}}\left\{ \ell \left( s_{\ast },\widehat{s}_{m}\right)
\right\} \text{ },  %\label{bad_oracle_min_pen_pp}
\end{equation*}%
where $\beta _{+}>0$ is defined in assumption (\textbf{Ap}$_{u}$) of (%
\textbf{SA}).
\end{thrm}

\noindent In order to theoretically validate the slope heuristics described
in Section \ref{ssection_principles}\ above, it remains, in addition to
Theorem \ref{theorem_min_pen_reg_pp}, to show that taking a penalty greater
than the empirical excess risk ensures an oracle inequality and that taking
two times the empirical excess risk yields asymptotic optimality of the
procedure. That's what we present now.

\begin{thrm}
\label{theorem_opt_pen_reg_pp}Suppose that the assumptions \textit{(\textbf{%
SA})} of Section \ref{section_main_assumptions} hold, and furthermore
suppose that for some $\delta \in \left[ 0,1\right) $ and $A_{p},A_{r}>0$,
there exists an event of probability at least $1-A_{p}n^{-2}$ on which, for
every model $m\in \mathcal{M}_{n}$ such that $D_{m}\geq A_{\mathcal{M}%
,+}\left( \ln n\right) ^{3}$, it holds%
\begin{equation*}
\left\vert 
%TCIMACRO{\TeXButton{pen}{\pen}}%
%BeginExpansion
\pen%
%EndExpansion
\left( m\right) -2\mathbb{E}\left[ \ell _{%
%TCIMACRO{\TeXButton{emp}{\emp}}%
%BeginExpansion
\emp%
%EndExpansion
}\left( \widehat{s}_{m},s_{m}\right) \right] \right\vert \leq \delta \left(
\ell \left( s_{\ast },s_{m}\right) +\mathbb{E}\left[ \ell _{%
%TCIMACRO{\TeXButton{emp}{\emp}}%
%BeginExpansion
\emp%
%EndExpansion
}\left( \widehat{s}_{m},s_{m}\right) \right] \right)  %\label{pen_id_pp}
\end{equation*}%
together with%
\begin{equation}
\left\vert 
%TCIMACRO{\TeXButton{pen}{\pen}}%
%BeginExpansion
\pen%
%EndExpansion
\left( m\right) \right\vert \leq A_{r}\left( \frac{\ell \left( s_{\ast
},s_{m}\right) }{\left( \ln n\right) ^{2}}+\frac{\left( \ln n\right) ^{3}}{n}%
\right) .  \label{pen_id_2_pp}
\end{equation}%
Then, for any $\eta \in \left( 0,\beta _{+}/\left( 1+\beta _{+}\right)
\right) $, there exist an integer $n_{0}$ only depending on $\eta ,\delta $
and $\beta _{+}$ and on constants in \textit{(\textbf{SA})}, a positive
constant $L_{3}$ only depending on $c_{\mathcal{M}}$ given in \textit{(%
\textbf{SA})} and on $A_{p}$, two positive constants $L_{4}$ and $L_{5}$
only depending on constants in \textit{(\textbf{SA})} and on $A_{r}$ and a
sequence 
\begin{equation*}
\theta _{n}\leq \frac{L_{4}}{\left( \ln n\right) ^{1/4}}
%\label{def_theta_n_pp}
\end{equation*}%
such that it holds for all $n\geq n_{0}$, with probability at least $%
1-L_{3}n^{-2}$, 
\begin{equation*}
D_{%
%TCIMACRO{\TeXButton{Mhat}{\widehat{m}}}%
%BeginExpansion
\widehat{m}%
%EndExpansion
}\leq n^{\eta +1/\left( 1+\beta _{+}\right) }
\end{equation*}%
and%
\begin{equation*}
\ell \left( s_{\ast },\widehat{s}_{\widehat{m}}\right) \leq \left( \frac{%
1+\delta }{1-\delta }+\frac{5\theta _{n}}{\left( 1-\delta \right) ^{2}}%
\right) \inf_{m\in \mathcal{M}_{n}}\left\{ \ell \left( s_{\ast },\widehat{s}%
_{m}\right) \right\} +L_{5}\frac{\left( \ln n\right) ^{3}}{n}.
%\label{oracle_opt_gene_pp}
\end{equation*}%
Assume that in addition, the following assumption holds,

\begin{description}
\item[(\textbf{Ap})] The bias decreases like a power of $D_{m}$: there exist 
$\beta _{-}\geq \beta _{+}>0$ and $C_{+},C_{-}>0$ such that%
\begin{equation*}
C_{-}D_{m}^{-\beta _{-}}\leq \ell \left( s_{\ast },s_{m}\right) \leq
C_{+}D_{m}^{-\beta _{+}}\text{ }.
\end{equation*}
\end{description}

\noindent Then it holds for all $n\geq n_{0}\left( \left( \text{\textbf{SA}}%
\right) ,C_{-},\beta _{-},\beta _{+},\eta ,\delta \right) $, with
probability at least $1-L_{3}n^{-2}$,%
\begin{equation*}
A_{\mathcal{M},+}\left( \ln n\right) ^{3}\leq D_{%
%TCIMACRO{\TeXButton{Mhat}{\widehat{m}}}%
%BeginExpansion
\widehat{m}%
%EndExpansion
}\leq n^{\eta +1/\left( 1+\beta _{+}\right) } % \label{dim_ap_pp}
\end{equation*}%
and%
\begin{equation}
\ell \left( s_{\ast },\widehat{s}_{\widehat{m}}\right) \leq \left( \frac{%
1+\delta }{1-\delta }+\frac{5\theta _{n}}{\left( 1-\delta \right) ^{2}}%
\right) \inf_{m\in \mathcal{M}_{n}}\left\{ \ell \left( s_{\ast },\widehat{s}%
_{m}\right) \right\} .  \label{oracle_opt_pp}
\end{equation}
\end{thrm}

Notice that taking $\delta =0$ in Theorem \ref{theorem_opt_pen_reg_pp} gives
an oracle inequality (\ref{oracle_opt_pp}) with leading constant equal to $%
1+5\theta _{n}$ and thus converging to one when the amount of data tends to
infinity. This shows the optimality of the penalty equal to two times the
minimal one, thus validating the slope heuristics for the selection of
models endowed with a strongly localized basis structure.\\
The proofs of Theorems \ref{theorem_min_pen_reg_pp} and \ref{theorem_opt_pen_reg_pp} simply derive from \cite{saum:13} and Theorem \ref{theorem_general_sup} above (see Section \ref{ssection_proof_slope} for more details). 

\section{V-fold model selection\label{section_Vfold_model_selection}}

We need some further notations and we follow here the notations of Arlot 
\cite{Arl:2008a}. In order to highlight the dependence in the training set,
we will denote $\widehat{s}_{m}\left( P_{n}\right) $ for the least-squares
estimator learned from the empirical distribution $P_{n}=1/n\sum_{i=1}^{n}%
\delta _{\left( X_{i},Y_{i}\right) }$. In $V$-fold sampling, we choose some
partition $\left( B_{j}\right) _{1\leq j\leq V}$ of the index set $\left\{
1,\ldots,n\right\} $ and define%
\begin{equation*}
P_{n}^{\left( j\right) }=\frac{1}{%
%TCIMACRO{\TeXButton{Card}{\card}}%
%BeginExpansion
\card%
%EndExpansion
\left( B_{j}\right) }\sum_{i\in B_{j}}\delta _{\left( X_{i},Y_{i}\right) }%
\text{ \ \ \ and \ \ \ }P_{n}^{\left( -j\right) }=\frac{1}{n-%
%TCIMACRO{\TeXButton{Card}{\card}}%
%BeginExpansion
\card%
%EndExpansion
\left( B_{j}\right) }\sum_{i\notin B_{j}}\delta _{\left( X_{i},Y_{i}\right) }
\end{equation*}%
together with the estimators,%
\begin{equation*}
\widehat{s}_{m}^{\left( -j\right) }=\widehat{s}_{m}\left( P_{n}^{\left(
-j\right) }\right) .
\end{equation*}

\subsection{Classical V-fold cross-validation}

In the VFCV procedure, the selected model $%
\widehat{m}_{\mathrm{VFCV}}$ optimizes the classical $V$-fold criterion,%
\begin{equation}
\widehat{m}_{\mathrm{VFCV}}\in \arg \min_{m\in \mathcal{M}_{n}}\left\{\mathrm{%
crit}_{\mathrm{VFCV}}\left(m\right) \right\},  \label{def_VFCV}
\end{equation}

where%
\begin{equation}
\mathrm{crit}_{\mathrm{VFCV}}\left( m\right) =\frac{1}{V}\sum_{j=1}^{V}P_{n}^{%
\left( j\right) }\gamma \left( \widehat{s}_{m}^{\left( -j\right) }\right) 
.  \label{def_crit_VFCV}
\end{equation}%
We assume that the partition is regular in the sense that for all $j\in
\left\{ 1,\ldots,V\right\} $, $%
%TCIMACRO{\TeXButton{Card}{\card}}%
%BeginExpansion
\card%
%EndExpansion
\left( B_{j}\right) =n/V$ and in practice we can always ensure that for all $%
j$, $\left\vert 
%TCIMACRO{\TeXButton{Card}{\card}}%
%BeginExpansion
\card%
%EndExpansion
\left( B_{j}\right) -n/V\right\vert <1$.

\begin{thrm}
\label{Theorem_VFCV}Assume that (\textbf{SA}) holds. Let $r\in \left(
2,+\infty \right) $ and $V\in \left\{ 2,\ldots,n-1\right\} $ satisfying $%
1<V\leq r$. Define the VFCV procedure as the model
selection procedure given by (\ref{def_VFCV}). Then, there exists a constant $L_{\left( \text{\textbf{SA}}\right) ,r}>0$ such that for all $n\geq
n_{0}\left( \left( \text{\textbf{SA}}\right) ,r\right) $, with probability
at least $1-L_{\left( \text{\textbf{SA}}\right) ,r}n^{-2}$,%
\begin{equation*}
\ell \left( s_{\ast },\widehat{s}_{\widehat{m}_{\mathrm{VFCV}}}\right) \leq
\left( 1+\frac{L_{\left( \text{\textbf{SA}}\right) ,r}}{\sqrt{\ln n}}\right)
\inf_{m\in \mathcal{M}_{n}}\left\{ \ell \left( s_{\ast },\widehat{s}%
_{m}^{\left( -1\right) }\right) \right\} +L_{\left( \text{\textbf{SA}}%
\right) ,r}\frac{\left( \ln n\right) ^{3}}{n}.  %\label{oracle_VFCV}
\end{equation*}
\end{thrm}

The proof of Theorem \ref{Theorem_VFCV} can be found in Section \ref{ssection_proof_VF}.\\
 In Theorem \ref{Theorem_VFCV}, we show an oracle inequality with leading
constant converging to one when the amount of data tends to infinity, that
compares the excess risk of the model selected \textit{via} VFCV to the excess risk of the best estimator learned with a
fraction of the amount of data equal to $1-V^{-1}$. Thus, VFCV allows to asymptotically recover the oracle learned with a
fraction of the amount of data equal to $1-V^{-1}$. This is natural since
the $V$-fold criterion given in (\ref{def_crit_VFCV}) is an unbiased
estimate of the risk of estimators learned with a fraction $1-V^{-1}$ of the
data.

Consequently, it seems from Theorem \ref{Theorem_VFCV} that there is some
room to improve the performances of VFCV for fixed $V$,
since the oracle learned with all the data has better performances (smaller
excess risk) than the oracle learned with only part of the initial data.
Furthermore, using the concentration inequalities derived in Theorem \ref%
{theorem_excess_risk_strong_loc}, we roughly have, for any $m\in \mathcal{M}%
_{n}$,%
\begin{eqnarray*}
\mathbb{E}\left[ \ell \left( s_{\ast },\widehat{s}_{m}^{\left( -1\right)
}\right) \right] &=&\ell \left( s_{\ast },s_{m}\right) +\mathbb{E}\left[
\ell \left( s_{m},\widehat{s}_{m}^{\left( -1\right) }\right) \right] \\
&\sim &\ell \left( s_{\ast },s_{m}\right) +\frac{1}{4}\frac{\mathcal{C}_{m}}{%
(1-V^{-1})n} \\
&\sim &\ell \left( s_{\ast },s_{m}\right) +\frac{V}{V-1}\mathbb{E}\left[
\ell \left( s_{m},\widehat{s}_{m}\right) \right] \\
&\leq &\frac{V}{V-1}\mathbb{E}\left[ \ell \left( s_{\ast },\widehat{s}%
_{m}\right) \right] .
\end{eqnarray*}%
The natural idea to overcome this issue is to try to select a model using an
unbiased estimate of the risk of the estimators $\widehat{s}_{m}$ (rather
than $\widehat{s}_{m}^{\left( -1\right) }$ for VFCV).
This is what we propose in the following section.

\subsection{V-fold penalization}

Let us consider the following penalization procedure, proposed by Arlot \cite%
{Arl:2008a} and called $V$-fold penalization,

\begin{equation*}
\widehat{m}_{\mathrm{penVF}}\in \arg \min_{m\in \mathcal{M}_{n}}\left\{ \text{%
crit}_{\mathrm{penVF}}\left( m\right) \right\},
\end{equation*}

where%
\begin{equation}\label{def_crit_penVF}
\mathrm{crit}_{\mathrm{penVF}}\left( m\right) =P_{n}\left( \gamma \left( 
\widehat{s}_{m}\right) \right) +%
%TCIMACRO{\TeXButton{pen}{\pen}}%
%BeginExpansion
\pen%
%EndExpansion
_{\mathrm{VF}}\left( m\right) ,
\end{equation}%
with 
\begin{equation}
%TCIMACRO{\TeXButton{pen}{\pen}}%
%BeginExpansion
\pen%
%EndExpansion
_{\mathrm{VF}}\left( m\right) =\frac{V-1}{V}\sum_{j=1}^{V}\left[ P_{n}\gamma
\left( \widehat{s}_{m}^{\left( -j\right) }\right) -P_{n}^{\left( -j\right)
}\gamma \left( \widehat{s}_{m}^{\left( -j\right) }\right) \right] .
\label{def_VF_pen}
\end{equation}%
The idea behind $V$-fold penalization is to use the $V$-fold penalty $%
%TCIMACRO{\TeXButton{pen}{\pen}}%
%BeginExpansion
\pen%
%EndExpansion
_{\mathrm{VF}}$ as an unbiased estimate of the \textit{ideal penalty }$%
%TCIMACRO{\TeXButton{pen}{\pen}}%
%BeginExpansion
\pen%
%EndExpansion
_{\mathrm{id}}$, the latter allowing to recover exactly the oracle $m_{\ast }$ defined in \ref{oracle_model}. Indeed, we can write%
\begin{eqnarray*}
m_{\ast } &\in \arg \min_{m\in \mathcal{M}_{n}}\left\{ P_{n}\left( \gamma \left( 
\widehat{s}_{m}\right) \right) +%
%TCIMACRO{\TeXButton{pen}{\pen}}%
%BeginExpansion
\pen%
%EndExpansion
_{\mathrm{id}}\left( m\right) \right\} ,
\end{eqnarray*}%
where%
\begin{equation}
%TCIMACRO{\TeXButton{pen}{\pen}}%
%BeginExpansion
\pen%
%EndExpansion
_{\mathrm{id}}\left( m\right) =P\left( \gamma \left( \widehat{s}_{m}\right)
\right) -P_{n}\left( \gamma \left( \widehat{s}_{m}\right) \right) .
\label{def_pen_id}
\end{equation}%
Comparing (\ref{def_VF_pen}) and (\ref{def_pen_id}), it is now clear that the $V$%
-fold penalty is a \textit{resampling estimate} of the ideal penalty where
for each $j\in \left\{ 1,\ldots,V\right\} $ the role of $P$ is played by $P_{n}$
and the role of $P_{n}$ is played by $P_{n}^{\left( -j\right) }$.

Now, the benefit compared to VFCV is that $V$-fold
penalization is asymptotically optimal, as stated in the following theorem.

\begin{thrm}
\label{Theorem_VFpen}Assume that (\textbf{SA}) holds. Let $r\in \left(
2,+\infty \right) $ and $V\in \left\{ 2,\ldots,n-1\right\} $ satisfying $%
1<V\leq r$. Define the $V$-fold penalization procedure as the model
selection procedure given in (\ref{def_crit_penVF}). Then, there exists a constant $L_{\left( \text{\textbf{SA}}\right) ,r}>0$ such that for all $n\geq n_{0}\left( \left( \text{\textbf{SA}}\right) ,r\right) $%
, with probability at least $1-L_{\left( \text{\textbf{SA}}\right) ,r}n^{-2}$%
,%
\begin{equation*}
\ell \left( s_{\ast },\widehat{s}_{\widehat{m}_{\mathrm{penVF}}}\right) \leq
\left( 1+\frac{L_{\left( \text{\textbf{SA}}\right) ,r}}{\sqrt{\ln n}}\right)
\inf_{m\in \mathcal{M}_{n}}\left\{ \ell \left( s_{\ast },\widehat{s}%
_{m}\right) \right\} +L_{\left( \text{\textbf{SA}}\right) ,r}\frac{\left(
\ln n\right) ^{3}}{n}.
\end{equation*}
\end{thrm}

The proof of Theorem \ref{Theorem_VFpen} can be found in Section \ref{sssection_proof_VFpen}.\\
Theorem \ref{Theorem_VFpen} exhibits an oracle inequality with leading
constant converging to one, comparing the risk of the model selected by $V$%
-fold penalization to the risk of an oracle model. This shows asymptotic
optimality of the procedure and extends to the case of the selection of
linear models endowed with a strongly localized basis structure, previous
optimality results obtained by Arlot \cite{Arl:2008a} for the selection of
histograms, also in heteroscedatic regression with random design.

\section{Excess risks' concentration\label{section_excess_risk_concentration}%
}

We formulate in this section optimal upper and lower bounds that describe
the concentration of the excess risks for a fixed parametric model, but with
dimension depending on the sample size. In the case of the existence of a
strongly localized basis, we prove optimal bounds for models with dimension
roughly smaller than $n$ (up to logarithmic factors).

The proofs, which involve sophisticated arguments from empirical process
theory, are partly based on earlier work by Saumard \cite{saum:12}.
Furthermore, we use some representation formulas for functionals of
M-estimators, which generalize previous excess risks representations exposed
by Saumard \cite{saum:12}, Chatterjee \cite{chatterjee2014}, Muro and van de
Geer \cite{MurovandeGeer:15} and van de Geer and Wainwright \cite%
{vandeGeerWain:16}. We give these formulas in Section \ref%
{ssection_rep_formulas}.

\subsection{Strongly localized bases case} \label{ssection_strong_loc_bas}

The following result of consistency in sup-norm for the least-squares
estimator is a preliminary result that will be needed in the proof of our
optimal concentration bounds.

\begin{thrm}
\label{theorem_general_sup}Let $\alpha >0$. Assume that $m$ is a linear
vector space of finite dimension $D_m$ satisfying (\textbf{Aslb}) and use
notations of (\textbf{Aslb}). Assume moreover that the following assumption
holds:

\begin{description}
\item[(Ab($m$))] A positive constant $A$ exists, that bounds the data and
the projection $s_{m}$ of the target $s_{\ast }$ on the model $m$: $%
\left\vert Y_{i}\right\vert \leq A<\infty ,$ $\left\Vert s_{m}\right\Vert
_{\infty }\leq A<\infty.$
\end{description}

\noindent If there exists $A_{+}>0$ such that 
\begin{equation*}
\max \left\{ D_m,b_m^{2}A_{b_m}\right\} \leq A_{+}\frac{n}{\left( \ln n\right) ^{2}%
},  %\label{assump_dim_p2Ap}
\end{equation*}%
then there exists a positive constant $L_{A,r_{m},\alpha }$ such that, for
all $n\geq n_{0}\left( A_{+},A_{c},r_{m},\alpha \right) $,%
\begin{equation}
\mathbb{P}\left( \left\Vert \widehat{s}_{m}-s_{m}\right\Vert _{\infty }\geq
L_{A,r_{m},\alpha }\sqrt{\frac{D_m\ln n}{n}}\right) \leq n^{-\alpha }.
\label{conv_sup_norm_reg_sup}
\end{equation}
\end{thrm}

Theorem \ref{theorem_general_sup} extends to the case of strongly localized
bases previous results obtained in \cite{saum:12} for the consistency in
sup-norm of least-squares estimators on linear models of histograms and
piecewise polynomials. Note that minimax rates of convergence in sup-norm---and more general $L_{q}$ norms, $1\leq q\leq \infty $---for random design
regression have been obtained by Stone \cite{Stone:82}.

Theorem \ref{theorem_general_sup} is based on new formulas for functionals
of M-estimators that are described in Section \ref{ssection_rep_formulas}\
below.
\begin{rmk} \label{remark_core_text}
The main results of our paper are proved for models endowed with a strongly
localized basis. In fact, we can also prove some results for the slightly
weaker and more classical assumption of localized basis, defined in (\ref{def_localized_basis}).
The main difference is that with models having a strongly localized basis we
can describe the optimality of model selection procedures for the selection
of models with dimension up to $n/\left( \ln n\right) ^{2}$, whereas for the
localized basis case, we describe optimal results for models with dimension
smaller than $n^{1/3}/\left( \ln n\right) ^{2}$. This is an issue for
instance in the slope heuristics, where the two algorithms of detection of
the minimal penalty are based on the behavior of the largest models in the collection at hand. At a technical level, the essential gap is that for models with localized bases, we are able to prove Inequality \eqref{conv_sup_norm_reg_sup} in Theorem \ref{theorem_general_sup}   
for models with linear dimension $D_{m}\ll n^{1/3}$ (see Remark \ref{remark_proof}).
\end{rmk}
Let us now detail our concentration bounds for the excess risks. Theorem \ref%
{theorem_excess_risk_strong_loc}\ below is a corollary of Theorem 2\ of \cite%
{saum:12}\ and Theorem \ref{theorem_general_sup}\ above. 

\begin{thrm}
\label{theorem_excess_risk_strong_loc}Let $A_{+},A_{-},\alpha >0$. Assume
that $m$ is a linear vector space of finite dimension $D_{m}$ satisfying (%
\textbf{Aslb}) and use notations of (\textbf{Aslb}). Assume moreover that Assumption (\textbf{Ab}($m$)) defined in Theorem \ref{theorem_general_sup} holds.
If we have 
\begin{equation*}
A_{-}\left( \ln n\right) ^{2}\leq D_{m}\leq \max \left\{
D_{m},b_m^{2}A_{b_m}\right\} \leq A_{+}\frac{n}{\left( \ln n\right) ^{2}}\text{ }%
,
\end{equation*}%
then a positive constant $L_{0}$ exists, only depending on $\alpha ,A_{-}$
and on the constants $A,\sigma _{\min }$ and $r_{m}$ such that by setting%
\begin{equation}
\varepsilon _{n}=L_{0}\max \left\{ \left( \frac{\ln n}{D_{m}}\right) ^{1/4},%
\text{ }\left( \frac{D_{m}\ln n}{n}\right) ^{1/4}\right\} \text{ },
\label{def_epsilon}
\end{equation}%
we have for all $n\geq n_{0}\left( A_{-},A_{+},A,r_{m},\sigma _{\min
},\alpha \right) $,%
\begin{align}
\mathbb{P}\left[ \left( 1-\varepsilon _{n}\right) \frac{\mathcal{C}_{m}}{n}%
\leq \ell \left( s_{m},\widehat{s}_{m}\right) \leq \left( 1+\varepsilon
_{n}\right) \frac{\mathcal{C}_{m}}{n}\right] & \geq 1-10n^{-\alpha }\text{ },
\label{lower_true_strong} \\
\mathbb{P}\left[ \left( 1-\varepsilon _{n}^{2}\right) \frac{\mathcal{C}_{m}}{%
n}\leq \ell _{%
%TCIMACRO{\TeXButton{emp}{\emp}}%
%BeginExpansion
\emp%
%EndExpansion
}\left( \widehat{s}_{m},s_{m}\right) \leq \left( 1+\varepsilon
_{n}^{2}\right) \frac{\mathcal{C}_{m}}{n}\right] & \geq 1-5n^{-\alpha }\text{
},  \label{upper_emp_strong}
\end{align}%
where $\mathcal{C}_{m}=\sum_{k=1}^{D_{m}}%
%TCIMACRO{\TeXButton{Var}{\var}}%
%BeginExpansion
\var%
%EndExpansion
\left( \left( Y-s_{m}\left( X\right) \right) \cdot \varphi _{k}\left(
X\right) \right) $.
\end{thrm}

Theorem \ref{theorem_excess_risk_strong_loc} exhibits the concentration of
the excess risk and the empirical excess risk around the same value equal to 
$n^{-1}\mathcal{C}_{m}$. Furthermore, it is easy to check that the term $%
\mathcal{C}_{m}$ is of the order of the linear dimension $D_{m}$. More
precisely, it satisfies%
\begin{equation*}
0<\frac{\sigma _{\min }D_{m}}{2}\leq \mathcal{C}_{m}\leq \frac{3AD_{m}}{2}%
.
\end{equation*}%
See \cite{saum:12}, Section 4.3 for the details, noticing that with the
notations of \cite{saum:12}, it holds $\mathcal{C}_{m}=D_{m}\mathcal{K}%
_{1,m}^{2}/4$. It is also worth noticing that the empirical excess risk
concentrates better than the true excess risk, the rate of concentration for
the empirical excess risk---given by the term $\varepsilon _{n}^{2}$---being the square of the concentration rate $\varepsilon _{n}\ $of the excess risk.
This will be explained at a heuristic level in Section \ref%
{ssection_rep_formulas} using representation formulas for the excess risks
in terms of empirical process.

Compared to other concentration results established in \cite{chatterjee2014}%
, \cite{MurovandeGeer:15} and \cite{vandeGeerWain:16} for the excess risk of
least-squares or more general M-estimators, Inequalities (\ref%
{lower_true_strong}) and (\ref{upper_emp_strong}) share the strong feature of
computing the exact concentration point, which is equal to $n^{-1}\mathcal{C}%
_{m}$. On contrary, the methodology built by Chatterjee \cite{chatterjee2014}
and extended in \cite{MurovandeGeer:15} and \cite{vandeGeerWain:16}, gives
the concentration of the excess risk around a point, but says nothing on the
value of this point. We explain further this important aspect in Section \ref%
{ssection_rep_formulas} below.

\subsection{Representation formulas for functionals of M-estimators\label%
{ssection_rep_formulas}}

In this section only, we assume that the contrast $\gamma $ defining the
estimator $\widehat{s}_{m}$ is general, so that $\widehat{s}_{m}$ is a
general M-estimator---assumed to exist---on a model $m$,%
\begin{eqnarray*}
\widehat{s}_{m} &\in &\arg \min_{s\in m}\left\{ P_{n}\left( \gamma \left(
s\right) \right) \right\} \\
&=&\arg \min_{s\in m}\left\{ \frac{1}{n}\sum_{i=1}^{n}\gamma \left( s\right)
\left( Z_{i}\right) \right\} ,
\end{eqnarray*}%
where $\left( Z_{1},\ldots,Z_{n}\right) \in \mathcal{Z}^{n}$ is a sample of
random variables living in some general measurable space $\mathcal{Z}$.

Define $\mathcal{F}$ a nonnegative functional from $m$ to $\mathbb{R}_{+}$: $%
\forall s\in m$, $\mathcal{F}\left( s\right) \geq 0$. Then the following 
\textit{representation} of $\mathcal{F}\left( \widehat{s}_{m}\right) $ in
terms of local extrema of the empirical process of interest holds.

\begin{prpstn}
\label{prop_func_rep}With the notations above, let us also write $m_{C}$
(resp. $d_{C}$), $C\geq 0$, the subset of the model $m$ such that the values
of the functional $\mathcal{F}$ on this subset are bounded above by (resp.
equal to) $C$: 
\begin{equation*}
m_{C}=\left\{ s\in m\text{ };\text{ }\mathcal{F}\left( s\right) \leq
C\right\} \text{ and }d_{C}=\left\{ s\in m\text{ };\text{ }\mathcal{F}\left(
s\right) =C\right\} .
\end{equation*}%
Then,%
\begin{equation}
\mathcal{F}\left( \widehat{s}_{m}\right) \in \arg \min_{C\geq 0}\left\{
\inf_{s\in d_{C}}P_{n}\left( \gamma \left( s\right) \right) \right\}
\label{formula_funct_rep}
\end{equation}%
and%
\begin{equation}
\mathcal{F}\left( \widehat{s}_{m}\right) \in \arg \min_{C\geq 0}\left\{
\inf_{s\in m_{C}}P_{n}\left( \gamma \left( s\right) \right) \right\} \text{ .%
}  \label{formula_funct_rep_ball}
\end{equation}
\end{prpstn}

Proposition \ref{prop_func_rep}, whose proof is simple and written in
Section \ref{ssection_proofs_excess_risk_rep}, casts the problem of bounding 
\textit{any} functional of a M-estimator into an empirical process question,
consisting of comparing local extrema of the empirical measure $P_{n}$ taken
on the contrasted functions of the model. Up to our knowledge, such a result
is new.

Considering the particular case of the sup-norm, formula (\ref%
{formula_funct_rep}) is our starting point to prove Theorem \ref%
{theorem_general_sup}. More precisely, we use the fact that taking 
\begin{equation*}
\mathcal{F}\left( \widehat{s}_{m}\right) =\left\Vert \widehat{s}_{m}-s_{m}\right\Vert _{\infty }\text{ , 
}
\end{equation*}%
formula (\ref{formula_funct_rep}) directly implies that for any $C\geq 0$,%
\begin{eqnarray*}
&&\mathbb{P}\left( \left\Vert \widehat{s}_{m}-s_{m}\right\Vert _{\infty }\geq C\right)  \\
&\leq &\mathbb{P}\left( \inf_{s\in m\backslash m_{C}}P_{n}\left( \gamma
\left( s\right) \right) \leq \inf_{s\in m_{C}}P_{n}\left( \gamma \left(
s\right) \right) \right) .
\end{eqnarray*}
See Section \ref{section_proof_slope_reg} for the complete proofs.

Another interesting application of Proposition \ref{prop_func_rep} would be
to derive bounds for the $L_{p},$ $p\geq 1,$ moments---or more general
Orlicz norms---of a M-estimator. We postpone this question for future work.

\begin{rmrk}
Nonnegativity of $\mathcal{F}$ is not essential (but suitable to our needs)
and considering functionals with negative values is also possible, with
straightforward adaptations of formulas of Proposition \ref{prop_func_rep}.
\end{rmrk}

Taking $\mathcal{F}$ to be the true or the empirical excess risk on $m$, we
get the following results, refining the representation formulas previously
obtained by \cite{saum:12}---see Remark 1 of Section 3 therein.

\begin{prpstn}
\label{prop_excess_rep_loc}With the notations above, let also $\mathcal{G}$
be a nonnegative functional on $m$ and $R_{0}\in \mathbb{R}_{+}\cup \left\{
+\infty \right\} $. If the following event holds $\left\{ \mathcal{G}\left( 
\widehat{s}_{m}\right) \leq R_{0}\right\} $ (the case $R_{0}=+\infty $
corresponds to the trivial total event), then by setting%
\begin{equation*}
\widetilde{m}_{C}=\left\{ s\in m\text{ };\text{ }\mathcal{F}\left( s\right)
\leq C\And \mathcal{G}\left( s\right) \leq R_{0}\right\} \text{ and }%
\widetilde{d}_{C}=\left\{ s\in m\text{ };\text{ }\mathcal{F}\left( s\right)
=C\And \mathcal{G}\left( s\right) \leq R_{0}\right\} ,
\end{equation*}%
it holds%
\begin{equation}
\ell \left( s_{m},\widehat{s}_{m}\right) \in \arg \max_{C\geq 0}\left\{
\sup_{s\in \widetilde{d}_{C}}\left\{ \left( P_{n}-P\right) \left( \gamma
\left( s_{m}\right) -\gamma \left( s\right) \right) \right\} -C\right\} 
,  \label{formula_excess_risk}
\end{equation}
\begin{equation}
\ell \left( s_{m},\widehat{s}_{m}\right) \in \arg \max_{C\geq 0}\left\{
\sup_{s\in \widetilde{m}_{C}}\left\{ \left( P_{n}-P\right) \left( \gamma
\left( s_{m}\right) -\gamma \left( s\right) \right) \right\} -C\right\} 
,  \label{formula_excess_risk_ball}
\end{equation}%
\begin{equation}
\ell _{%
%TCIMACRO{\TeXButton{emp}{\emp}}%
%BeginExpansion
\emp%
%EndExpansion
}\left( s_{m},\widehat{s}_{m}\right) =\max_{C\geq 0}\left\{ \sup_{s\in 
\widetilde{d}_{C}}\left\{ \left( P_{n}-P\right) \left( \gamma \left(
s_{m}\right) -\gamma \left( s\right) \right) \right\} -C\right\} ,
\label{formula_emp_excess_risk}
\end{equation}%
and 
\begin{equation}
\ell _{%
%TCIMACRO{\TeXButton{emp}{\emp}}%
%BeginExpansion
\emp%
%EndExpansion
}\left( s_{m},\widehat{s}_{m}\right) =\max_{C\geq 0}\left\{ \sup_{s\in 
\widetilde{m}_{C}}\left\{ \left( P_{n}-P\right) \left( \gamma \left(
s_{m}\right) -\gamma \left( s\right) \right) \right\} -C\right\}
\label{formula_emp_excess_risk_ball}
\end{equation}
\end{prpstn}

 \begin{figure}[t]
 \centering
 \includegraphics[width=0.9\textwidth]{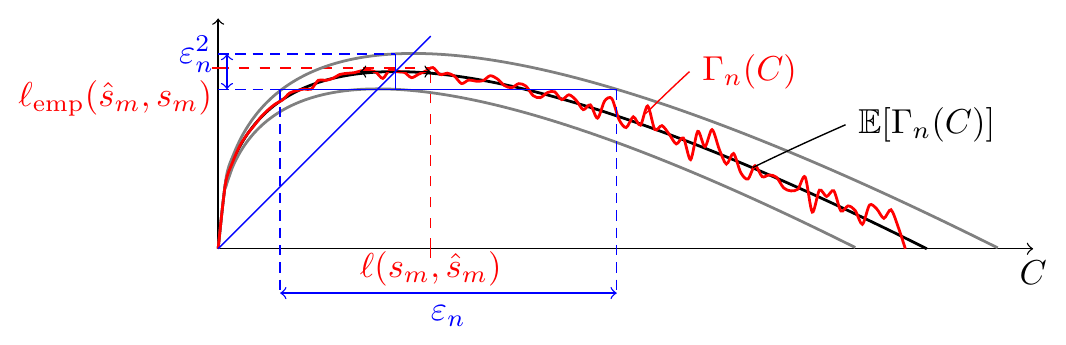}
 \caption{The true and empirical excess risks are given respectively as the maximizer and the maximum of the same function $\Gamma _{n}$. If $\Gamma _{n}$ is regular around its maximum this explains why concentration rate for the empirical excess risk---given by $\varepsilon _{n}^2$---is better than for true excess risk---given by $\varepsilon _{n}$.}
\label{fig:rep}
\end{figure}

The same type of excess risks representation as the one obtained in (\ref%
{formula_excess_risk_ball}) are at the core of the approach to excess risk's
concentration recently developed by Chatterjee \cite{chatterjee2014}, Muro
and van de Geer \cite{MurovandeGeer:15} and van de Geer and Wainwright \cite%
{vandeGeerWain:16}. The main difference with our approach is that these
authors rather use the parametrization $t=\sqrt{C}$ and take into advantage
an argument of concavity with respect to $t$ of the supremum of the
empirical process on "balls" of excess risk smaller than $t^{2}$. We refer
to van de Geer and Wainwright \cite{vandeGeerWain:16} for more details about
this concavity argument (called "second order margin condition" by these
authors). But with this concavity argument, nothing can be said \textit{a
priori} about the point around which the excess risk concentrates. To obtain
optimal bounds on this point, as in Theorem \ref%
{theorem_excess_risk_strong_loc}\ above, we rather apply a technology
developed in \cite{saum:12} and based on the least-squares \textit{contrast
expansion }around the projection $s_{m}$ of the target. We refer to Section
3 of \cite{saum:12} for a detailed presentation of the latter approach.

Proposition \ref{prop_excess_rep_loc} also allows to make it transparent the
fact the empirical excess risk has better concentration rates---given by the
term $\varepsilon _{n}^{2}$ in Theorem \ref{theorem_excess_risk_strong_loc}---than the excess risk---which concentrates at the rate $\varepsilon _{n}$.
Indeed, if we set 
\begin{equation*}
\Gamma _{n}\left( C\right) :=\sup_{s\in \widetilde{d}_{C}}\left\{ \left(
P_{n}-P\right) \left( \gamma \left( s_{m}\right) -\gamma \left( s\right)
\right) \right\} -C,
\end{equation*}%
with $\left\{ \mathcal{G}\left( \widehat{s}_{m}\right) \leq R_{0}\right\}
=\left\{ \left\Vert \widehat{s}_{m}-s_{m}\right\Vert _{\infty }\leq L\sqrt{%
\frac{D_{m}\ln n}{n}}\right\} $, the proof of Theorem \ref%
{theorem_excess_risk_strong_loc} shows that $\Gamma _{n}\left( C\right) $
concentrates around the quantity $2\sqrt{n^{-1}\mathcal{C}_{m}C}-C$, which
is parabolic around its maximum. The conclusion can now be directly read in
Figure~\ref{fig:rep}.
\section{Numerical experiments}\label{section_experiments}
A simulation study was conducted in order to compare the numerical performances of the model selection procedures we have discussed. We consider wavelet models as non trivial illustrative examples of the theory developed above for the selection of linear estimators using the slope heuristics and $V$-fold model selection. However, it is a rather different question than designing the best possible estimators using wavelet expansions, since these estimators are likely to be nonlinear as for the thresholding strategies (see e.g., \cite{anto:01} for a comparative simulation study of wavelet based estimators). Although a linear wavelet estimator is not as flexible, or potentially as powerful, as a nonlinear one, it still preserves the computational benefits of wavelet methods. See e.g., \cite{antoniadis:94} which is a key reference for linear wavelet methods in nonparametric regression.
All simulations have been conducted with Matlab and the wavelet toolbox Wavelab850 \cite{donoho:06} that is freely available from \url{http://statweb.stanford.edu/~wavelab/}. In order to reproduce all the experiments, the codes used to generate the numerical results presented in this paper will be available online at \url{https://github.com/fabnavarro}.

\subsection{Computational aspects}
For sample sizes $n=256,1024,4096$, data were generated according to
\[
Y_i = s_*(X_i)+\sigma(X_i)\varepsilon_i, \quad i=1,\ldots,n
\]
where $X_i$'s are uniformly distributed on $[0,1]$, $\varepsilon_i$'s are independent $\mathcal{N}(0,1)$ variables and independent of $X_i$'s. 
In the case of fixed design, thanks to Mallat's pyramid algorithm (see \cite{mallat:08}), the computation of wavelet-based estimators is straightforward and fast.
In the case where the function $s_*$ is observed on a random grid, the implementation requires some extra precautions and several strategies have been proposed in the literature (see e.g. \cite{cai:98,hall:97}). In the context of random uniform design regression estimation, \cite{cai:99} have examined convergence rates when the unknown function is in a H\"{o}lder class. They showed that the standard equispaced wavelet method with universal thresholding can be directly applied to the nonequispaced data (without a loss in the rate of convergence). In this simulations study, we have adopted this approach, since it preserves the computational simplicity and efficiency of the equispaced algorithm. The same choice was made in the context of wavelet regression in random design with heteroscedastic dependent errors by \cite{kulik:09}. Thus, in this case, the collection of models is computed by a simple application of Mallat's algorithm using the ordered $Y_{i}$'s as input variables.

\subsection{Examples}
\begin{figure}[!t]
\centering
\subfigure[\textit{Wave}]{\includegraphics[width=0.23\textwidth]{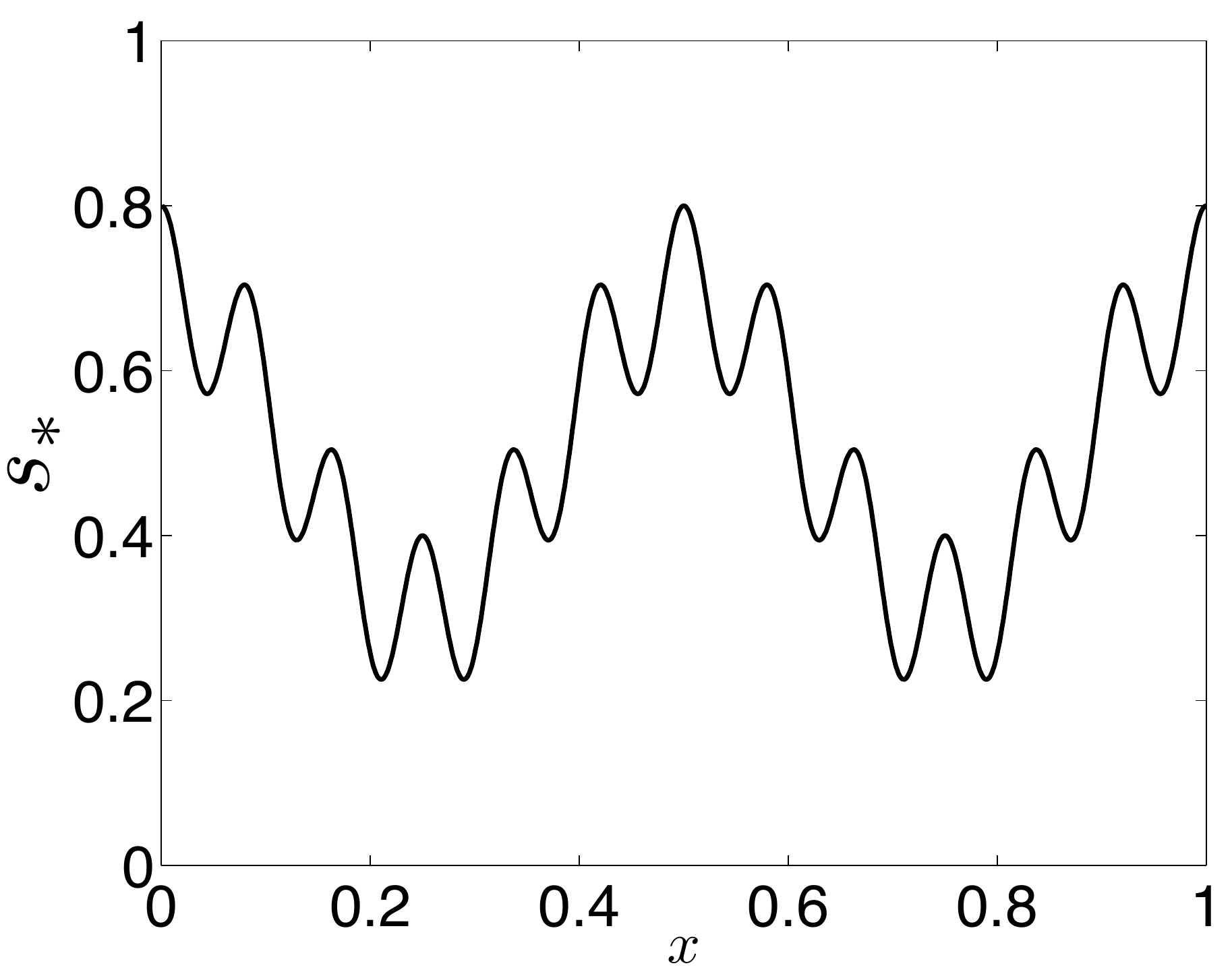}}
\subfigure[\textit{HeaviSine}]{\includegraphics[width=0.23\textwidth]{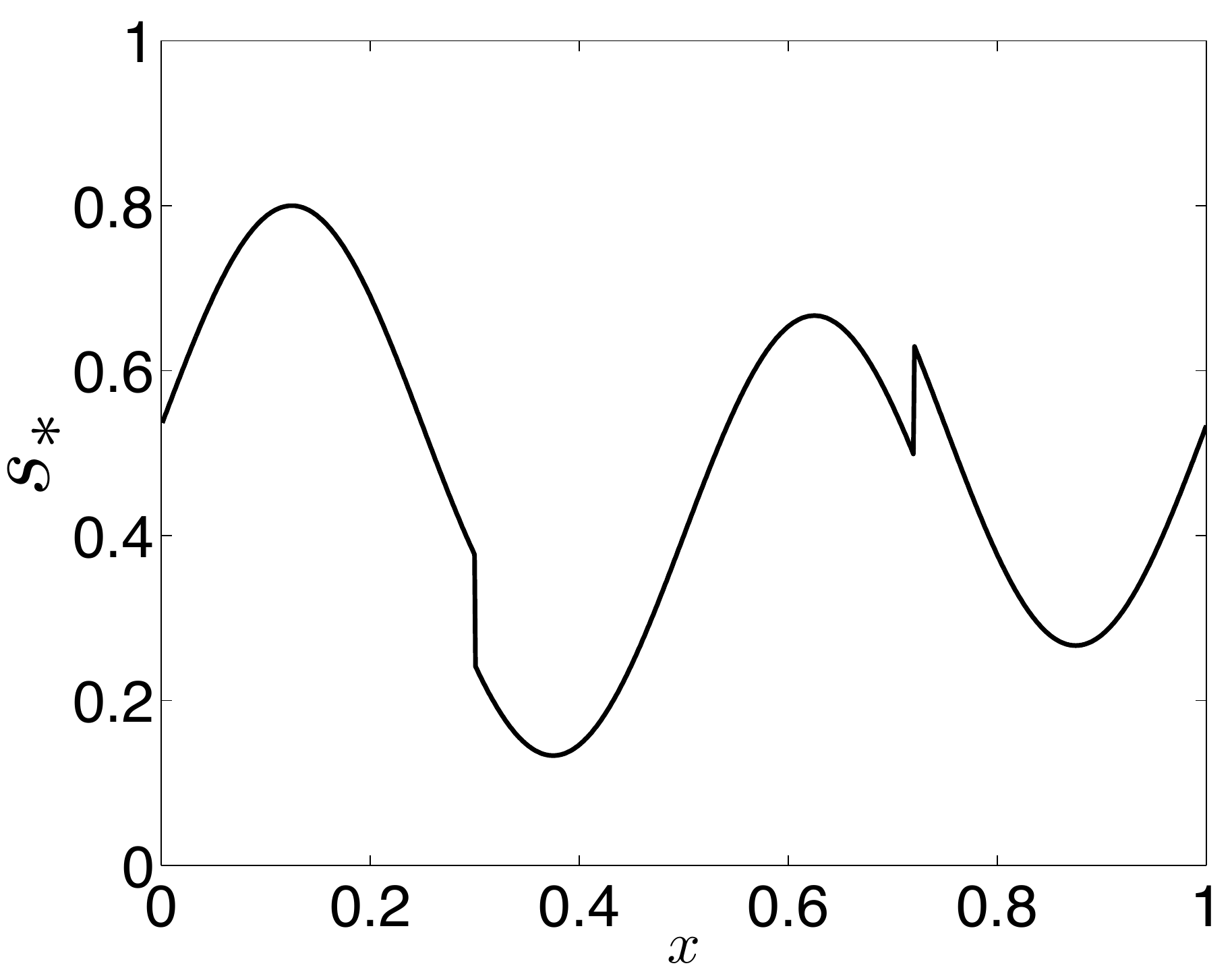}}
\subfigure[\textit{Doppler}]{\includegraphics[width=0.23\textwidth]{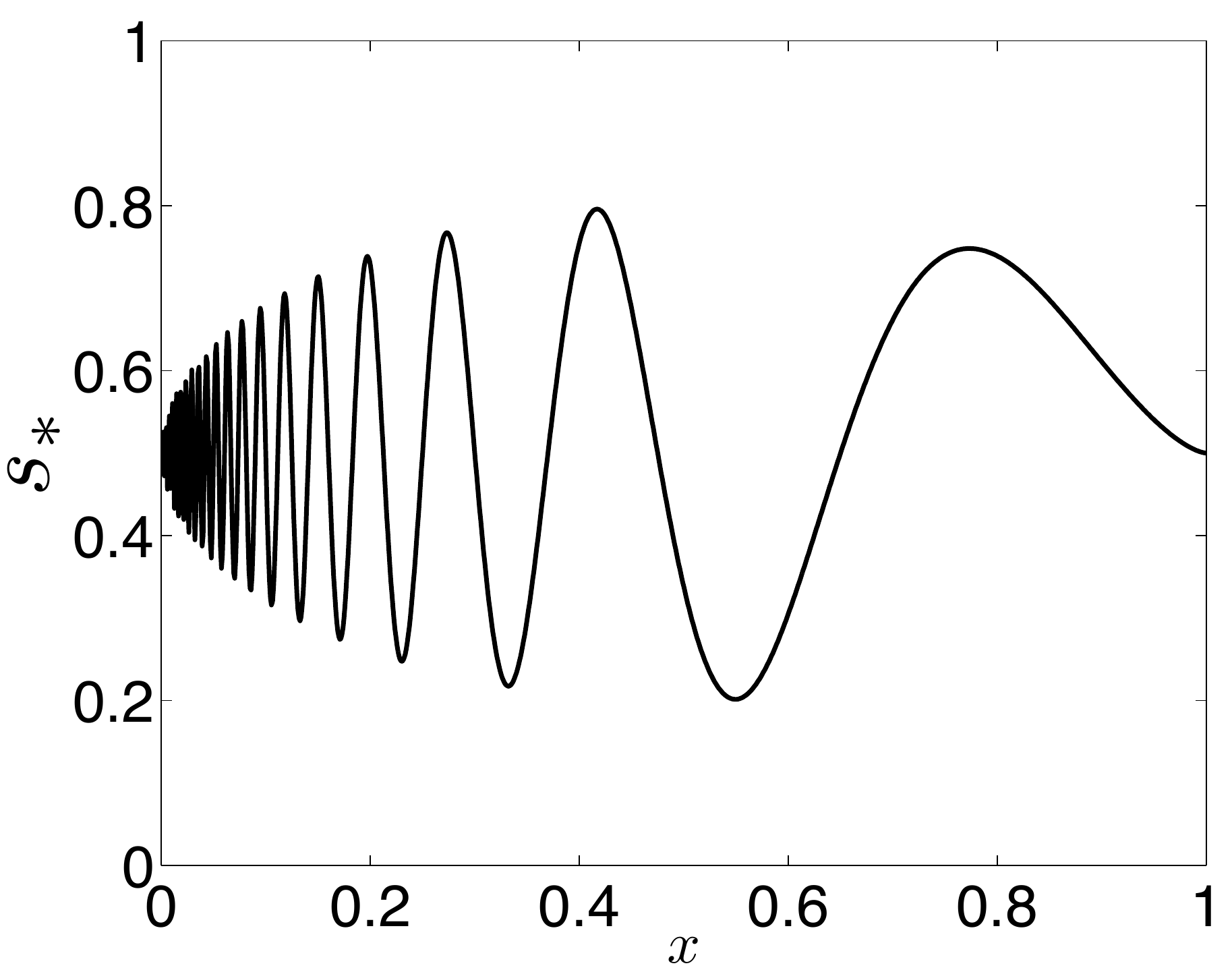}}
\subfigure[\textit{Spikes}]{\includegraphics[width=0.23\textwidth]{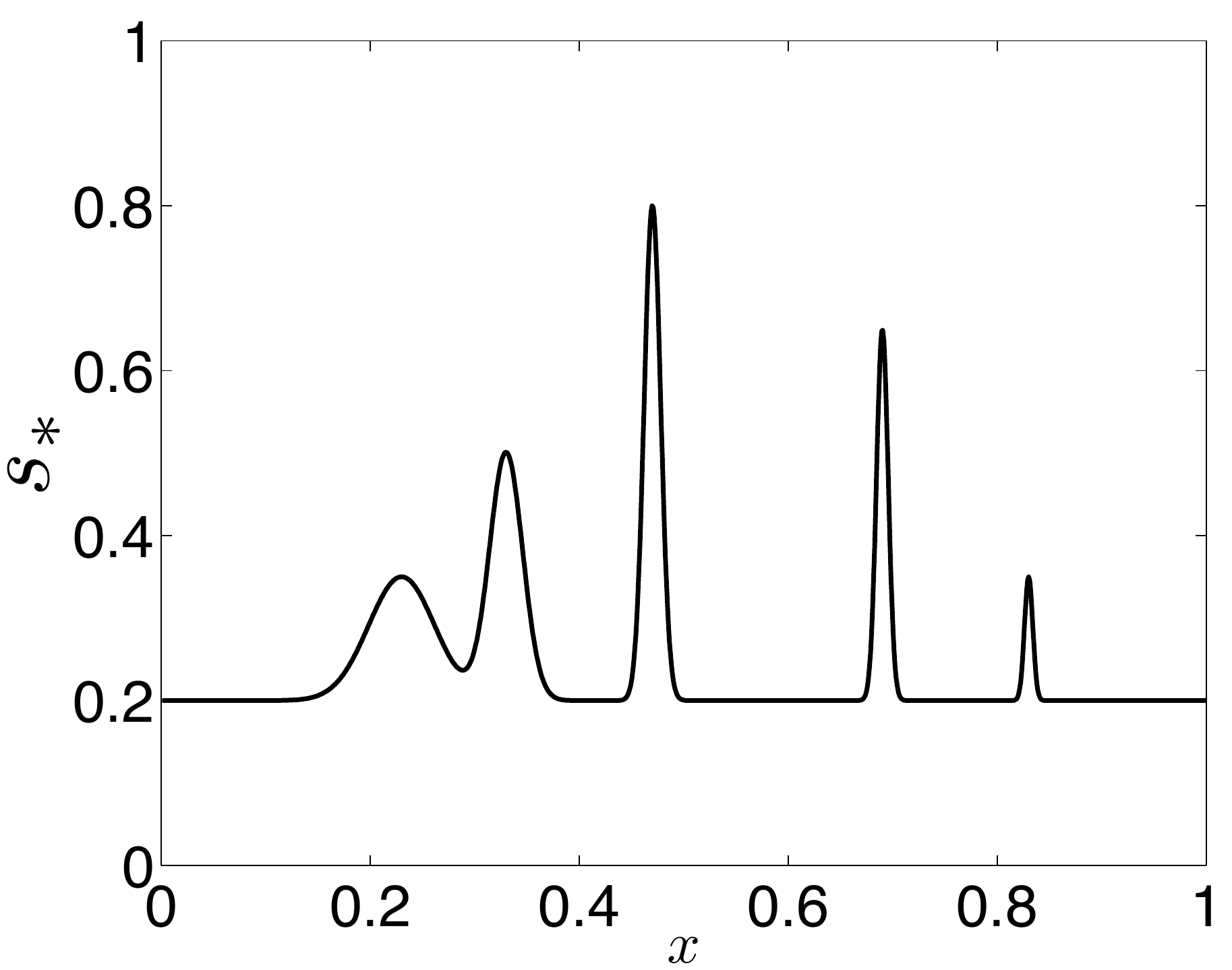}}

\subfigure[]{
\includegraphics[width=0.23\textwidth]{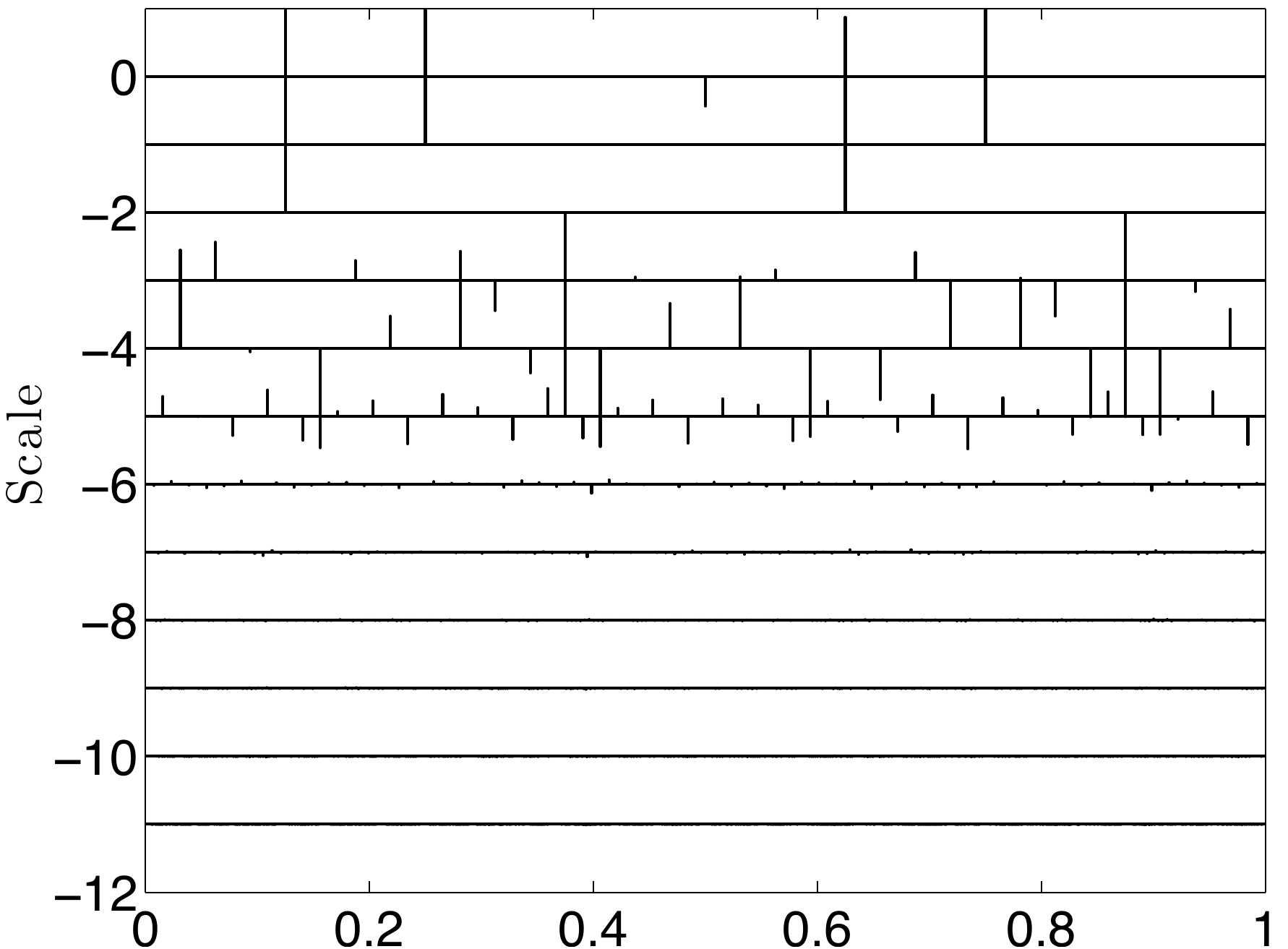}
\includegraphics[width=0.23\textwidth]{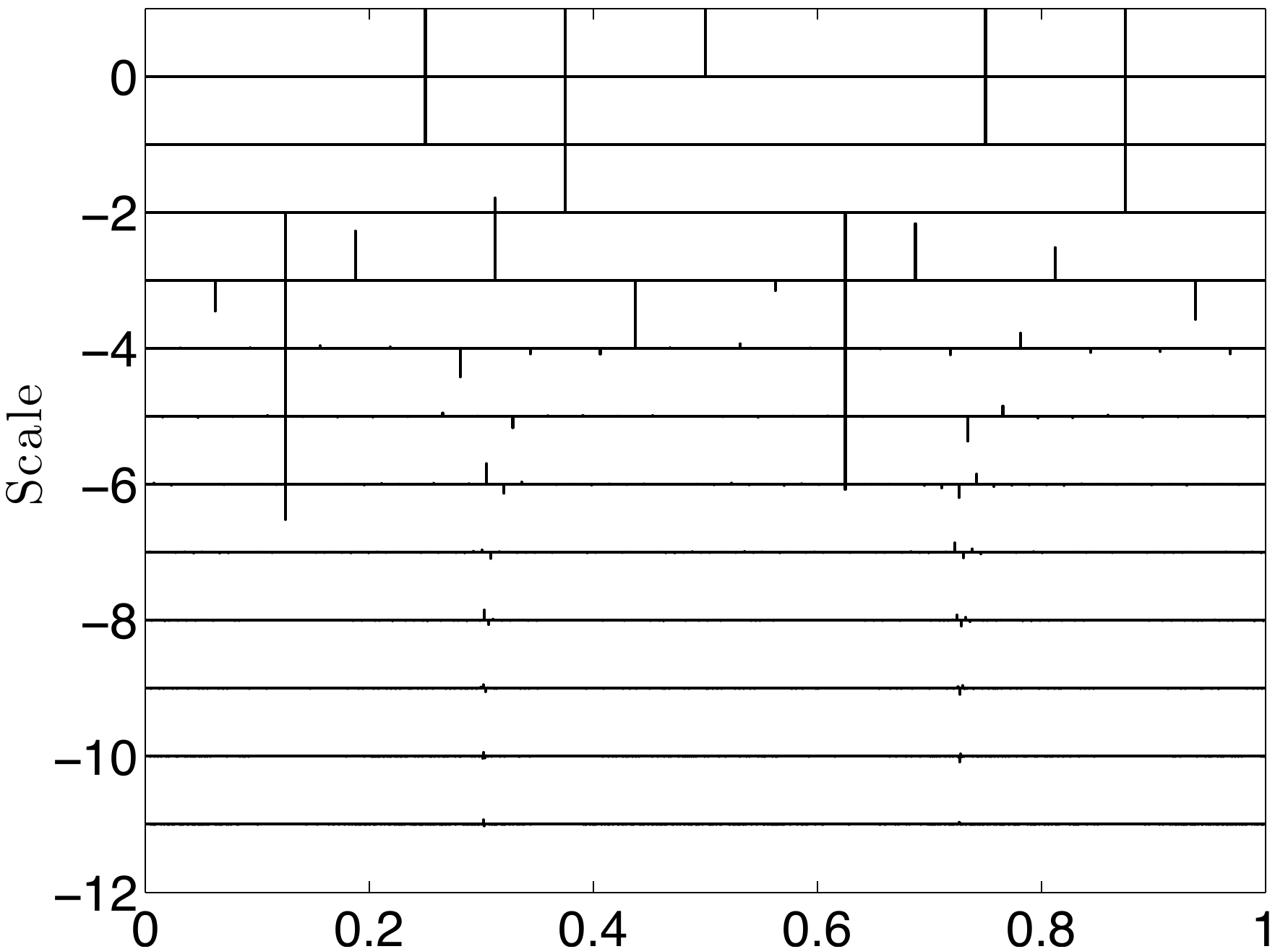}
\includegraphics[width=0.23\textwidth]{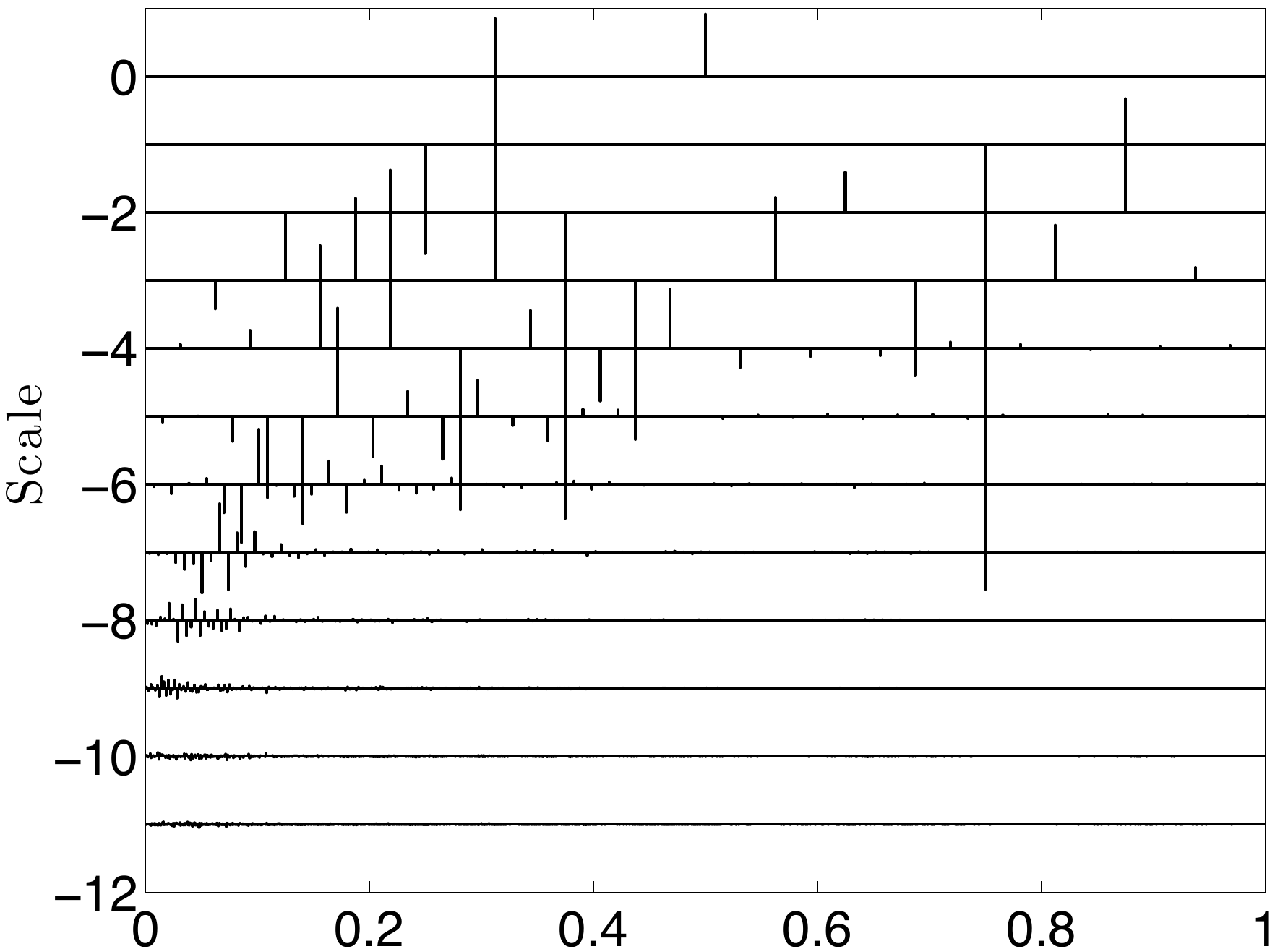}
\includegraphics[width=0.23\textwidth]{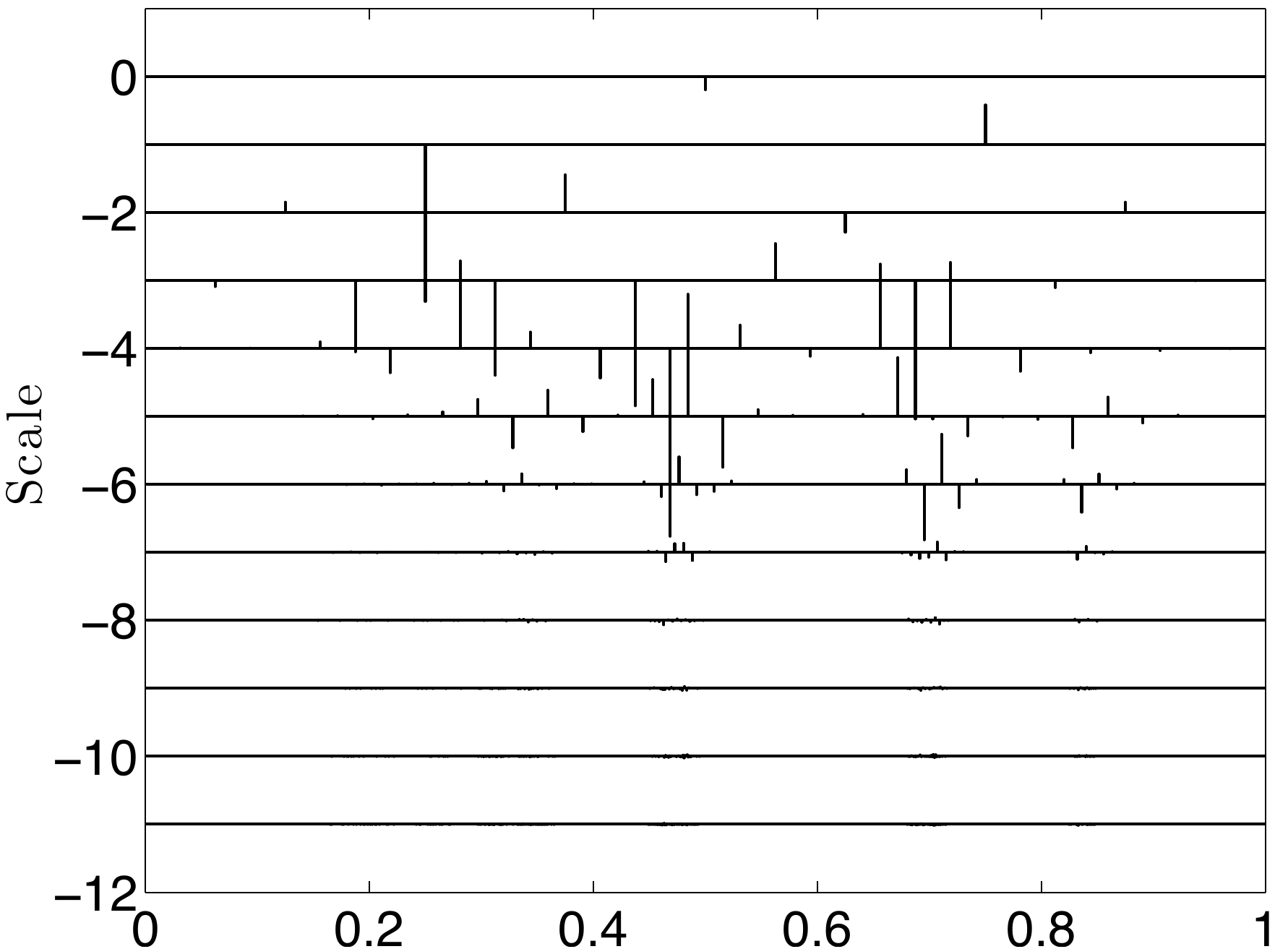}}

\caption{(a)--(d): The four test functions used in the simulation study sampled at 4096 points. (e): Wavelet coefficients of the test functions.}
\label{fig:target}
\end{figure} 

Four standard regression functions representing different level of spatial variability (\textit{Wave, HeaviSine, Doppler} and \textit{Spikes}, see \cite{donoho:94a,marron:98,cai:99a}) and the following four $\sigma(\cdot)$ scenarios were considered:
\begin{itemize}
\item[(a)] \textit{Low Homoscedastic Noise}: $\sigma_{l1}(x) = 0.01$;
\item[(b)] \textit{Low Heteroscedastic Noise}: $\sigma_{l2}(x) = 0.02x$;
\item[(c)] \textit{High Homoscedastic Noise}: $\sigma_{h1}(x) = 0.05$;
\item[(d)] \textit{High Heteroscedastic Noise}: $\sigma_{h2}(x) = 0.1x$. 
\end{itemize}
The test functions are plotted in Figure \ref{fig:target} and a visual idea of the four noise levels is given in Figures~\ref{fig:singleSpikes}(a)--(d). Several different wavelets were used. In the following, we only report in detail the results for Daubechies' compactly supported wavelet with 8 vanishing moments. %\ref{fig:singleWave}(a)--(d) and 

\begin{figure}[t!]
\centering
\subfigure[$\sigma_{l1}(x)=0.01$]{\includegraphics[width=0.23\textwidth]{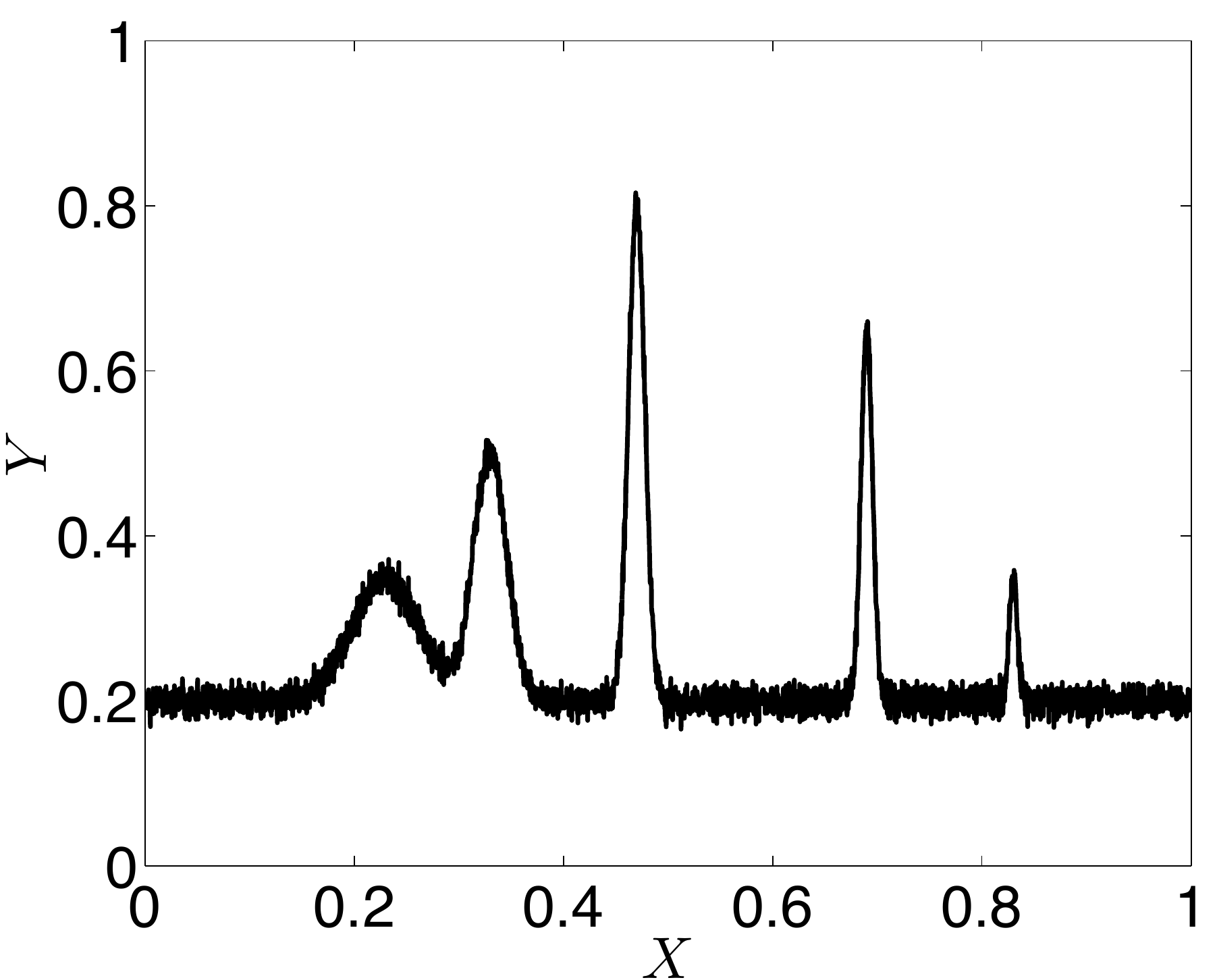}}
\subfigure[$\sigma_{l2}(x)=0.02x$]{\includegraphics[width=0.23\textwidth]{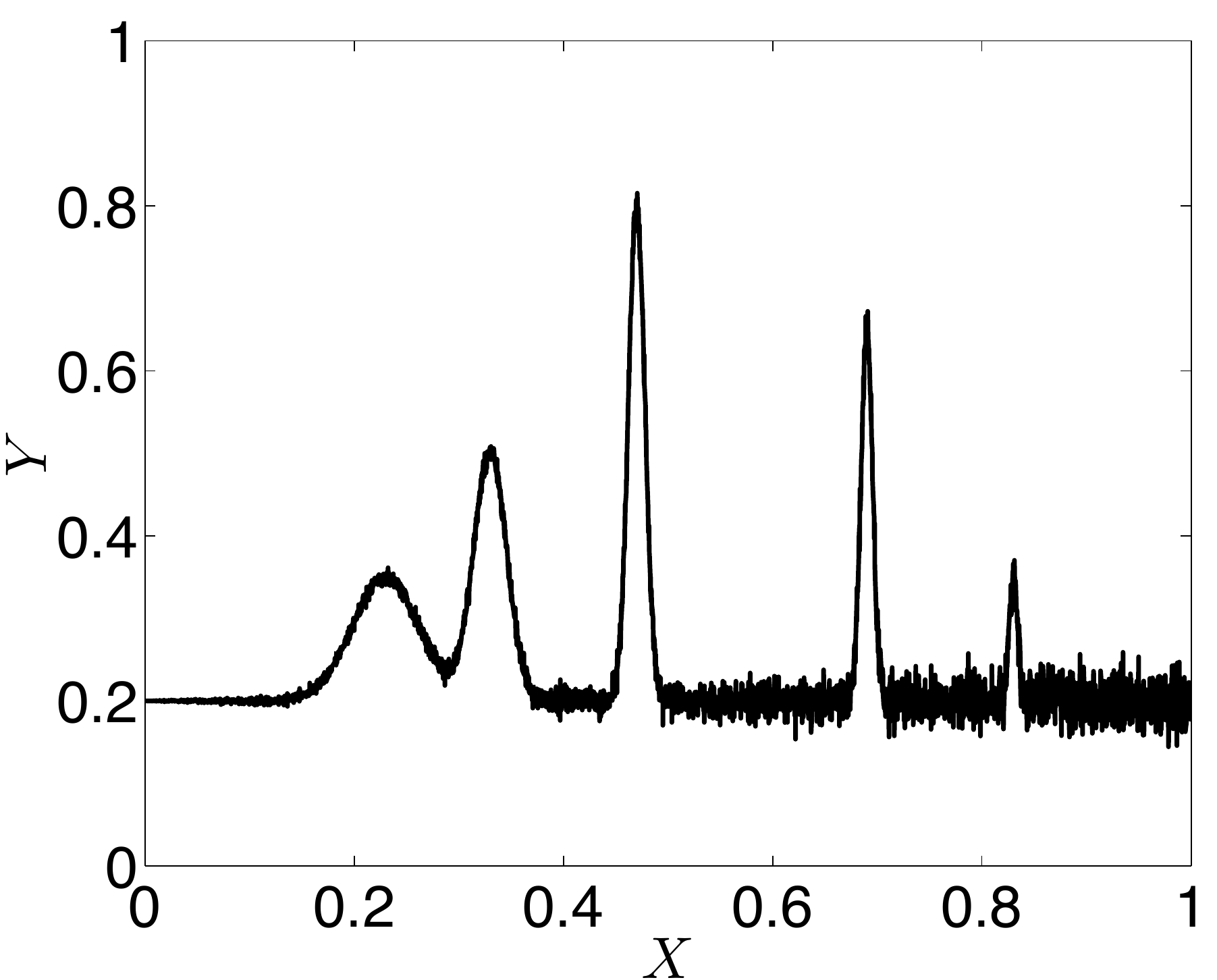}}
\subfigure[$\sigma_{h1}(x)=0.05$]{\includegraphics[width=0.23\textwidth]{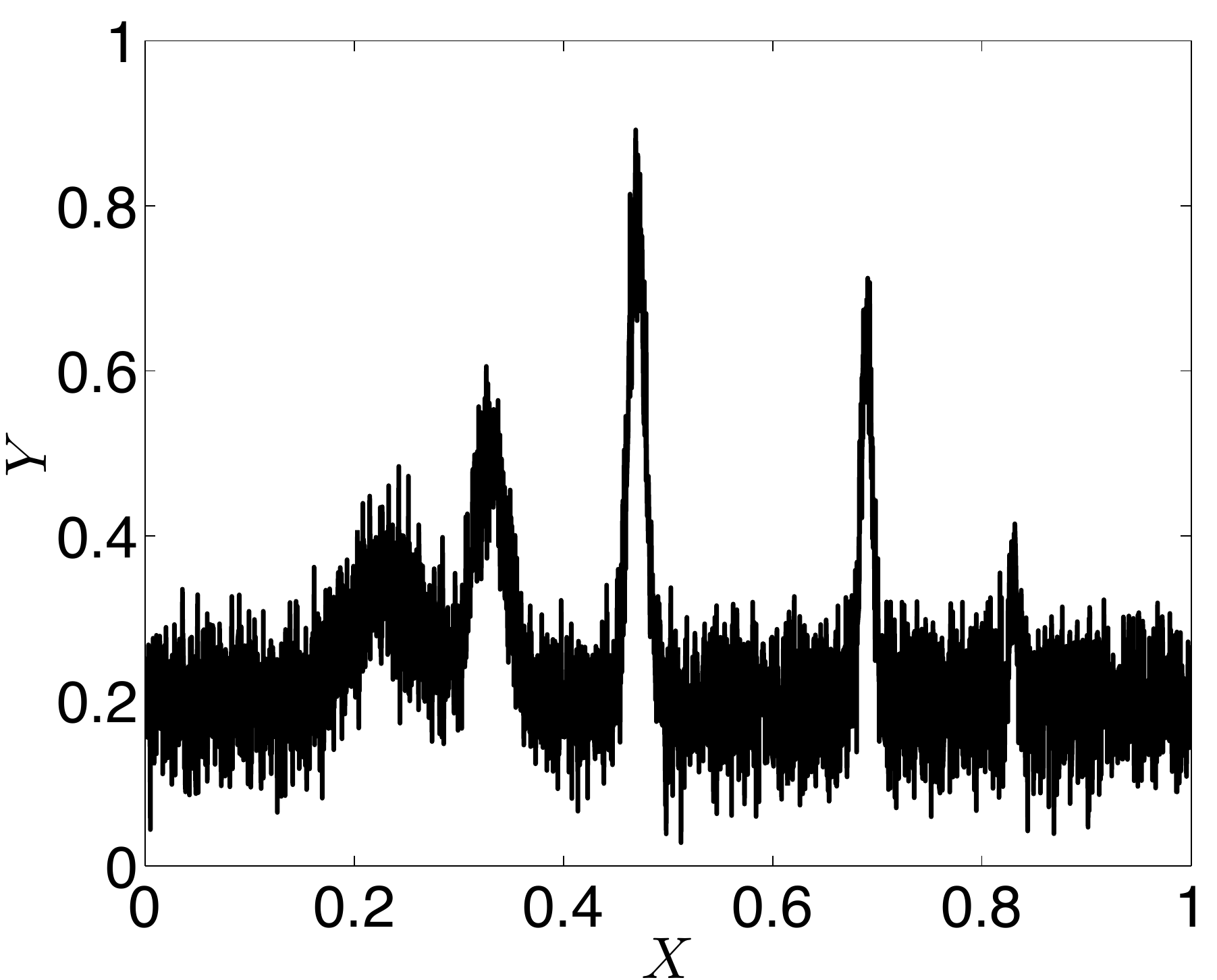}}
\subfigure[$\sigma_{h2}(x)=0.1x$]{\includegraphics[width=0.23\textwidth]{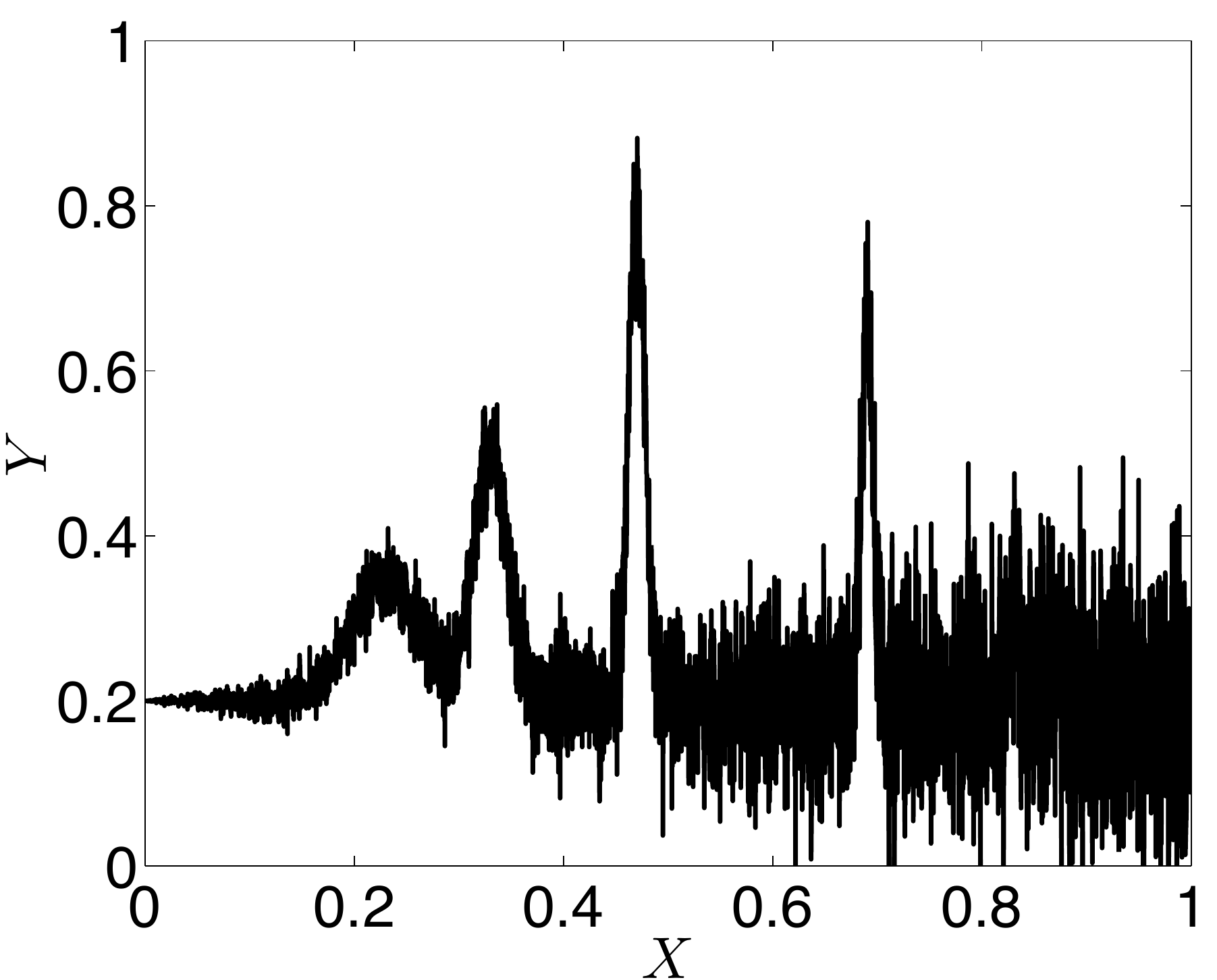}}
\subfigure[]{
\includegraphics[width=0.23\textwidth]{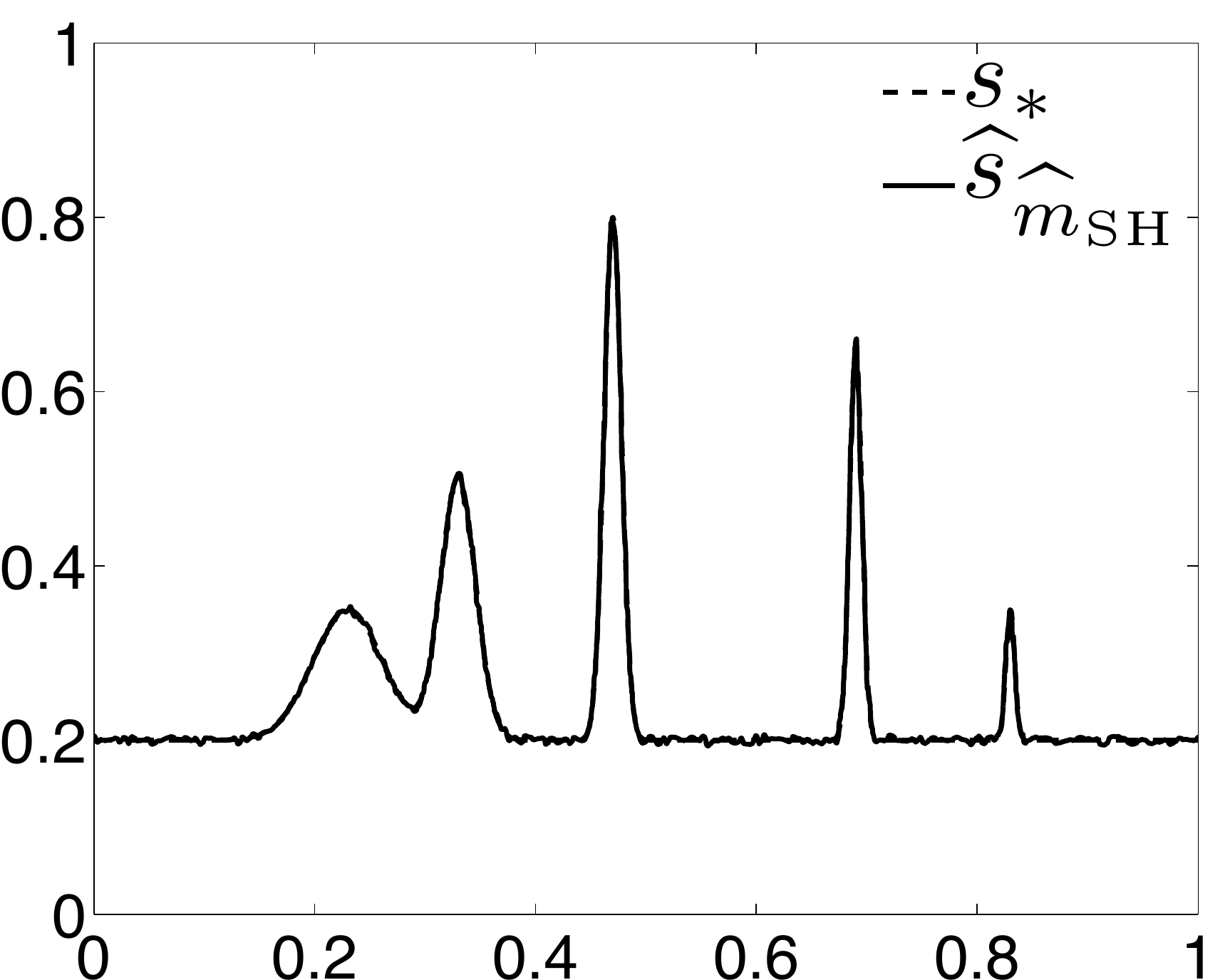}
\includegraphics[width=0.23\textwidth]{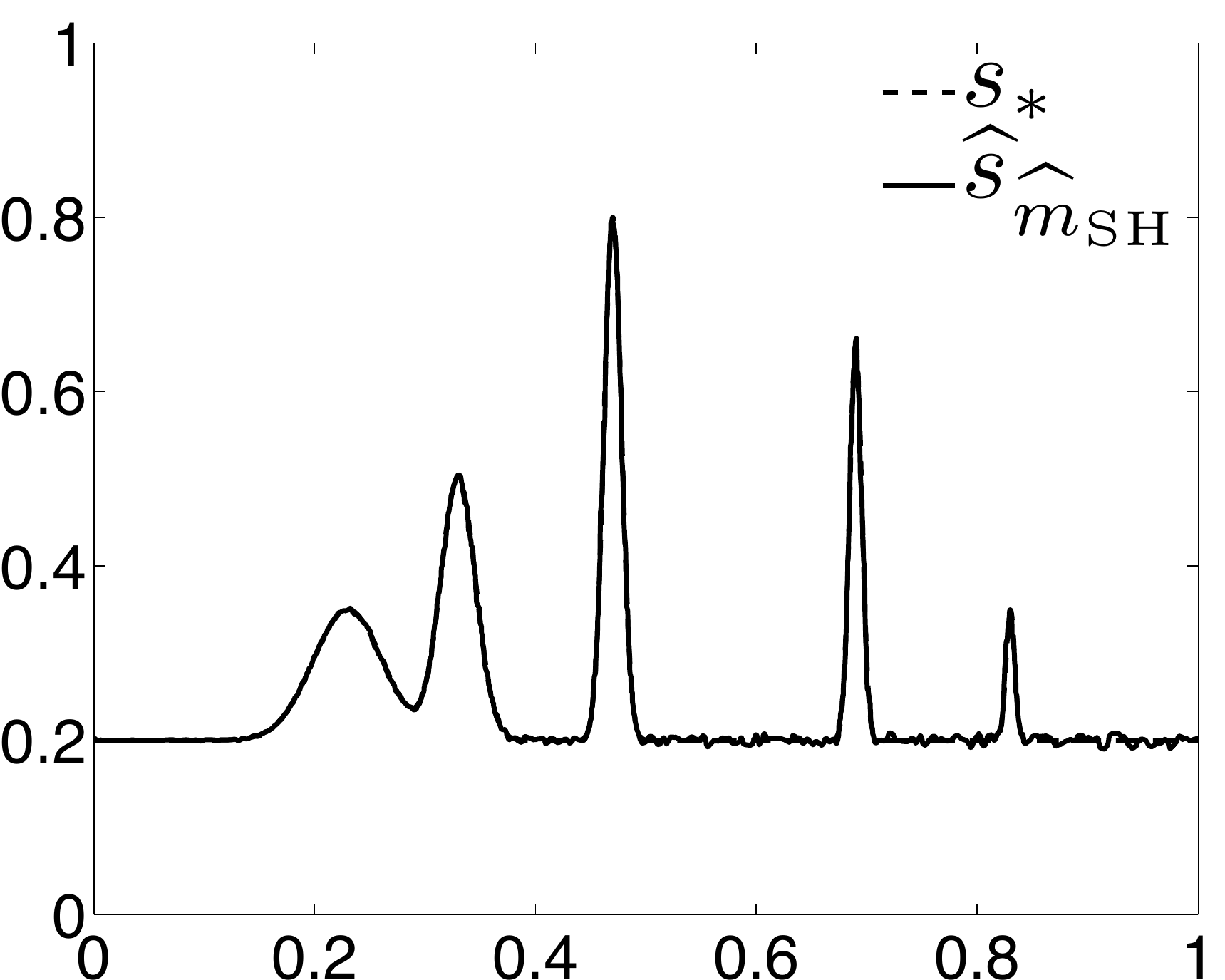}
\includegraphics[width=0.23\textwidth]{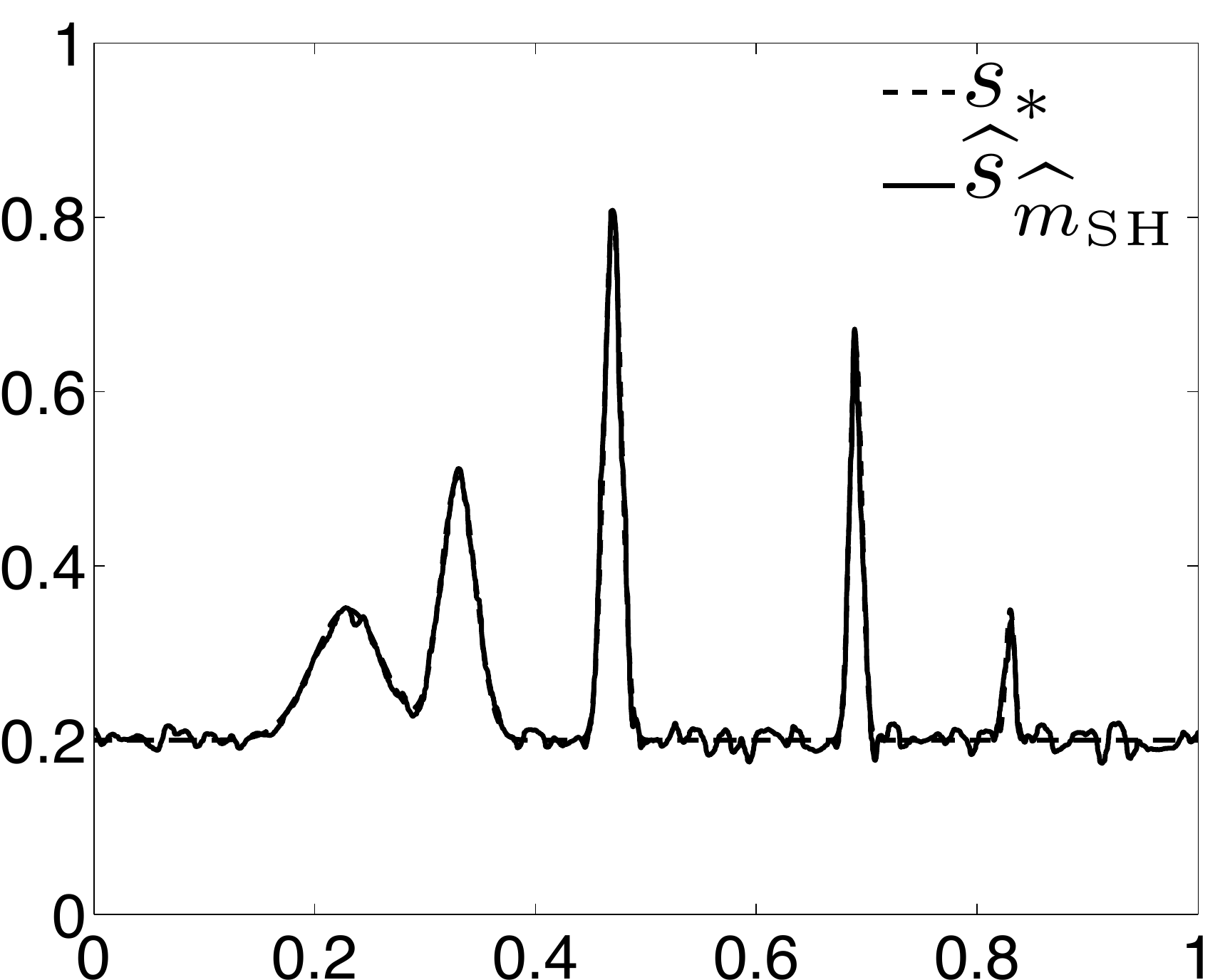}
\includegraphics[width=0.23\textwidth]{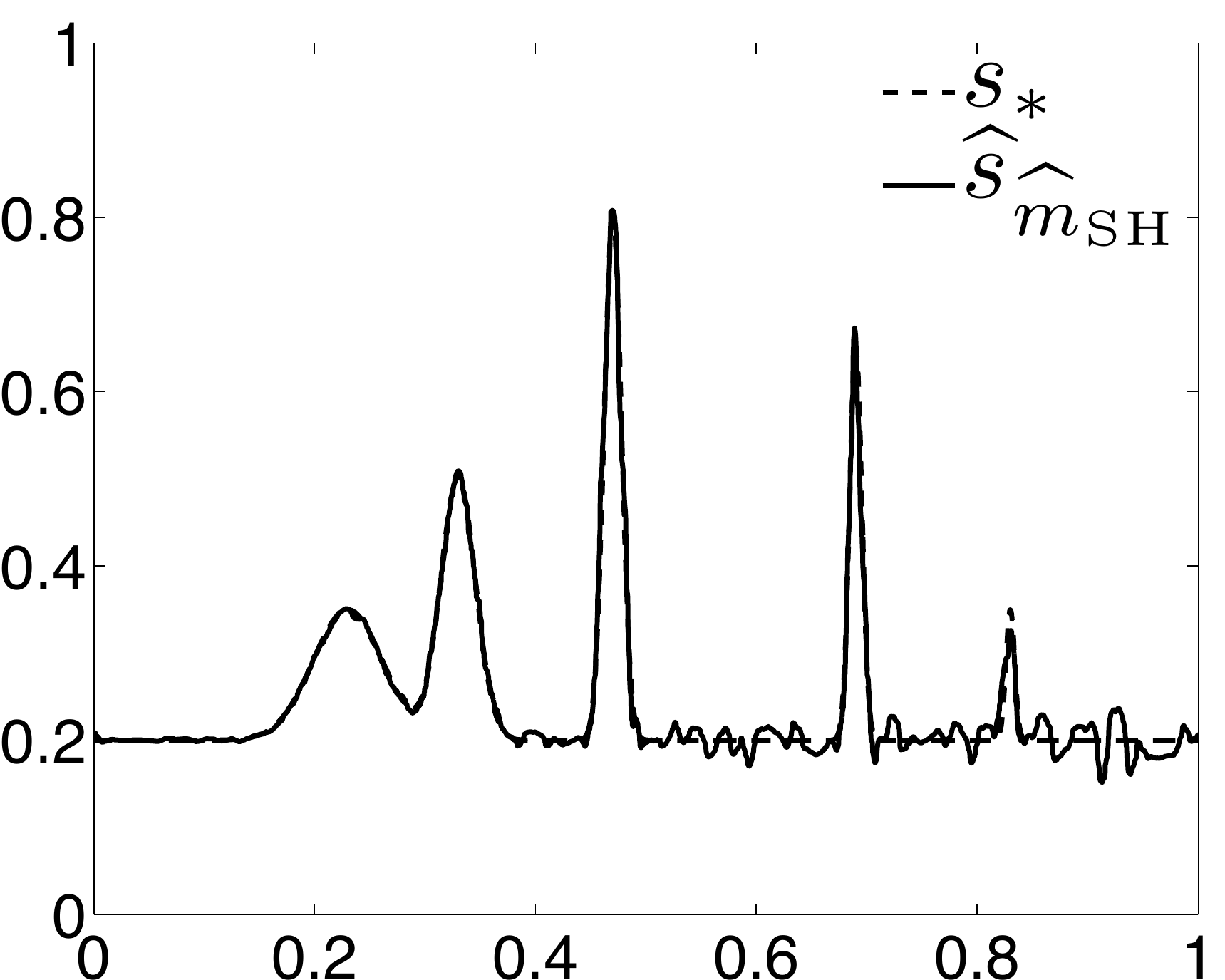}}
\subfigure[]{
\includegraphics[width=0.23\textwidth]{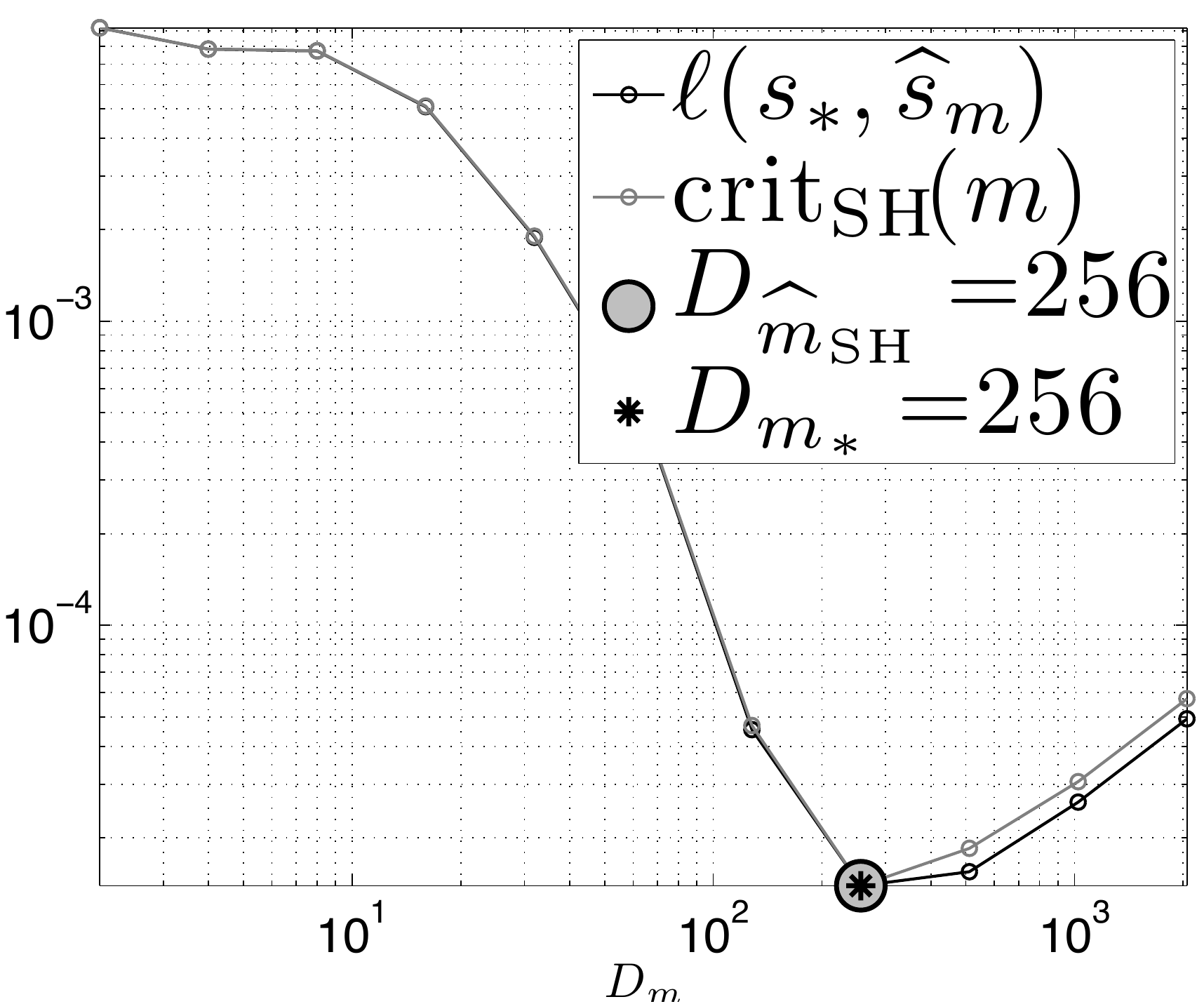}
\includegraphics[width=0.23\textwidth]{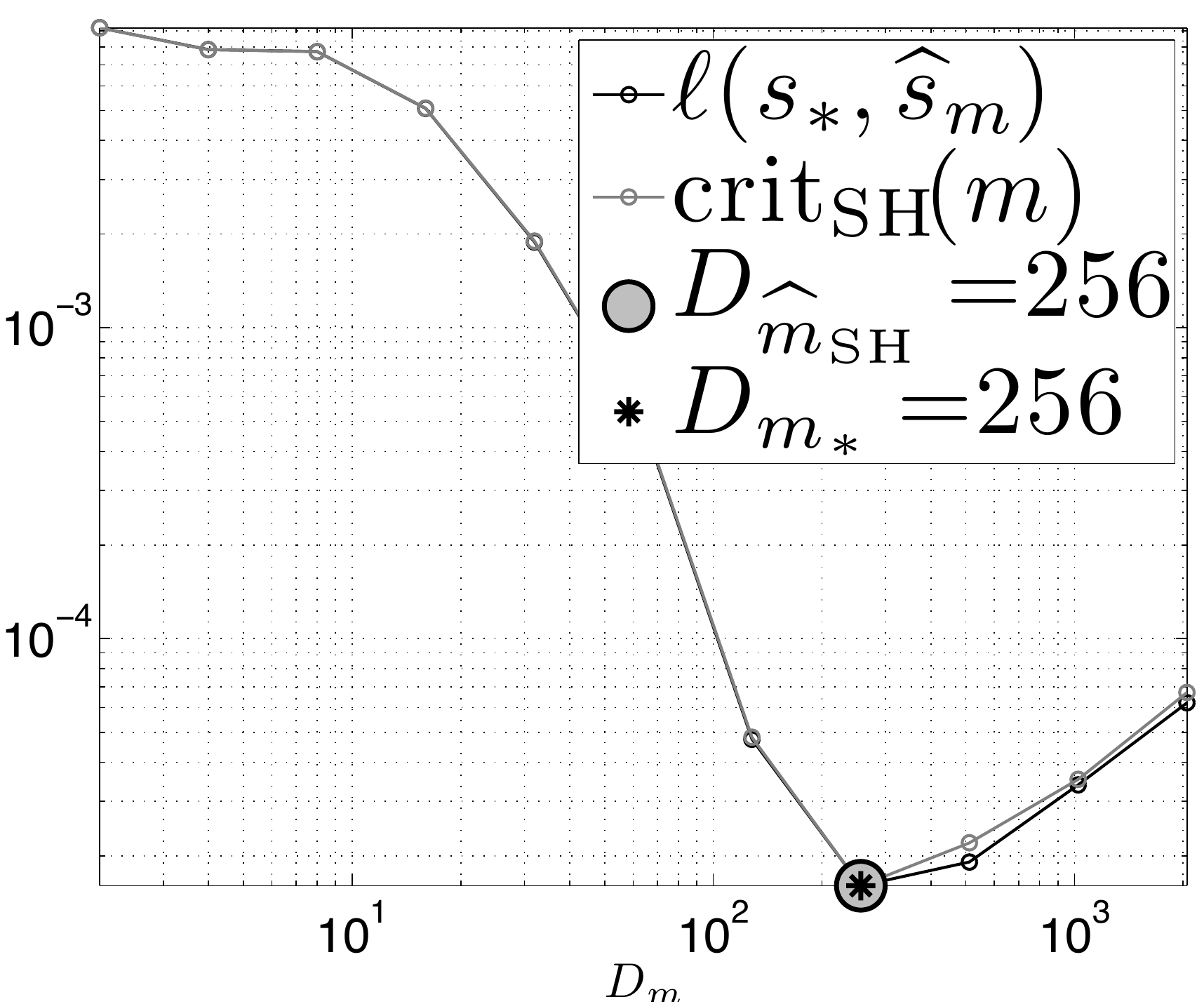}
\includegraphics[width=0.23\textwidth]{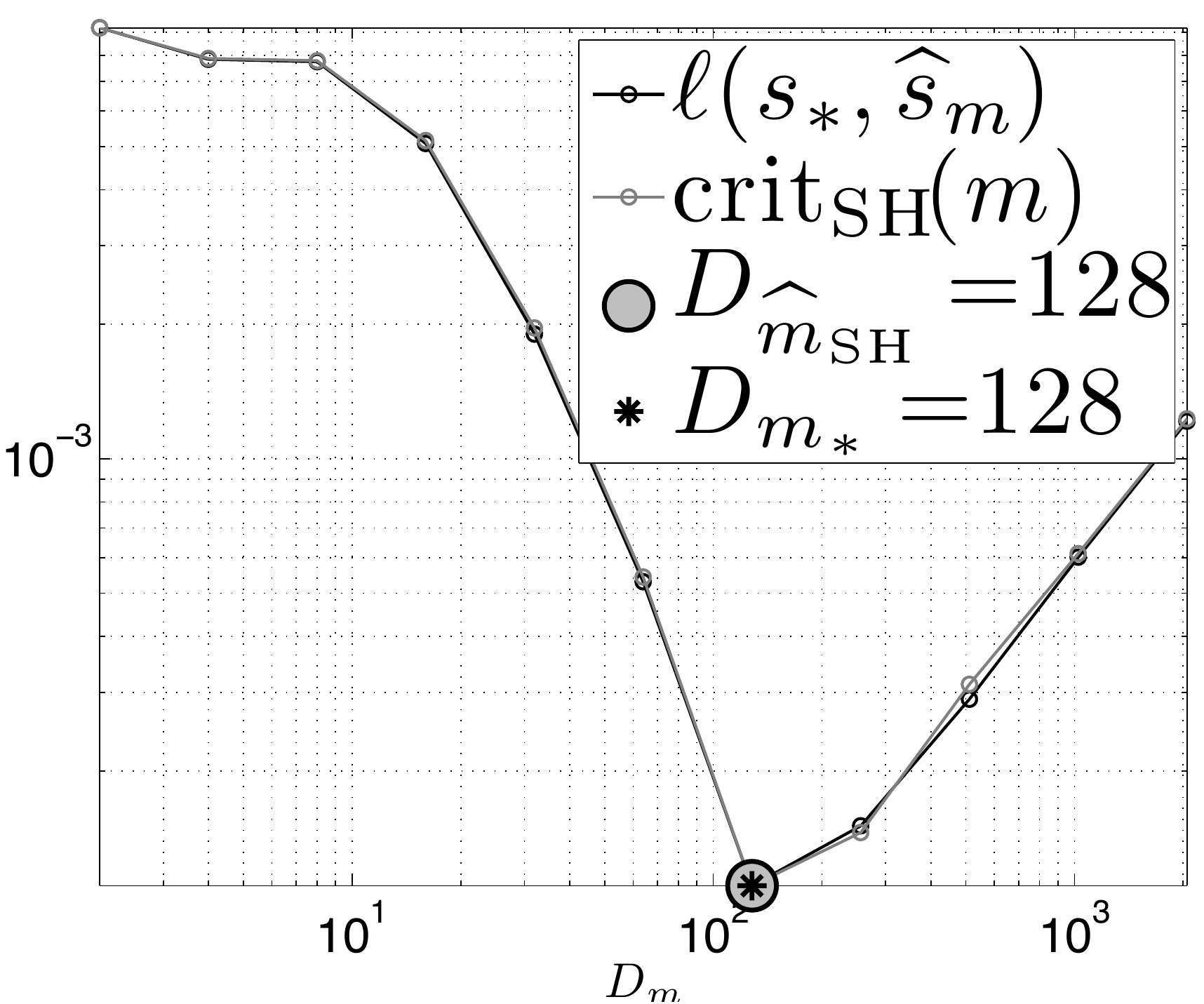}
\includegraphics[width=0.23\textwidth]{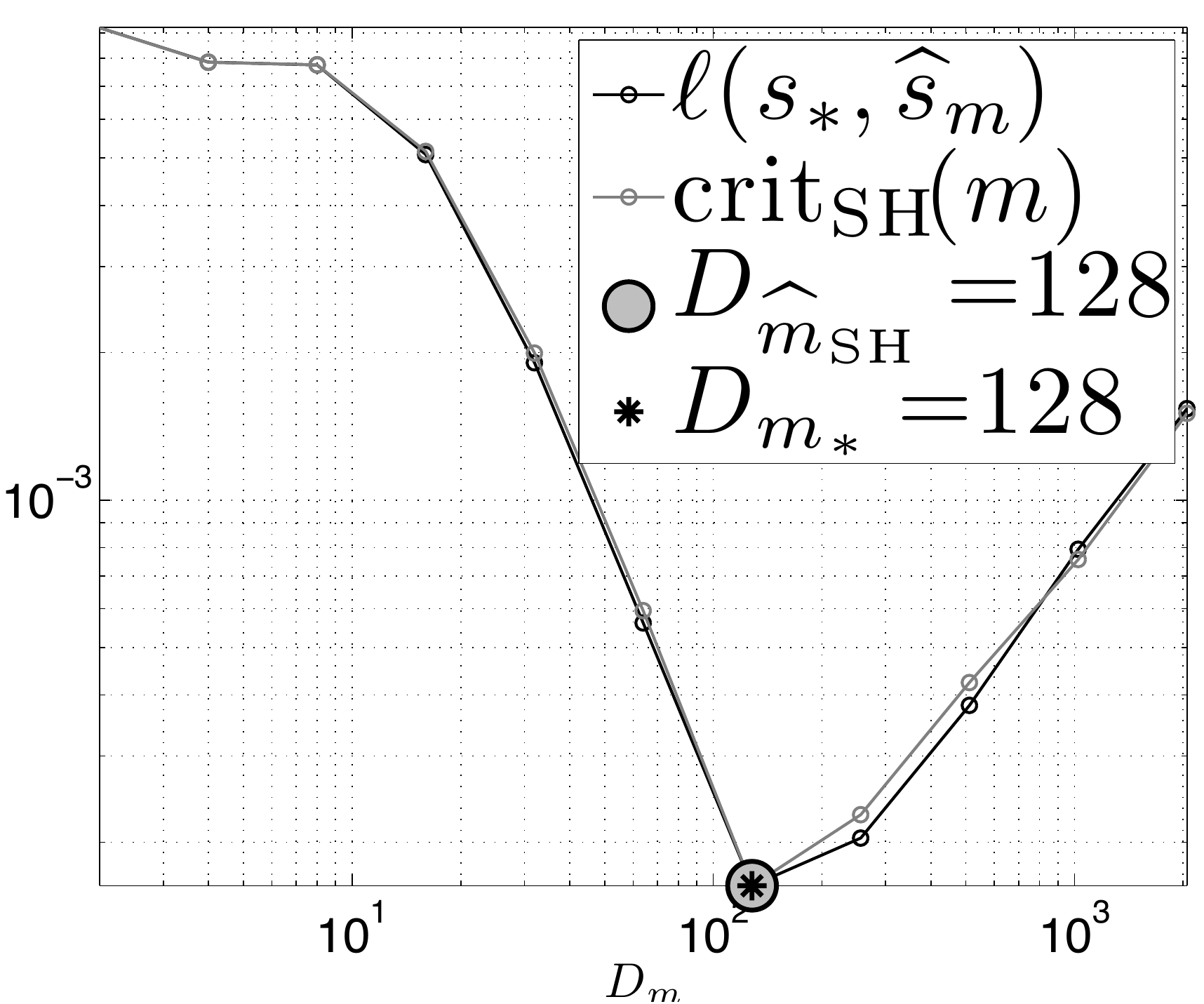}}
\subfigure[]{
\includegraphics[width=0.23\textwidth]{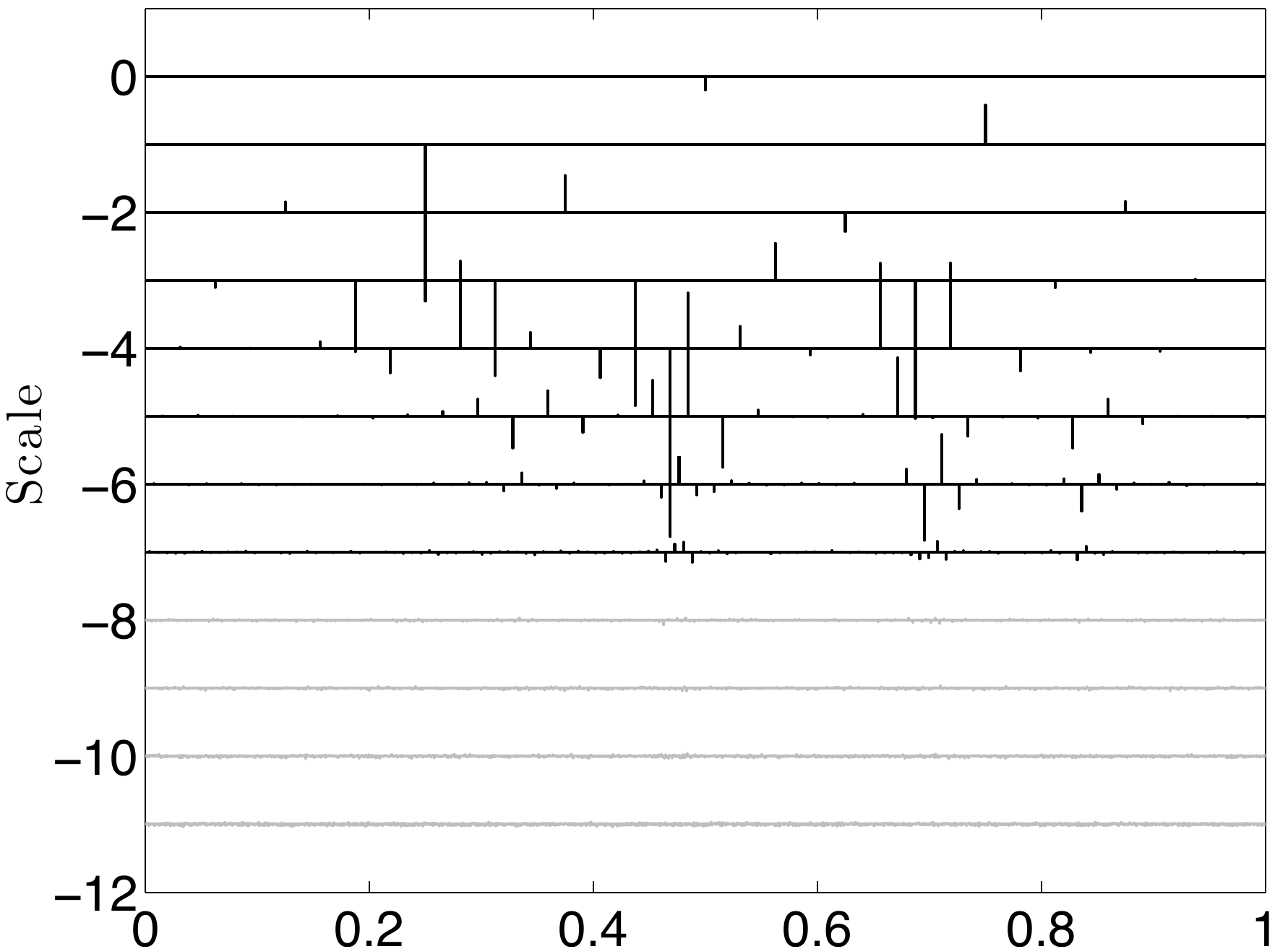}
\includegraphics[width=0.23\textwidth]{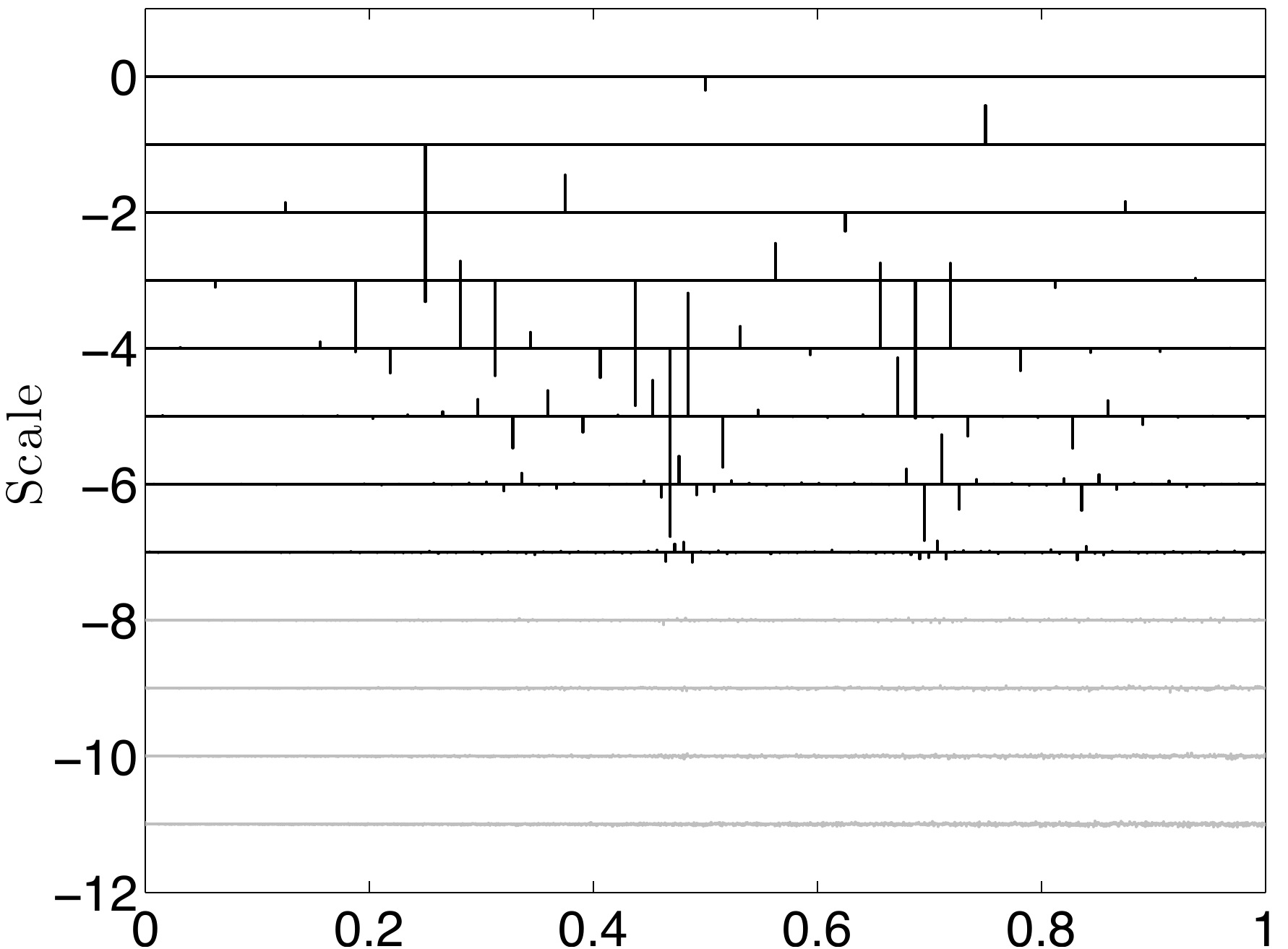}
\includegraphics[width=0.23\textwidth]{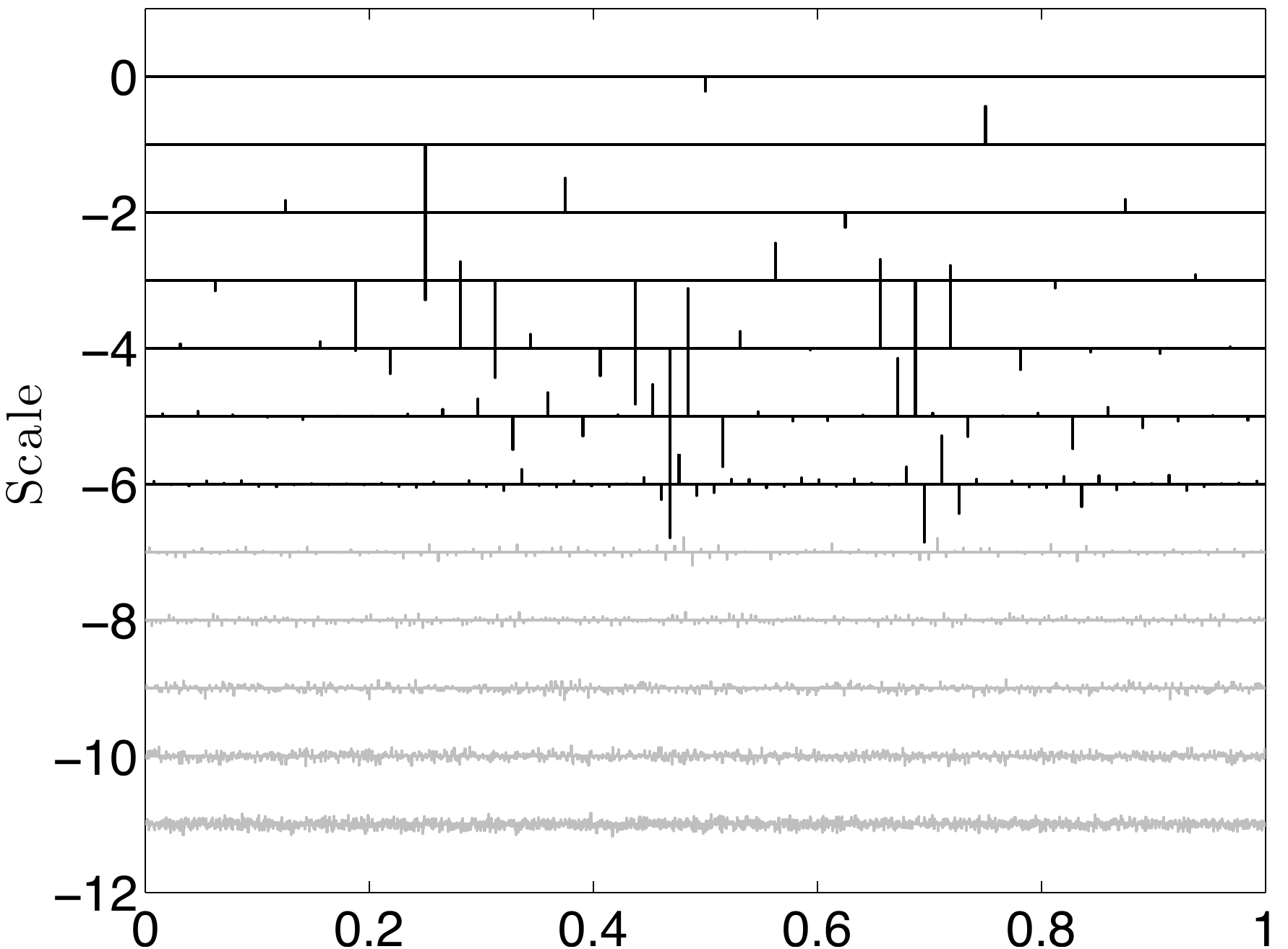}
\includegraphics[width=0.23\textwidth]{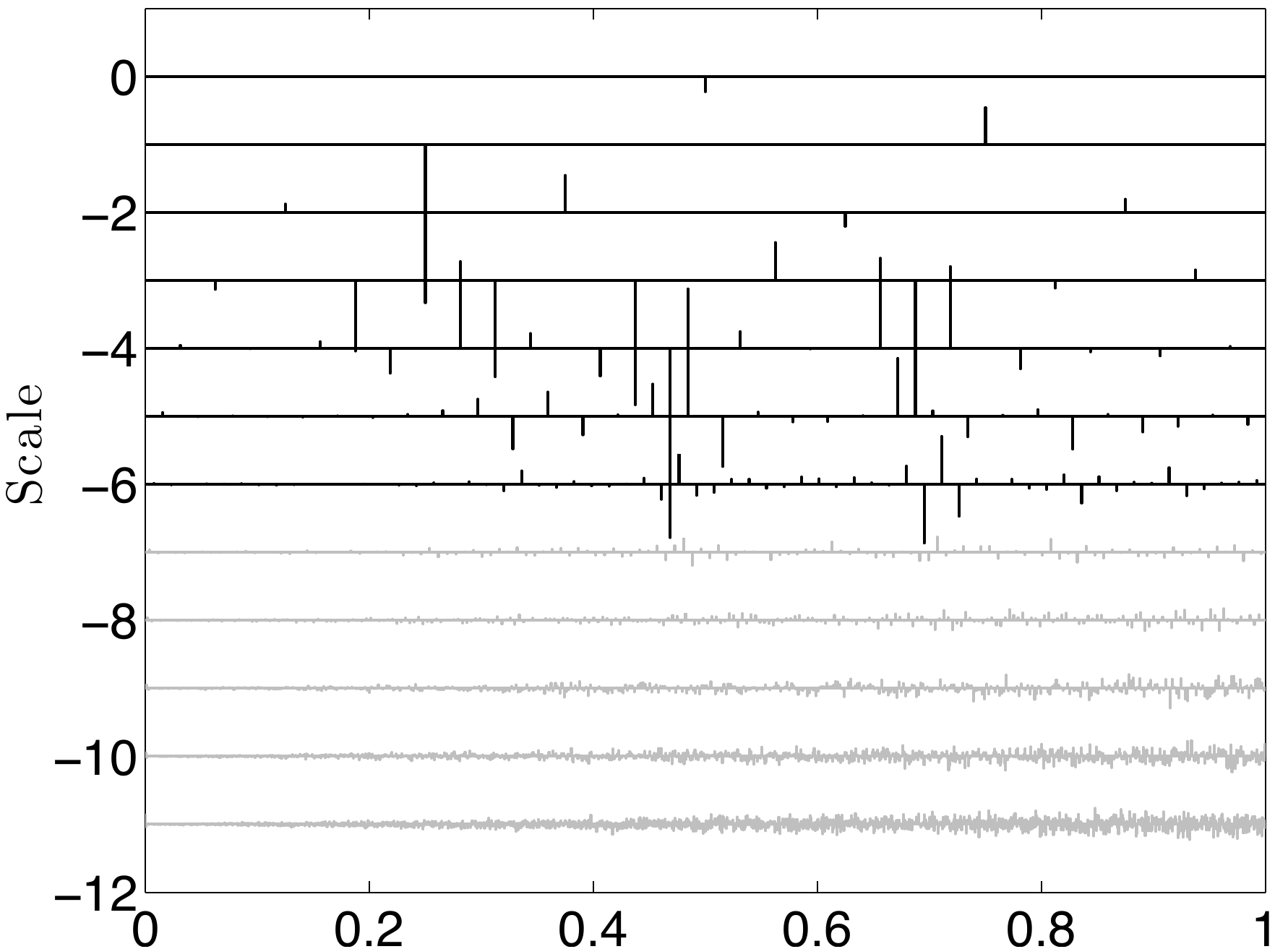}}

\caption{(a)-(d): Noisy version of \textit{Spikes} for each $\sigma(\cdot)$ scenarios. (e): Typical reconstructions from a single simulation with $n=4096$. The dotted line is the true signal and the solid one depicts the estimates $\widehat{s}_{\widehat{m}_{\mathrm{SH}}}$. (f): Graph of the excess risk $\ell(s_\ast, \widehat{s}_m)$ against the dimension $D_m$ and (shifted) $\mathrm{crit}_{\mathrm{SH}}(m)$ (in a log-log scale). The gray circle represents the global minimizer $\widehat{m}$ of $\mathrm{crit_{\mathrm{SH}}}(m)$ and the black star the oracle model $m_\ast$. (g): Noisy and selected (black) wavelet coefficients (see Figure~\ref{fig:target}(e) for a visual comparison with the original wavelet coefficients).}
\label{fig:singleSpikes}
\end{figure} 

\subsection{Four model selection procedures}%\label{procComp}
The performance of the following four model selection methods were compared: 
\begin{itemize}
\item The slope heuritics (SH): 
\begin{equation*}%\label{oracle_SH}
\widehat{m}_{\mathrm{SH}}\in\arg \min_{m\in \mathcal{M}_{n}}\left\{\mathrm{crit_{SH}}(m)\right\},
\end{equation*}
with
\[
\mathrm{crit_{SH}}(m) = P_{n}\left( \gamma \left( \widehat{s}_{m}\right) \right) +\pen_{\mathrm{SH}}(m),
\]
and
\[
\pen_{\mathrm{SH}}(m) = 2\frac{\widehat{\alpha}_{\mathrm{min}}D_m}{n}.
\]
where $\widehat{\alpha}_{\mathrm{min}}$ is obtained from the dimension jump method (see Figure~ \ref{fig:dimjump}). Practical issues about SH are addressed in \cite{BauMauMich:12} and our implementation is based on the Matlab package CAPUSHE. 
\item Mallow's $C_p$ (Cp): 
\begin{equation*}%\label{oracle_Cp}
\widehat{m}_{\mathrm{Cp}}\in\arg \min_{m\in \mathcal{M}_{n}}\left\{\mathrm{crit_{Cp}}(m)\right\},
\end{equation*}
with
\[
\mathrm{crit_{Cp}}(m) = P_{n}\left( \gamma \left( \widehat{s}_{m}\right) \right) +\pen_{\mathrm{Cp}}(m),
\]
and
\[
\pen_{\mathrm{Cp}}(m) = 2\frac{\widehat{\sigma}^2D_m}{n},
\]
where $\widehat{\sigma}^2$ is globally estimated  by the classical variance estimator defined as
\[
\widehat{\sigma}^2 = \frac{d^2(Y_{1\ldots n}, m_{n/2})}{n-n/2},
\]
where $Y_{1\ldots n}=\left(Y_i\right)_{1\leq i\leq n}\in \mathbb{R}^n$, $m_{n/2}$ is the largest model of dimension $n/2$, and $d$ is the Euclidean distance on $\mathbb{R}^n$.
\item Nason's 2-fold cross-validation (2FCV).
Nason adjusted the usual 2FCV method---which cannot be applied directly to wavelet estimation---for choosing the threshold parameter in wavelet shrinkage \cite{nason:96}. Adapting his strategy to our context, we test, for every model of the collection, an interpolated wavelet estimator learned from the (ordered) even-indexed data against the odd-indexed data and vice versa. More precisely, considering the data $X_i$ are ordered, the selected model $\widehat{m}_{\mathrm{2FCV}}$ is obtained by minimizing \eqref{def_crit_VFCV} with $V=2$, $B_1=\{2,4,\ldots,n\}$ and $B_2=\{1,3,\ldots,n-1\}$.

\item A penalized version of Nason's 2-fold cross-validation (pen2F). As for the 2FCV, we compute $\widehat{m}_{\mathrm{pen2F}}$ by minimizing \eqref{def_crit_penVF}  with $V=2$, $B_1=\{2,4,\ldots,n\}$ and $B_2=\{1,3,\ldots,n-1\}$.
\end{itemize}
For each method, the model collection described in Section~\ref{section_Haar} is constructed by adding successively whole resolution levels of wavelet coefficients. Thus, the considered dimensions are $\{D_m, m\in\mathcal{M}_{n}\}=\{2^j, j=1,\ldots,\log_2(n)-1\}$. Note that unlike the local behaviours of the nonlinear models (e.g. thresholding), these linear models operate in a global fashion since entire scale levels of coefficients are suppressed (see Figures~\ref{fig:singleWave}(g),\ref{fig:singleSpikes}(g) for an illustration).

Typical estimations from a single simulation with $n=4096$ are depicted in \ref{fig:singleSpikes}(e)  for the \textit{Spikes} function. Figure~\ref{fig:singleSpikes}(f) also contains a plot of the excess risk $\ell(s_\ast, \widehat{s}_m)$ against the dimension $D_m$ and a vertical shift of the curve $\mathrm{crit}_{\mathrm{SH}}(m)$ is also overlayed for visualization purposes.  It can be observed that $\mathrm{crit}_{\mathrm{SH}}(m)$ gives a very reliable estimate for the risk $\ell(s_\ast, \widehat{s}_m)$, and in turn, also a high-quality estimate of the optimal model.  Indeed, for all cases, SH consistently selects the best model.

\subsection{Model selection performances} 

We compared the procedures on $N=1000$ independent data sets of size $n$ ranging from $256$ to $4096$. As in Arlot \cite{Arl:2008a}, we estimate the quality of the model-selection strategies through the following constant
\[
C_{\mathrm{or}} = \mathbb{E}\left[\frac{\left\Vert \widehat{s}_{\widehat{m}}-s_{\ast}\right\Vert _{2}^{2}}{\inf_{m\in\mathcal{M}_n}\left\Vert \widehat{s}_{m_\ast}-s_{\ast}\right\Vert _{2}^{2}}\right]
\]
which represents the constant that would appear in front of an oracle inequality. This ratio, which is greater than $1$, represents the accuracy of the model selection procedure. The average $C_{\mathrm{or}}$ over $1000$ replications are given in Tables~\ref{tab:low} and ~\ref{tab:high}.
 
 \subsection{Results and discussion}
 
It can be seen from Tables~\ref{tab:low} and ~\ref{tab:high} that none of the methods clearly outperforms the others in all cases. However, in our experiments, Mallows' $C_{p}$ seems to perform slightly better in many situations, both in the low and high noise regimes and for either homoscedastic and heteroscedastic noise. Also, the slope heuristics has roughly comparable results with Mallows' $C_{p}$, except for the small sample size case $n=256$, where Mallows' $C_{p}$
performs better, especially in the low noise regime. The quite bad behavior of the slope heuristics in the latter case (low noise, small sample size) can be explained by the fact that in such situation, the oracle model is the greatest model, that the slope heuristics tries to avoid through the use of the dimension jump. 

In the low noise regime (Table~\ref{tab:low}), $2$-fold penalization is slightly better
than $2$-fold cross-validation, especially when the sample size is small ($n=256$). Moreover, $2$-fold penalization is competitive with Mallows' $C_{p}$ in the low noise regime. When the noise is high (Table~\ref{tab:high}), $2$FCV and pen$2$F give roughly equivalent results.

Finally, it seems surprising that Mallows' $C_{p}$ and the slope heuristics,
that are based on linear penalties, outperform cross-validation methods in
the heteroscedastic noise case. Indeed linear penalties are proved to be
asymptotically suboptimal in such case, see Arlot \cite{Arl:2010}, while we
proved in Theorem~\ref{Theorem_VFpen} that $V$-fold penalization for a fixed $V$ is
asymptotically optimal. However, in order to be able to use Mallat's
algorithm for the discrete wavelet transform, we restricted ourselves to the 
$2$-fold and this could be the reason for the rather mild performances of
the cross-validation techniques compared to Mallows' $C_{p}$. Indeed, it is
well-known that in general, it is better to take $V=5$ or $10$ instead of $2$
(see for instance \cite{ArlotCelisse:10}), because it reduces the variance
of the cross-validation criterion. Also, Nason's cross-validation for
wavelet models allows to use Mallat's algorithm, but at the price of an
approximation of the original cross-validation criterion. These two aspects
might be at the origin of the superiority of Mallows' $C_{p}$ over the
cross-validation techniques, at least in the heteroscedastic case.

\begin{table}\centering
\begin{tabular}{|C|C|C||CCCC|}
\hline
s_\ast&\sigma_{\cdot}&n&\textbf{SH}&\textbf{Cp}&\textbf{2FCV}&\textbf{pen2F}\\\hline
\multirow{6}{*}{\textit{Wave}}&\multirow{3}{*}{$l1$}
&256&1.980\pm 0.011&1.106\pm 0.008&1.406\pm 0.019&\cellcolor{lgray}1.034\pm 0.005\\
&&1024&1.051\pm 0.002&\cellcolor{lgray}1.031\pm 0.002&1.062\pm 0.002&1.056\pm 0.004\\
&&4096&\cellcolor{lgray}1.021\pm 0.001&\cellcolor{lgray}1.021\pm 0.001&1.055\pm 0.002&\cellcolor{lgray}1.021\pm 0.001\\
&\multirow{3}{*}{$l2$}
&256&1.799\pm 0.009&1.140\pm 0.008&1.341\pm 0.015&\cellcolor{lgray}1.042\pm 0.005\\
&&1024&\cellcolor{lgray}1.021\pm 0.002&1.027\pm 0.002&1.029\pm 0.003&1.084\pm 0.006\\
&&4096&1.033\pm 0.002&1.032\pm 0.002&\cellcolor{lgray}1.015\pm 0.001&1.039\pm 0.002\\\hline
\multirow{6}{*}{\textit{HeaviSine}}&\multirow{3}{*}{$l1$}
&256&1.482\pm 0.014&1.157\pm 0.005&1.437\pm 0.016&\cellcolor{lgray}1.084\pm 0.006\\
&&1024&1.065\pm 0.003&\cellcolor{lgray}1.023\pm 0.002&1.155\pm 0.006&1.062\pm 0.004\\
&&4096&1.011\pm 0.001&\cellcolor{lgray}1.008\pm 0.001&1.101\pm 0.004&\cellcolor{lgray}1.010\pm 0.001\\
&\multirow{3}{*}{$l2$}
&256&1.357\pm 0.012&1.122\pm 0.005&1.357\pm 0.013&\cellcolor{lgray}1.063\pm 0.004\\
&&1024&1.048\pm 0.003&\cellcolor{lgray}1.032\pm 0.002&1.133\pm 0.006&1.093\pm 0.006\\
&&4096&1.016\pm 0.001&\cellcolor{lgray}1.013\pm 0.001&1.064\pm 0.003&1.020\pm 0.001\\\hline
\multirow{6}{*}{\textit{Doppler}}&\multirow{3}{*}{$l1$}
&256&2.890\pm 0.039&1.106\pm 0.008&1.852\pm 0.038&\cellcolor{lgray}1.072\pm 0.008\\
&&1024&2.091\pm 0.015&1.064\pm 0.006&1.486\pm 0.022&\cellcolor{lgray}1.013\pm 0.003\\
&&4096&1.010\pm 0.001&\cellcolor{lgray}1.000\pm 0.000&1.141\pm 0.007&1.025\pm 0.003\\
&\multirow{3}{*}{$l2$}
&256&2.820\pm 0.040&1.127\pm 0.009&1.784\pm 0.036&\cellcolor{lgray}1.059\pm 0.006\\
&&1024&1.874\pm 0.013&1.078\pm 0.006&1.419\pm 0.016&\cellcolor{lgray}1.009\pm 0.002\\
&&4096&1.024\pm 0.002&\cellcolor{lgray}1.002\pm 0.000&1.187\pm 0.006&1.019\pm 0.003\\\hline
\multirow{6}{*}{\textit{Spikes}}&\multirow{3}{*}{$l1$}
&256&3.541\pm 0.071&1.092\pm 0.007&2.075\pm 0.062&\cellcolor{lgray}1.062\pm 0.010\\
&&1024&1.077\pm 0.006&\cellcolor{lgray}1.021\pm 0.002&1.198\pm 0.012&1.045\pm 0.003\\
&&4096&1\cellcolor{lgray}.008\pm 0.001&\cellcolor{lgray}1.008\pm 0.001&1.029\pm 0.002&1.014\pm 0.001\\
&\multirow{3}{*}{$l2$}
&256&3.236\pm 0.058&1.087\pm 0.007&2.008\pm 0.055&\cellcolor{lgray}1.071\pm 0.011\\
&&1024&1.054\pm 0.004&\cellcolor{lgray}1.013\pm 0.001&1.187\pm 0.012&1.069\pm 0.004\\
&&4096&\cellcolor{lgray}1.007\pm 0.001&\cellcolor{lgray}1.007\pm 0.001&\cellcolor{lgray}1.009\pm 0.001&1.019\pm 0.002\\\hline
\end{tabular}
\caption{Comparison of mean performance $C_{\mathrm{or}}$ for each procedure over $N=1000$ realizations of the low noise level setting with corresponding empirical standard deviation divided by $\sqrt{N}$.}
\label{tab:low}
\end{table}

\begin{table}\centering
\begin{tabular}{|C|C|C||CCCC|}
\hline
s_\ast&\sigma_{\cdot}&n&\textbf{SH}&\textbf{Cp}&\textbf{2FCV}&\textbf{pen2F}\\\hline
\multirow{6}{*}{\textit{Wave}}&\multirow{3}{*}{$h1$}
&256&1.029\pm 0.004&\cellcolor{lgray}1.016\pm 0.003&1.236\pm 0.011&1.158\pm 0.009\\
&&1024&\cellcolor{lgray}1.003\pm 0.001&\cellcolor{lgray}1.002\pm 0.001&\cellcolor{lgray}1.002\pm 0.001&1.033\pm 0.005\\
&&4096&1.011\pm 0.002&1.008\pm 0.002&\cellcolor{lgray}1.000\pm 0.000&1.040\pm 0.004\\
&\multirow{3}{*}{$h2$}
&256&1.076\pm 0.006&\cellcolor{lgray}1.052\pm 0.006&1.252\pm 0.010&1.244\pm 0.012\\
&&1024&1.022\pm 0.005&1.014\pm 0.004&\cellcolor{lgray}1.004\pm 0.002&1.072\pm 0.008\\
&&4096&1.020\pm 0.004&1.019\pm 0.004&\cellcolor{lgray}1.006\pm 0.002&1.067\pm 0.007\\\hline
\multirow{6}{*}{\textit{HeaviSine}}&\multirow{3}{*}{$h1$}
&256&\cellcolor{lgray}1.096\pm 0.005&\cellcolor{lgray}1.090\pm 0.005&1.115\pm 0.006&1.185\pm 0.013\\
&&1024&\cellcolor{lgray}1.057\pm 0.003&\cellcolor{lgray}1.054\pm 0.003&1.123\pm 0.006&1.075\pm 0.004\\
&&4096&\cellcolor{lgray}1.029\pm 0.002&\cellcolor{lgray}1.028\pm 0.002&1.081\pm 0.004&1.041\pm 0.003\\
&\multirow{3}{*}{$h2$}
&256&1.155\pm 0.009&1.153\pm 0.011&\cellcolor{lgray}1.125\pm 0.008&1.300\pm 0.020\\
&&1024&\cellcolor{lgray}1.101\pm 0.006&\cellcolor{lgray}1.091\pm 0.006&1.133\pm 0.007&1.159\pm 0.010\\
&&4096&\cellcolor{lgray}1.047\pm 0.003&\cellcolor{lgray}1.046\pm 0.003&1.122\pm 0.006&1.083\pm 0.005\\\hline
\multirow{6}{*}{\textit{Doppler}}&\multirow{3}{*}{$h1$}
&256&1.330\pm 0.011&1.107\pm 0.005&1.347\pm 0.013&\cellcolor{lgray}1.043\pm 0.003\\
&&1024&1.054\pm 0.003&\cellcolor{lgray}1.025\pm 0.002&1.108\pm 0.005&1.067\pm 0.005\\
&&4096&\cellcolor{lgray}1.013\pm 0.001&\cellcolor{lgray}1.014\pm 0.001&1.029\pm 0.002&1.021\pm 0.001\\
&\multirow{3}{*}{$h2$}
&256&1.224\pm 0.010&1.076\pm 0.004&1.291\pm 0.011&\cellcolor{lgray}1.053\pm 0.003\\
&&1024&\cellcolor{lgray}1.035\pm 0.002&\cellcolor{lgray}1.031\pm 0.002&1.079\pm 0.004&1.098\pm 0.007\\
&&4096&\cellcolor{lgray}1.010\pm 0.001&\cellcolor{lgray}1.009\pm 0.001&1.022\pm 0.003&1.023\pm 0.002\\\hline
\multirow{6}{*}{\textit{Spikes}}&\multirow{3}{*}{$h1$}
&256&1.156\pm 0.009&\cellcolor{lgray}1.047\pm 0.003&1.282\pm 0.014&1.076\pm 0.005\\
&&1024&\cellcolor{lgray}1.006\pm 0.001&\cellcolor{lgray}1.005\pm 0.001&1.094\pm 0.007&1.029\pm 0.004\\
&&4096&\cellcolor{lgray}1.012\pm 0.002&\cellcolor{lgray}1.010\pm 0.001&\cellcolor{lgray}1.009\pm 0.002&1.021\pm 0.002\\
&\multirow{3}{*}{$h2$}
&256&1.119\pm 0.008&\cellcolor{lgray}1.052\pm 0.004&1.284\pm 0.014&1.126\pm 0.006\\
&&1024&\cellcolor{lgray}1.015\pm 0.002&\cellcolor{lgray}1.014\pm 0.002&1.137\pm 0.008&1.059\pm 0.008\\
&&4096&\cellcolor{lgray}1.015\pm 0.002&\cellcolor{lgray}1.011\pm 0.002&\cellcolor{lgray}1.014\pm 0.003&1.030\pm 0.004\\\hline
\end{tabular}
\caption{Comparison of mean performance $C_{\mathrm{or}}$ for each procedure over $N=1000$ realizations of the high noise level setting with corresponding empirical standard deviation divided by $\sqrt{N}$.}
\label{tab:high}
\end{table}

\section{Proofs\label{section_proof_slope_reg}}

\subsection{Proofs related to Section \protect\ref{section_strong_loc_bas} 
\label{ssection_proofs_loc_bas}}

\begin{proof}[Proof of Proposition \protect\ref{prop_strong_loc_loc}]
The proof simply follows from the following computations. For every $\beta
=\left( \beta _{k}\right) _{k=1}^{D_m}\in \mathbb{R}^{D_m}$,%
\begin{eqnarray*}
\left\Vert \sum_{k=1}^{D_m}\beta _{k}\varphi _{k}\right\Vert _{\infty } &\leq
&\sum_{i=1}^{b_m}\left\Vert \sum_{l\in \Pi _{i}}\beta _{l}\varphi
_{l}\right\Vert _{\infty } \\
&\leq &\sum_{i=1}^{b_m}A_{c}\max_{l\in \Pi _{i}}\left\Vert \varphi
_{l}\right\Vert _{\infty } \times \max_{l\in \Pi _{i}}\left\vert \beta _{l}\right\vert
\\
&\leq &A_{c}r_{m}\sum_{i=1}^{b_m}\sqrt{A_{i}}\max_{l\in \Pi _{i}}\left\vert
\beta _{l}\right\vert \\
&\leq &A_{c}r_{m}^{2}\sqrt{D_m}\max_{k\in \left\{ 1,\ldots,D_m\right\} }\left\vert
\beta _{k}\right\vert .
\end{eqnarray*}
\end{proof}

\begin{proof}[Proof of Proposition \protect\ref{prop_wav}]
The fact that $\left\{ \psi _{\lambda }^{\text{per}}\text{ };\text{ }\lambda
\in \Lambda _{b_m}\right\} $ is an orthonormal family - and thus an
orthonormal basis of $m$ - is a classical fact of wavelet theory (see for
instance \cite{CohenDaubVial:93}). Take $m>0$ such that 
\begin{equation*}
\supp\left( \psi _{0}\right) \bigcup \supp\left(\phi
_{0}\right) \subset \left[0,m\right] .
\end{equation*}%
For $j\geq 0$ and $1\leq k\leq 2^{j}$, we have%
\begin{equation*}
\left\Vert \psi _{j,k}^{\text{per}}\right\Vert _{\infty }\leq \left( \left[ m%
\right] +2\right) \left\Vert \psi _{j,k}\right\Vert _{\infty }\leq \left( %
\left[ m\right] +2\right) 2^{j/2}\left\Vert \psi _{0}\right\Vert _{\infty }%
,
\end{equation*}%
where $\left[ m\right] $ is the integer part of $m$. We thus take $%
A_{j}=2^{j}$ for $j\geq 0$ and $A_{-1}=1$, which gives%
\begin{equation*}
\sum_{i=-1}^{b_m}\sqrt{A_{i}}\leq \left( 1+\sqrt{2}\right) \sqrt{D_m},
\end{equation*}%
since $D_m=2^{b_m+1}$. By taking $r_{m}=\max \left\{ \left( \left[ m\right]
+2\right) \left\Vert \psi _{0}\right\Vert _{\infty },1+\sqrt{2}\right\} $,
we thus get, for any $j\geq -1$ and $k\in \left\{ 1,\ldots,2^{j}\right\} $,%
\begin{equation*}
\left\Vert \psi _{j,k}^{\text{per}}\right\Vert _{\infty }\leq r_{m}\sqrt{%
A_{j}}\text{ \ \ and \ \ }\sum_{i=-1}^{b_m}\sqrt{A_{i}}\leq r_{m}\sqrt{D_m}\text{
.}
\end{equation*}%
It remains to prove that there exists $A_{c}>0$ such that, by denoting for $%
\mu \in \Lambda _{b_m}$ and $j\in \left\{ -1,0,1,\ldots,m\right\} $, 
\begin{equation*}
\Lambda _{j\left\vert \mu \right. }=\left\{ \lambda \in \Lambda \left(
j\right) \text{ };\supp\left( \psi _{\mu }\right) \bigcap \supp\left( \psi _{\lambda }\right) \neq \emptyset \right\} ,
\end{equation*}%
one has 
\begin{equation}
\max_{\mu \in \Lambda \left( i\right) }%
%TCIMACRO{\TeXButton{Card}{\card}}%
%BeginExpansion
\card%
%EndExpansion
\left( \Lambda _{j\left\vert \mu \right. }\right) \leq A_{c}\left(
A_{j}A_{i}^{-1}\vee 1\right) .  
\label{slgoal}
\end{equation}%
Take $j_{0}=\max \left\{ \left[ \log _{2}\left( m\right) \right]
+1,0\right\} $. Then for all $j\geq j_{0}$ and $k\in \left\{
1,\ldots,2^{j-j_{0}}\right\}$, $\supp\left( \psi _{j,k}\right) \subset \left[
0,1\right) $. Furthermore, for every $k\in \left\{ 1,\ldots,2^{j-j_{0}}\right\} 
$ set $\Gamma \left( k\right) =\left\{ 2^{j-j_{0}}l+k;l\in \left\{
0,\ldots,2^{j_{0}}-1\right\} \right\} $. Then $\left\{ \Gamma \left( k\right)
;k\in \left\{ 1,\ldots,2^{j-j_{0}}\right\} \right\} $ form a partition of $%
\left\{ 1,\ldots,2^{j}\right\} $ and for $k,k^{\prime }\in \left\{
1,\ldots,2^{j-j_{0}}\right\} $, $k\neq k^{\prime }$,%
\begin{equation*}
\supp\left( \psi _{j,k}\right) \bigcap \supp\left(\psi
_{j,k^{\prime }}\right) =\emptyset .
\end{equation*}%
It is then easy to see that taking $A_{c}=2^{j_{0}}$ gives (\ref{slgoal}).
\end{proof}

\subsection{Proofs related to the slope heuristics\label%
{ssection_proof_slope}}

We first notice that, from \cite{saum:13}, Section 5, Theorems \ref%
{theorem_min_pen_reg_pp}, \ref{theorem_opt_pen_reg_pp} are valid under the
following general set of assumptions (i.e. by replacing (\textbf{SA}) by 
\textbf{(GSA) }in the statement\textbf{\ }of the theorems):

\smallskip

\noindent \textbf{General set of assumptions: (GSA)}

\smallskip

Assume (\textbf{P1}), (\textbf{P2}), (\textbf{P3}), (\textbf{Ab}), (\textbf{%
An}) and (\textbf{Ap}$_{u}$) of (\textbf{SA}). Furthermore suppose that,

\begin{description}
\item[(\textbf{Alb})] there exists a constant $r_{\mathcal{M}}$ such that
for each $m\in \mathcal{M}_{n}$ one can find an orthonormal basis $\left(
\varphi _{k}\right) _{k=1}^{D_{m}}$ satisfying, for all $\left( \beta
_{k}\right) _{k=1}^{D_{m}}\in \mathbb{R}^{D_{m}},$%
\begin{equation*}
\left\Vert \sum_{k=1}^{D_{m}}\beta _{k}\varphi _{k}\right\Vert _{\infty
}\leq r_{\mathcal{M}}\sqrt{D_{m}}\left\vert \beta \right\vert _{\infty }%
\text{ },  
%\label{loc}
\end{equation*}%
where $\left\vert \beta \right\vert _{\infty }=\max \left\{ \left\vert \beta
_{k}\right\vert ;k\in \left\{ 1,\ldots,D_{m}\right\} \right\} $.

\item[(Ac$_{\infty }$)] a positive integer $n_{1}$ exists such that, for all 
$n\geq n_{1}$, there exist a positive constant $A_{cons}$ and an event $%
\Omega _{\infty }$ of probability at least $1-n^{-2-\alpha _{\mathcal{M}}}$,
on which for all $m\in \mathcal{M}_{n}$,%
\begin{equation*}
\left\Vert \widehat{s}_{m}-s_{m}\right\Vert _{\infty }\leq A_{cons}\sqrt{\frac{%
D_{m}\ln n}{n}}\text{ }.  
%\label{def_R_n_D_selection}
\end{equation*}
\end{description}

\smallskip

Now the proofs of Theorems \ref{theorem_min_pen_reg_pp} and \ref%
{theorem_opt_pen_reg_pp} simply rely on the fact that assumptions (\textbf{%
Alb}) and (\textbf{Ac}$_{\infty }$) in (\textbf{GSA}) are ensured under (%
\textbf{SA}). Indeed, assumption (\textbf{Alb}) in (\textbf{GSA}) is
satisfied under assumption (\textbf{Auslb}) in the set of assumptions (%
\textbf{SA}), see Proposition \ref{prop_strong_loc_loc}. Furthermore,
Theorem \ref{theorem_general_sup} shows that assumption (\textbf{Ac}$%
_{\infty }$) in (\textbf{GSA}) is also satisfied under assumption (\textbf{%
Auslb}).

\subsection{Proofs related to V-fold procedures} \label{ssection_proof_VF}

\subsubsection{Proofs related to V-fold cross-validation} %\label{sssection_proof_VFCV}

Theorem \ref{Theorem_VFCV} is a straightforward consequence of the following
result, that will be proved below. Recall that the set of assumptions (%
\textbf{GSA}) is defined in Section \ref{ssection_proof_slope} above.

\begin{thrm}
\label{theorem_VFCV_gene}Assume that (\textbf{GSA}) holds. Let $r\in \left(
2,+\infty \right) $ and $V\in \left\{ 2,\ldots,n-1\right\} $ satisfying $%
1<V\leq r$. Define the VFCV procedure as the model
selection procedure given by (\ref{def_VFCV}) and (\ref{def_crit_VFCV}). Then, for all $n\geq n_{0}\left( \left( \text{\textbf{GSA}}\right) ,r\right) 
$, with probability at least $1-L_{\left( \text{\textbf{GSA}}\right)
,r}n^{-2}$,%
\begin{equation*}
\ell \left( s_{\ast },\widehat{s}_{\widehat{m}}^{\left( -1\right) }\right)
\leq \left( 1+\frac{L_{\left( \text{\textbf{GSA}}\right) ,r}}{\sqrt{\ln n}}%
\right) \inf_{m\in \mathcal{M}_{n}}\left\{ \ell \left( s_{\ast },\widehat{s}%
_{m}^{\left( -1\right) }\right) \right\} +L_{\left( \text{\textbf{GSA}}%
\right) ,r}\frac{\left( \ln n\right) ^{3}}{n}.
%\label{oracle_opt_gene_n1}
\end{equation*}
\end{thrm}

\noindent \textbf{Proof of Theorem \ref{theorem_VFCV_gene}. }All along the%
\textbf{\ }proof, the value of the constant $L_{\text{(\textbf{GSA})},r}$ may vary from line to line.
We set 
\begin{equation*}
%TCIMACRO{\TeXButton{crit}{\crit}}%
%BeginExpansion
\crit%
%EndExpansion
_{%
%TCIMACRO{\TeXButton{VFCV}{\VFCV}}%
%BeginExpansion
\VFCV%
%EndExpansion
}^{0}\left( m\right) =%
%TCIMACRO{\TeXButton{crit}{\crit}}%
%BeginExpansion
\crit%
%EndExpansion
_{%
%TCIMACRO{\TeXButton{VFCV}{\VFCV}}%
%BeginExpansion
\VFCV%
%EndExpansion
}\left( m\right) -\frac{1}{V}\sum_{j=1}^{V}P_{n}^{(j)}\left( \gamma \left(
s_{\ast }\right) \right) .
\end{equation*}%
It is worth noting that the difference between $%
%TCIMACRO{\TeXButton{crit}{\crit}}%
%BeginExpansion
\crit%
%EndExpansion
_{%
%TCIMACRO{\TeXButton{VFCV}{\VFCV}}%
%BeginExpansion
\VFCV%
%EndExpansion
}^{0}\left( m\right) $ and $%
%TCIMACRO{\TeXButton{crit}{\crit}}%
%BeginExpansion
\crit%
%EndExpansion
_{%
%TCIMACRO{\TeXButton{VFCV}{\VFCV}}%
%BeginExpansion
\VFCV%
%EndExpansion
}\left( m\right) $ is a quantity independent of $m$, when $m$ varies in $%
\mathcal{M}_{n}$. Hence, the procedure defined by $%
%TCIMACRO{\TeXButton{crit}{\crit}}%
%BeginExpansion
\crit%
%EndExpansion
_{%
%TCIMACRO{\TeXButton{VFCV}{\VFCV}}%
%BeginExpansion
\VFCV%
%EndExpansion
}^{0}$ gives the same result as the VFCV procedure
defined by $%
%TCIMACRO{\TeXButton{crit}{\crit}}%
%BeginExpansion
\crit%
%EndExpansion
_{%
%TCIMACRO{\TeXButton{VFCV}{\VFCV}}%
%BeginExpansion
\VFCV%
%EndExpansion
}$. It will be convenient for our analysis to consider $%
%TCIMACRO{\TeXButton{crit}{\crit}}%
%BeginExpansion
\crit%
%EndExpansion
_{%
%TCIMACRO{\TeXButton{VFCV}{\VFCV}}%
%BeginExpansion
\VFCV%
%EndExpansion
}^{0}$ instead of $%
%TCIMACRO{\TeXButton{crit}{\crit}}%
%BeginExpansion
\crit%
%EndExpansion
_{%
%TCIMACRO{\TeXButton{VFCV}{\VFCV}}%
%BeginExpansion
\VFCV%
%EndExpansion
}$.

We get for all $m\in \mathcal{M}_{n}$,%
\begin{eqnarray}
%TCIMACRO{\TeXButton{crit}{\crit}}%
%BeginExpansion
\crit%
%EndExpansion
_{%
%TCIMACRO{\TeXButton{VFCV}{\VFCV}}%
%BeginExpansion
\VFCV%
%EndExpansion
}^{0}\left( m\right) &=&\frac{1}{V}\sum_{j=1}^{V}P_{n}^{\left( j\right)
}\left( \gamma \left( \widehat{s}_{m}^{\left( -j\right) }\right) -\gamma
\left( s_{\ast }\right) \right)  \notag \\
&=&\frac{1}{V}\sum_{j=1}^{V}\left[ P_{n}^{\left( j\right) }\left( \gamma
\left( \widehat{s}_{m}^{\left( -j\right) }\right) -\gamma \left(
s_{m}\right) \right) \right.  \notag \\
&&+\left. \left( P_{n}^{\left( j\right) }-P\right) \left( \gamma \left(
s_{m}\right) -\gamma \left( s_{\ast }\right) \right) +P\left( \gamma \left(
s_{m}\right) -\gamma \left( s_{\ast }\right) \right) \right]  \notag \\
&=&\ell \left( s_{\ast}, \widehat{s}_{m}^{\left( -1\right) }\right) +\Delta
_{V}\left( m\right) +\bar{\delta}\left( m\right)  \label{crit_0_VF}
\end{eqnarray}%
where%
\begin{equation*}
\Delta _{V}\left( m\right) =\frac{1}{V}\sum_{j=1}^{V}P_{n}^{\left( j\right)
}\left( \gamma \left( \widehat{s}_{m}^{\left( -j\right) }\right) -\gamma
\left( s_{m}\right) \right) -P\left( \gamma \left( \widehat{s}_{m}^{\left(
-1\right) }\right) -\gamma \left( s_{m}\right) \right),
\end{equation*}
and $\bar{\delta}\left( m\right)$ has been defined in Lemma \ref{lemma_delta_VFCV}. Furthermore denote 
\begin{equation*}
%TCIMACRO{\TeXButton{p}{\p}}%
%BeginExpansion
\p%
%EndExpansion
_{1}^{\left( -1\right) }\left( m\right) =P\left( \gamma \left( \widehat{s}%
_{m}^{\left( -1\right) }\right) -\gamma \left( s_{m}\right) \right) \text{ \
\ and \ \ }%
%TCIMACRO{\TeXButton{p}{\p}}%
%BeginExpansion
\p%
%EndExpansion
_{2}^{\left( -1\right) }\left( m\right) =P_{n}^{\left( -1\right) }\left(
\gamma \left( s_{m}\right) -\gamma \left( \widehat{s}_{m}^{\left( -1\right)
}\right) \right) .
\end{equation*}%
Let $\Omega _{n}$ be the event on which:

\begin{itemize}
\item For all models $m\in \mathcal{M}_{n}$ of dimension $D_{m}$ such that $%
A_{\mathcal{M},+}\left( \ln n\right) ^{3}\leq D_{m}$, it holds 
\begin{align}
\left\vert 
%TCIMACRO{\TeXButton{p}{\p}}%
%BeginExpansion
\p%
%EndExpansion
_{1}^{\left( -1\right) }\left( m\right) -\mathbb{E}\left[ 
%TCIMACRO{\TeXButton{p}{\p}}%
%BeginExpansion
\p%
%EndExpansion
_{2}^{\left( -1\right) }\left( m\right) \right] \right\vert & \leq L_{\text{(%
\textbf{GSA})},r}\varepsilon _{n}\left( m\right) \mathbb{E}\left[ 
%TCIMACRO{\TeXButton{p}{\p}}%
%BeginExpansion
\p%
%EndExpansion
_{2}^{\left( -1\right) }\left( m\right) \right]  \label{line_p1_VF} \\
\left\vert 
%TCIMACRO{\TeXButton{p}{\p}}%
%BeginExpansion
\p%
%EndExpansion
_{2}^{\left( -1\right) }\left( m\right) -\mathbb{E}\left[ 
%TCIMACRO{\TeXButton{p}{\p}}%
%BeginExpansion
\p%
%EndExpansion
_{2}^{\left( -1\right) }\left( m\right) \right] \right\vert & \leq L_{\text{(%
\textbf{GSA})},r}\varepsilon _{n}^{2}\left( m\right) \mathbb{E}\left[ 
%TCIMACRO{\TeXButton{p}{\p}}%
%BeginExpansion
\p%
%EndExpansion
_{2}^{\left( -1\right) }\left( m\right) \right] \nonumber %\label{line_p2_VF}
\end{align}%
together with%
\begin{eqnarray}
\left\vert \Delta _{V}\left( m\right) \right\vert &\leq &L_{\text{(\textbf{%
GSA}),}r}\varepsilon _{n}\left( m\right) \mathbb{E}\left[ 
%TCIMACRO{\TeXButton{p}{\p}}%
%BeginExpansion
\p%
%EndExpansion
_{2}^{(-1)}\left( m\right) \right]  \label{line_Delta_VF} \\
\left\vert \bar{\delta}\left( m\right) \right\vert &\leq &\frac{\ell \left(
s_{\ast },s_{m}\right) }{\sqrt{D_{m}}}+L_{\text{(\textbf{GSA})},r}\frac{\ln n%
}{\sqrt{D_{m}}}\mathbb{E}\left[ 
%TCIMACRO{\TeXButton{p}{\p}}%
%BeginExpansion
\p%
%EndExpansion
_{2}^{\left( -1\right) }\left( m\right) \right]  \label{line_delta_bar_VF}
\end{eqnarray}

\item For all models $m\in \mathcal{M}_{n}$ of dimension $D_{m}$ such that $%
D_{m}\leq A_{\mathcal{M},+}\left( \ln n\right) ^{3}$, it holds%
\begin{eqnarray}
\left\vert \Delta _{V}\left( m\right) \right\vert &\leq &L_{\text{(\textbf{%
GSA})},r}\frac{\left( \ln n\right) ^{2}}{n}  \label{Delta_VF_small} \\
\left\vert \bar{\delta}\left( m\right) \right\vert &\leq &L_{(\mathbf{GSA}%
),r}\left( \sqrt{\frac{\ell \left( s_{\ast },s_{m}\right) \ln n}{n}}+\frac{%
\ln n}{n}\right) \nonumber \\%\label{delta_bar_VF_small} 
%TCIMACRO{\TeXButton{p}{\p}}%
%BeginExpansion
\p%
%EndExpansion
_{2}^{\left( -1\right) }\left( m\right) &\leq &L_{\text{(\textbf{GSA})},r}%
\frac{D_{m}\vee \ln n}{n}\leq L_{\text{(\textbf{GSA})},r}\frac{\left( \ln
n\right) ^{3}}{n}  \nonumber \\%\label{p2_VF_small}
%TCIMACRO{\TeXButton{p}{\p}}%
%BeginExpansion
\p%
%EndExpansion
_{1}^{\left( -1\right) }\left( m\right) &\leq &L_{\text{(\textbf{GSA})},r}%
\frac{D_{m}\vee \ln n}{n}\leq L_{\text{(\textbf{GSA})},r}\frac{\left( \ln
n\right) ^{3}}{n}  \label{p1_VF_small}
\end{eqnarray}
\end{itemize}

\noindent By Theorem 2 of \cite{saum:12} and Lemma 4 of \cite{saum:13}
applied with $\alpha =2+\alpha _{\mathcal{M}}$ and sample size $%
n_{V}=n\left( V-1\right) /V$, Corollary \ref{corollary_V_fold} and Lemma %
\ref{lemma_delta_VFCV} applied with $\alpha =2+\alpha _{\mathcal{M}}$, we
get for all $n\geq n_{0}\left( \text{(\textbf{GSA})},r\right) $, 
\begin{equation*}
\mathbb{P}\left( \Omega _{n}\right) \geq 1-L_{\text{(\textbf{GSA})}%
,r}\sum_{m\in \mathcal{M}_{n}}n^{-2-\alpha _{\mathcal{M}}}\geq 1-L_{\text{(%
\textbf{GSA})},r}n^{-2}\text{ }.
\end{equation*}

\smallskip

\noindent \noindent \underline{\textbf{Control on the criterion }$%
%TCIMACRO{\TeXButton{crit}{\crit}}%
%BeginExpansion
\crit_{\VFCV}^0$\textbf{\ for models of dimension not too small: }}

\smallskip

\noindent We consider models $m\in \mathcal{M}_{n}$ such that $A_{\mathcal{M}%
,+}\left( \ln n\right) ^{3}\leq D_{m}$.

\begin{eqnarray*}
%TCIMACRO{\TeXButton{crit}{\crit}}%
%BeginExpansion
\crit%
%EndExpansion
_{%
%TCIMACRO{\TeXButton{VFCV}{\VFCV}}%
%BeginExpansion
\VFCV%
%EndExpansion
}^{0}\left( m\right) &=&\frac{1}{V}\sum_{j=1}^{V}P_{n}^{\left( j\right)
}\left( \gamma \left( \widehat{s}_{m}^{\left( -j\right) }\right) -\gamma
\left( s_{\ast }\right) \right) \\
&=&\frac{1}{V}\sum_{j=1}^{V}\left[ P_{n}^{\left( j\right) }\left( \gamma
\left( \widehat{s}_{m}^{\left( -j\right) }\right) -\gamma \left(
s_{m}\right) \right) \right. \\
&&+\left. \left( P_{n}^{\left( j\right) }-P\right) \left( \gamma \left(
s_{m}\right) -\gamma \left( s_{\ast }\right) \right) +P\left( \gamma \left(
s_{m}\right) -\gamma \left( s_{\ast }\right) \right) \right] \\
&=&\ell \left(s_{\ast },\widehat{s}_{m}^{\left( -1\right) }\right) +\Delta
_{V}\left( m\right) +\bar{\delta}\left( m\right)
\end{eqnarray*}%
By (\ref{line_p1_VF}), (\ref{line_Delta_VF}) and (\ref{line_delta_bar_VF})
we have on $\Omega _{n}$, 
\begin{eqnarray*}
\max \left\{ \left\vert \Delta _{V}\left( m\right) \right\vert ,\left\vert 
\bar{\delta}\left( m\right) \right\vert \right\} &\leq &L_{\text{(\textbf{GSA%
}),}r}\varepsilon _{n}\left( m\right) \left( \ell \left( s_{\ast
},s_{m}\right) +\mathbb{E}\left[ 
%TCIMACRO{\TeXButton{p}{\p}}%
%BeginExpansion
\p%
%EndExpansion
_{2}^{(-1)}\left( m\right) \right] \right) \\
&\leq &L_{\text{(\textbf{GSA}),}r}\varepsilon _{n}\left( m\right) \ell \left(s_{\ast },\widehat{s}_{m}^{\left( -1\right) }\right) .
\end{eqnarray*}%
Hence, identity (\ref{crit_0_VF}) gives%
\begin{equation}
\left\vert 
%TCIMACRO{\TeXButton{crit}{\crit}}%
%BeginExpansion
\crit%
%EndExpansion
_{%
%TCIMACRO{\TeXButton{VFCV}{\VFCV}}%
%BeginExpansion
\VFCV%
%EndExpansion
}^{0}\left( m\right) -\ell \left(s_{\ast },\widehat{s}_{m}^{\left( -1\right) }\right) \right\vert \leq L_{\text{(\textbf{GSA}),}r}\varepsilon
_{n}\left( m\right) \ell \left(s_{\ast },\widehat{s}_{m}^{\left( -1\right) }\right) .  \label{control_crit_VF_large}
\end{equation}

\smallskip

\noindent \underline{\textbf{Control on the criterion }$\crit_{\VFCV}^0$\textbf{\ for models of small dimension: }}

\smallskip

\noindent We consider models $m\in \mathcal{M}_{n}$ such that $D_{m}\leq A_{%
\mathcal{M},+}\left( \ln n\right) ^{3}$. By (\ref{Delta_VF_small}), (\ref%
{delta_bar_small}) and (\ref{p1_VF_small}), it holds on $\Omega _{n}$, for
any $\tau >0$ and for all $m\in \mathcal{M}_{n}$ such that $D_{m}\leq A_{%
\mathcal{M},+}\left( \ln n\right) ^{3}$,

\begin{align}
& \left\vert 
%TCIMACRO{\TeXButton{crit}{\crit}}%
%BeginExpansion
\crit%
%EndExpansion
_{%
%TCIMACRO{\TeXButton{VFCV}{\VFCV}}%
%BeginExpansion
\VFCV%
%EndExpansion
}^{0}\left( m\right) -\ell \left( s_{\ast },\widehat{s}_{m}^{\left( -1\right) }\right) \right\vert  \notag \\
& \leq L_{\text{(\textbf{GSA})},r}\frac{\left( \ln n\right) ^{2}}{n}+L_{%
\text{(\textbf{GSA})},r}\left( \sqrt{\frac{\ell \left( s_{\ast
},s_{m}\right) \ln n}{n}}+\frac{\ln n}{n}\right)  \notag \\
& \leq L_{\text{(\textbf{GSA})},r}\frac{\left( \ln n\right) ^{2}}{n}+L_{%
\text{(\textbf{GSA})},r}\tau \ell \left( s_{\ast },s_{m}\right) +\left( \tau
^{-1}+1\right) L_{\text{(\textbf{GSA})}}\frac{\ln n}{n}.  \notag
\end{align}%
Hence, by taking $\tau =\left( \ln n\right) ^{-2}$ in the last display we
get,%
\begin{equation}
\left\vert 
%TCIMACRO{\TeXButton{crit}{\crit}}%
%BeginExpansion
\crit%
%EndExpansion
_{%
%TCIMACRO{\TeXButton{VFCV}{\VFCV}}%
%BeginExpansion
\VFCV%
%EndExpansion
}^{0}\left( m\right) -\ell \left( s_{\ast },\widehat{s}_{m}^{\left( -1\right) }\right) \right\vert \leq L_{\text{(\textbf{GSA}),}r}\left( \frac{%
\ell \left( s_{\ast },\widehat{s}_{m}^{\left( -1\right) }\right) }{\left(
\ln n\right) ^{2}}+\frac{\left( \ln n\right) ^{3}}{n}\right) .
\label{majo_small_2}
\end{equation}

\smallskip

\noindent \underline{\textbf{Oracle inequalities: }}

\smallskip

\noindent We exploit the following inequality, that defines the selected
model $\widehat{m}$,%
\begin{equation}
%TCIMACRO{\TeXButton{crit}{\crit}}%
%BeginExpansion
\crit%
%EndExpansion
_{%
%TCIMACRO{\TeXButton{VFCV}{\VFCV}}%
%BeginExpansion
\VFCV%
%EndExpansion
}^{0}\left( \widehat{m}\right) \leq \inf_{m\in \mathcal{M}_{n}}\left\{ 
%TCIMACRO{\TeXButton{crit}{\crit}}%
%BeginExpansion
\crit%
%EndExpansion
_{%
%TCIMACRO{\TeXButton{VFCV}{\VFCV}}%
%BeginExpansion
\VFCV%
%EndExpansion
}^{0}\left( m\right) \right\} .  \label{def_m_hat_crit_0_VF}
\end{equation}%
Indeed, using (\ref{control_crit_VF_large}) and (\ref{majo_small_2}), we get
that on $\Omega _{n}$ it holds,%
\begin{eqnarray}
&&%
%TCIMACRO{\TeXButton{crit}{\crit}}%
%BeginExpansion
\crit%
%EndExpansion
_{%
%TCIMACRO{\TeXButton{VFCV}{\VFCV}}%
%BeginExpansion
\VFCV%
%EndExpansion
}^{0}\left( \widehat{m}\right)  \notag \\
&\geq &\left( 1-L_{\text{(\textbf{GSA}),}r}\left[ \frac{1}{\left( \ln
n\right) ^{2}}-\sup_{m:D_{m}\geq A_{\mathcal{M},+}\left( \ln n\right)
^{3}}\varepsilon _{n}\left( m\right) \right] \right) \ell \left( s_{\ast },%
\widehat{s}_{\widehat{m}}^{\left( -1\right) }\right) -L_{\text{(\textbf{GSA}%
),}r}\frac{\left( \ln n\right) ^{3}}{n}  \notag \\
&\geq &\left( 1-\frac{L_{\text{(\textbf{GSA}),}r}}{\sqrt{\ln n}}\right) \ell
\left( s_{\ast },\widehat{s}_{\widehat{m}}^{\left( -1\right) }\right) -L_{%
\text{(\textbf{GSA}),}r}\frac{\left( \ln n\right) ^{3}}{n}.
\label{lower_crit_VF}
\end{eqnarray}%
Furthermore, using again (\ref{control_crit_VF_large}) and (\ref%
{majo_small_2}), we get%
\begin{eqnarray}
&&\inf_{m\in \mathcal{M}_{n}}\left\{ 
%TCIMACRO{\TeXButton{crit}{\crit}}%
%BeginExpansion
\crit%
%EndExpansion
_{%
%TCIMACRO{\TeXButton{VFCV}{\VFCV}}%
%BeginExpansion
\VFCV%
%EndExpansion
}^{0}\left( m\right) \right\}  \notag \\
&\leq &\left( 1+\frac{L_{\text{(\textbf{GSA}),}r}}{\sqrt{\ln n}}\right)
\inf_{m\in \mathcal{M}_{n}}\left\{ \ell \left( s_{\ast },\widehat{s}%
_{m}^{\left( -1\right) }\right) \right\} +L_{\text{(\textbf{GSA}),}r}\frac{%
\left( \ln n\right) ^{3}}{n}.  \label{majo_crit_VF}
\end{eqnarray}%
Putting (\ref{lower_crit_VF}) and (\ref{majo_crit_VF}) in (\ref%
{def_m_hat_crit_0_VF}), we get that for all $n\geq n_{0}\left( \text{(%
\textbf{GSA}),}r\right) $,%
\begin{eqnarray*}
\ell \left( s_{\ast },\widehat{s}_{\widehat{m}}^{\left( -1\right) }\right)
&\leq &\left( 1-\frac{L_{\text{(\textbf{GSA}),}r}}{\sqrt{\ln n}}\right) ^{-1}%
\left[ \left( 1+\frac{L_{\text{(\textbf{GSA}),}r}}{\sqrt{\ln n}}\right)
\inf_{m\in \mathcal{M}_{n}}\left\{ \ell \left( s_{\ast },\widehat{s}%
_{m}^{\left( -1\right) }\right) \right\} +L_{\text{(\textbf{GSA}),}r}\frac{%
\left( \ln n\right) ^{3}}{n}\right] \\
&\leq &\left( 1+\frac{L_{\text{(\textbf{GSA}),}r}}{\sqrt{\ln n}}\right)
\inf_{m\in \mathcal{M}_{n}}\left\{ \ell \left( s_{\ast },\widehat{s}%
_{m}^{\left( -1\right) }\right) \right\} +L_{\text{(\textbf{GSA}),}r}\frac{%
\left( \ln n\right) ^{3}}{n}.
\end{eqnarray*}%
This concludes the proof of Theorem \ref{theorem_VFCV_gene}.

\subsubsection{Proofs related to V-fold penalization} \label{sssection_proof_VFpen}

Recall that the set of assumptions (\textbf{GSA}) is defined in Section \ref%
{ssection_proof_slope} above. The proof of Theorem \ref{Theorem_VFpen} will be based
on the following theorem, proved in \cite{saum:13} - see Theorem 2 and its
proof under (\textbf{GSA})\ therein.

\begin{thrm}
\label{theorem_opt_pen_reg_pp copy(1)}Suppose that the assumptions \textit{(%
\textbf{GSA})} of Section \ref{section_main_assumptions} hold, and
furthermore suppose that for some $\delta \in \left[ 0,1\right) $ and $%
A_{p},A_{r}>0$, there exists an event of probability at least $1-A_{p}n^{-2}$
on which, for every model $m\in \mathcal{M}_{n}$ such that $D_{m}\geq A_{%
\mathcal{M},+}\left( \ln n\right) ^{3}$, it holds%
\begin{equation*}
\left\vert 
%TCIMACRO{\TeXButton{pen}{\pen}}%
%BeginExpansion
\pen%
%EndExpansion
\left( m\right) -2\mathbb{E}\left[ P_{n}\left( \gamma \left( s_{m}\right)
-\gamma \left( \widehat{s}_{m}\right) \right) \right] \right\vert \leq
\delta \left( \ell \left( s_{\ast },s_{m}\right) +\mathbb{E}\left[
P_{n}\left( \gamma \left( s_{m}\right) -\gamma \left( \widehat{s}_{m}\right)
\right) \right] \right)
\end{equation*}%
together with%
\begin{equation*}
\left\vert 
%TCIMACRO{\TeXButton{pen}{\pen}}%
%BeginExpansion
\pen%
%EndExpansion
\left( m\right) \right\vert \leq A_{r}\left( \frac{\ell \left( s_{\ast
},s_{m}\right) }{\left( \ln n\right) ^{2}}+\frac{\left( \ln n\right) ^{3}}{n}%
\right) .
\end{equation*}%
Then there exist an integer $n_{0}$ only depending on $\delta $ and $\beta
_{+}$ and on constants in \textit{(\textbf{GSA})}, a positive constant $%
L_{3} $ only depending on $c_{\mathcal{M}}$ given in \textit{(\textbf{GSA})}
and on $A_{p}$, two positive constants $L_{4}$ and $L_{5}$ only depending on
constants in \textit{(\textbf{GSA})} and on $A_{r}$ and a sequence 
\begin{equation*}
\theta _{n}\leq \frac{L_{4}}{\left( \ln n\right) ^{1/4}}
\end{equation*}%
such that it holds for all $n\geq n_{0}$, with probability at least $%
1-L_{3}n^{-2}$, 
\begin{equation*}
\ell \left( s_{\ast },\widehat{s}_{\widehat{m}}\right) \leq \left( \frac{%
1+\delta }{1-\delta }+\frac{5\theta _{n}}{\left( 1-\delta \right) ^{2}}%
\right) \inf_{m\in \mathcal{M}_{n}}\left\{ \ell \left( s_{\ast },\widehat{s}%
_{m}\right) \right\} +L_{5}\frac{\left( \ln n\right) ^{3}}{n}.
\end{equation*}
\end{thrm}

We now prove Theorem \ref{Theorem_VFpen}.

\noindent \textbf{Proof of Theorem \ref{Theorem_VFpen}. }We set 
\begin{equation*}
%TCIMACRO{\TeXButton{pen}{\pen}}%
%BeginExpansion
\pen%
%EndExpansion
_{0}\left( m\right) =%
%TCIMACRO{\TeXButton{pen}{\pen}}%
%BeginExpansion
\pen%
%EndExpansion
_{\mathrm{VF}}\left( m\right) -\frac{V-1}{V}\sum_{j=1}^{V}\left( P_{n}\left(
\gamma \left( s_{\ast }\right) \right) -P_{n}^{(-j)}\left( \gamma \left(
s_{\ast }\right) \right) \right) \text{ }.
\end{equation*}%
It is worth noting that the penalization procedure defined by $%
%TCIMACRO{\TeXButton{pen}{\pen}}%
%BeginExpansion
\pen%
%EndExpansion
_{0}$ gives the same result as the procedure defined by $%
%TCIMACRO{\TeXButton{pen}{\pen}}%
%BeginExpansion
\pen%
%EndExpansion
_{\mathrm{VF}}$. It will be convenient for our analysis to consider $%
%TCIMACRO{\TeXButton{pen}{\pen}}%
%BeginExpansion
\pen%
%EndExpansion
_{0}$ instead of $%
%TCIMACRO{\TeXButton{pen}{\pen}}%
%BeginExpansion
\pen%
%EndExpansion
_{\mathrm{VF}}$. Our strategy is to derive Theorem \ref{Theorem_VFpen}\ as a
corollary of Theorem \ref{theorem_opt_pen_reg_pp copy(1)}\ applied with $%
%TCIMACRO{\TeXButton{pen}{\pen}}%
%BeginExpansion
\pen%
%EndExpansion
\equiv 
%TCIMACRO{\TeXButton{pen}{\pen}}%
%BeginExpansion
\pen%
%EndExpansion
_{0}$.

As $P_{n}=(1-V^{-1})P_{n}^{\left( -j\right) }+V^{-1}P_{n}^{\left( j\right) }$%
, we get for all $m\in \mathcal{M}_{n}$,%
\begin{eqnarray*}
%TCIMACRO{\TeXButton{pen}{\pen}}%
%BeginExpansion
\pen%
%EndExpansion
_{0}\left( m\right) &=&\frac{V-1}{V}\sum_{j=1}^{V}\left( P_{n}\left( \gamma
\left( \widehat{s}_{m}^{(-j)}\right) -\gamma \left( s_{\ast }\right) \right)
-P_{n}^{\left( -j\right) }\left( \gamma \left( \widehat{s}_{m}^{(-j)}\right)
-\gamma \left( s_{\ast }\right) \right) \right) \\
&=&\frac{V-1}{V^{2}}\sum_{j=1}^{V}\left( P_{n}^{\left( j\right) }\left(
\gamma \left( \widehat{s}_{m}^{(-j)}\right) -\gamma \left( s_{\ast }\right)
\right) -P_{n}^{\left( -j\right) }\left( \gamma \left( \widehat{s}%
_{m}^{(-j)}\right) -\gamma \left( s_{\ast }\right) \right) \right) \\
&=&\frac{V-1}{V^{2}}\sum_{j=1}^{V}\left( P_{n}^{\left( j\right) }\left(
\gamma \left( \widehat{s}_{m}^{(-j)}\right) -\gamma \left( s_{m}\right)
\right) -P_{n}^{\left( -j\right) }\left( \gamma \left( \widehat{s}%
_{m}^{(-j)}\right) -\gamma \left( s_{m}\right) \right) \right) \\
&&+\frac{V-1}{V^{2}}\sum_{j=1}^{V}\left( \left( P_{n}^{\left( j\right)
}-P\right) \left( \gamma \left( s_{m}\right) -\gamma \left(
s_{\ast }\right) \right) -\left( P_{n}^{\left( -j\right) }-P\right) \left(
\gamma \left( s_{m}\right) -\gamma \left( s_{\ast }\right) \right)
\right) \\
&=&\frac{V-1}{V}\left( \text{\={p}}_{1}\left( m\right) +\text{\={p}}%
_{2}\left( m\right) +\bar{\delta}\left( m\right) -\bar{\delta}^{^{\prime
}}\left( m\right) \right)
\end{eqnarray*}%
where%
\begin{equation*}
\text{\={p}}_{1}\left( m\right) =\frac{1}{V}\sum_{j=1}^{V}P_{n}^{\left(
j\right) }\left( \gamma \left( \widehat{s}_{m}^{(-j)}\right) -\gamma \left(
s_{m}\right) \right) \text{ , \={p}}_{2}\left( m\right) =\frac{1}{V}%
\sum_{j=1}^{V}P_{n}^{\left( -j\right) }\left( \gamma \left( s_{m}\right)
-\gamma \left( \widehat{s}_{m}^{(-j)}\right) \right) ,
\end{equation*}%
and $\bar{\delta}\left( m\right)$ and $\bar{\delta}'\left( m\right)$ have been defined in Lemma \ref{lemma_delta_VFCV}.
We also set%
\begin{equation*}
\text{p}_{1}\left( m\right) =P\left( \gamma \left( \widehat{s}_{m}\right)
-\gamma \left( s_{m}\right) \right) \text{ \ \ and \ \ p}_{2}\left( m\right)
=P_{n}\left( \gamma \left( s_{m}\right) -\gamma \left( \widehat{s}%
_{m}\right) \right) .
\end{equation*}%
Let $\Omega _{n}$ be the event on which:

\begin{itemize}
\item For all models $m\in \mathcal{M}_{n}$ of dimension $D_{m}$ such that $%
A_{\mathcal{M},+}\left( \ln n\right) ^{3}\leq D_{m}$, it holds 
\begin{align*}
\left\vert \text{p}_{1}\left( m\right) -\mathbb{E}\left[ \text{p}_{2}\left(
m\right) \right] \right\vert & \leq L_{\text{(\textbf{GSA})}}\varepsilon
_{n}\left( m\right) \mathbb{E}\left[ \text{p}_{2}\left( m\right) \right]
 \\%\label{line_1_n1}
\left\vert \text{p}_{2}\left( m\right) -\mathbb{E}\left[ \text{p}_{2}\left(
m\right) \right] \right\vert & \leq L_{\text{(\textbf{GSA})}}\varepsilon
_{n}^{2}\left( m\right) \mathbb{E}\left[ \text{p}_{2}\left( m\right) \right]
%\label{line_2_n1}
\end{align*}%
where $\varepsilon
_{n}\left( m\right)$ is defined in Theorem \ref{theorem_excess_risk_strong_loc}, together with%
\begin{eqnarray}
\left\vert \text{\={p}}_{1}\left( m\right) -\frac{V}{V-1}\mathbb{E}\left[ 
\text{p}_{2}\left( m\right) \right] \right\vert &\leq &L_{\text{(\textbf{GSA}%
)},r}\varepsilon _{n}\left( m\right) \mathbb{E}\left[ \text{p}_{2}\left(
m\right) \right] \nonumber  \\%\label{line_p1n2}
\left\vert \text{\={p}}_{2}\left( m\right) -\frac{V}{V-1}\mathbb{E}\left[ 
\text{p}_{2}\left( m\right) \right] \right\vert &\leq &L_{\text{(\textbf{GSA}%
)},r}\varepsilon _{n}^{2}\left( m\right) \mathbb{E}\left[ \text{p}_{2}\left(
m\right) \right]  \nonumber\\%\label{line_p2n1} 
\max \left\{ \left\vert \bar{\delta}\left( m\right) \right\vert ,\left\vert 
\bar{\delta}^{\prime }\left( m\right) \right\vert \right\} &\leq &\frac{\ell
\left( s_{\ast },s_{m}\right) }{\sqrt{D_{m}}}+L_{\text{(\textbf{GSA})},r}%
\frac{\ln n}{\sqrt{D_{m}}}\mathbb{E}\left[ \text{p}_{2}\left( m\right) %
\right]  \label{line_delta}
\end{eqnarray}

\item For all models $m\in \mathcal{M}_{n}$ of dimension $D_{m}$ such that $%
D_{m}\leq A_{\mathcal{M},+}\left( \ln n\right) ^{3}$, it holds%
\begin{eqnarray}
\max \left\{ \left\vert \bar{\delta}\left( m\right) \right\vert ,\left\vert 
\bar{\delta}^{\prime }\left( m\right) \right\vert \right\} &\leq &L_{\text{(%
\textbf{GSA})},r}\left( \sqrt{\frac{\ell \left( s_{\ast },s_{m}\right) \ln n%
}{n}}+\frac{\ln n}{n}\right)  \label{delta_small} \\
\text{\={p}}_{2}\left( m\right) &\leq &L_{\text{(\textbf{GSA})},r}\frac{%
D_{m}\vee \ln n}{n}\leq L_{\text{(\textbf{GSA})},r}\frac{\left( \ln n\right)
^{3}}{n}  \label{p2bar_small} \\
\text{\={p}}_{1}\left( m\right) &\leq &L_{\text{(\textbf{GSA})},r}\left( 
\frac{\left( \ln n\right) ^{2}}{n}+\frac{D_{m}\vee \ln n}{n}\right) \leq L_{%
\text{(\textbf{GSA})},r}\frac{\left( \ln n\right) ^{3}}{n}
\label{p1bar_small}
\end{eqnarray}
\end{itemize}

\noindent By Theorem 2 of \cite{saum:12} and Lemma 4 of \cite{saum:13}
applied with $\alpha =2+\alpha _{\mathcal{M}}$ and sample size $%
n_{V}=n\left( V-1\right) /V$, Corollary \ref{corollary_V_fold} and Lemma %
\ref{lemma_delta_VFCV} applied with $\alpha =2+\alpha _{\mathcal{M}}$, we
get for all $n\geq n_{0}\left( \text{(\textbf{GSA})},r\right) $, 
\begin{equation*}
\mathbb{P}\left( \Omega _{n}\right) \geq 1-L\sum_{m\in \mathcal{M}%
_{n}}n^{-2-\alpha _{\mathcal{M}}}\geq 1-L_{c_{\mathcal{M}}}n^{-2}\text{ }.
\end{equation*}

\smallskip

\noindent We consider models $m\in \mathcal{M}_{n}$ such that $A_{\mathcal{M}%
,+}\left( \ln n\right) ^{3}\leq D_{m}$. Notice that (\ref{line_delta})
implies by (\ref{def_epsilon}) that, for all $m\in \mathcal{M}_{n}$ such
that $A_{\mathcal{M},+}\left( \ln n\right) ^{3}\leq D_{m}$,%
\begin{eqnarray*}
\max \left\{ \left\vert \bar{\delta}\left( m\right) \right\vert ,\left\vert 
\bar{\delta}^{\prime }\left( m\right) \right\vert \right\} &\leq &L_{\text{(%
\textbf{GSA})},r}\left( \frac{\left( \ln n\right) ^{3}}{D_{m}}%
%TCIMACRO{\TeXButton{.}{\cdot}}%
%BeginExpansion
\cdot%
%EndExpansion
\frac{\ln n}{D_{m}}\right) ^{1/4}\times \left( \ell \left( s_{\ast
},s_{m}\right) +\mathbb{E}\left[ p_{2}\left( m\right) \right] \right) \\
&\leq &L_{\text{(\textbf{GSA})},r}\varepsilon _{n}\left( m\right) \left(
\ell \left( s_{\ast },s_{m}\right) +\mathbb{E}\left[ p_{2}\left( m\right) %
\right] \right) .
\end{eqnarray*}%
We deduce that on $\Omega _{n}$ we have, for all models $m\in \mathcal{M}%
_{n} $ such that $A_{\mathcal{M},+}\left( \ln n\right) ^{3}\leq D_{m}$ and
for all $n\geq n_{0}\left( \text{(\textbf{GSA),}}r\right) $,

\begin{eqnarray}
&&\left\vert 
%TCIMACRO{\TeXButton{pen}{\pen}}%
%BeginExpansion
\pen%
%EndExpansion
_{0}\left( m\right) -2\mathbb{E}\left[ \text{p}_{2}\left( m\right) \right]
\right\vert  \notag \\
&\leq &\frac{V-1}{V}\left( \left\vert \text{\={p}}_{1}\left( m\right) -\frac{V}{%
V-1}\mathbb{E}\left[ \text{p}_{2}\left( m\right) \right] \right\vert
+\left\vert \text{\={p}}_{2}\left( m\right) -\frac{V}{V-1}\mathbb{E}%
\left[ \text{p}_{2}\left( m\right) \right] \right\vert \right)  \notag \\
&&+\max \left\{ \left\vert \bar{\delta}\left( m\right) \right\vert
,\left\vert \bar{\delta}^{\prime }\left( m\right) \right\vert \right\} 
\notag \\
&\leq &L_{\text{(\textbf{GSA})},r}\varepsilon _{n}\left( m\right) \left(
\ell \left( s_{\ast },s_{m}\right) +\mathbb{E}\left[ \text{p}_{2}\left(
m\right) \right] \right)  \label{control_pen_ho_res}
\end{eqnarray}

\noindent Let us now consider models $m\in \mathcal{M}_{n}$ such that $%
D_{m}\leq A_{\mathcal{M},+}\left( \ln n\right) ^{3}$. By (\ref{delta_small}%
), (\ref{p2bar_small}) and (\ref{p1bar_small}), we have on $\Omega _{n}$,%
\begin{eqnarray}
\left\vert 
%TCIMACRO{\TeXButton{pen}{\pen}}%
%BeginExpansion
\pen%
%EndExpansion
_{0}\left( m\right) \right\vert &=&\frac{V-1}{V}\left\vert \text{\={p}}%
_{1}\left( m\right) +\text{\={p}}_{2}\left( m\right) +\bar{\delta}\left(
m\right) -\bar{\delta}^{^{\prime }}\left( m\right) \right\vert  \notag \\
&\leq &L_{\text{(\textbf{GSA})},r}\left( \sqrt{\frac{\ell \left( s_{\ast
},s_{m}\right) \ln n}{n}}+\frac{\left( \ln n\right) ^{3}}{n}\right)  \notag
\\
&\leq &L_{\text{(\textbf{GSA})},r}\left( \frac{\ell \left( s_{\ast
},s_{m}\right) }{\left( \ln n\right) ^{2}}+\frac{\left( \ln n\right) ^{3}}{n}%
\right)  \label{control_pen_VF_small}
\end{eqnarray}%
Inequality (\ref{control_pen_VF_small}) implies that inequality (\ref%
{pen_id_2_pp}) of Theorem \ref{theorem_opt_pen_reg_pp}\ is satisfied with $%
A_{r}=L_{\text{(\textbf{GSA})},r}$. From (\ref{control_pen_ho_res}) and (\ref%
{control_pen_VF_small}), we thus apply Theorem \ref{theorem_opt_pen_reg_pp
copy(1)}\ with $A_{p}=L_{A_{p},c_{\mathcal{M}}}$, and this gives Theorem \ref%
{Theorem_VFpen} with%
\begin{equation*}
\theta _{n}=L_{\text{(\textbf{GSA})},r}\left( \left( \ln n\right)
^{-2}+\sup_{m\in \mathcal{M}_{n}}\left\{ \varepsilon _{n}\left( m\right) ;%
\text{ }A_{\mathcal{M},+}\left( \ln n\right) ^{3}\leq D_{m}\leq n^{\eta
+1/\left( 1+\beta _{+}\right) }\right\} \right) .
\end{equation*}

\subsection{Proofs related to Section \protect\ref%
{section_excess_risk_concentration}}%\label{ssection_proof_excess_risk_concen}

\subsubsection{Proofs for strongly localized bases}

\noindent \textbf{Proof of Theorem \ref{theorem_general_sup}.} Let $C>0$. Set%
\begin{equation*}
\mathcal{F}_{C}^{\infty }:=\left\{ s\in m\text{ };\left\Vert
s-s_{m}\right\Vert _{\infty }\leq C\right\}
\end{equation*}%
and%
\begin{equation*}
\mathcal{F}_{>C}^{\infty }:=\left\{ s\in m\text{ };\left\Vert
s-s_{m}\right\Vert _{\infty }>C\right\} =m\backslash \mathcal{F}_{C}^{\infty
}.
\end{equation*}%
Take an orthonormal basis $\left( \varphi _{k}\right) _{k=1}^{D_{m}}$ of $\left(
m,\left\Vert 
%TCIMACRO{\TeXButton{.}{\cdot}}%
%BeginExpansion
\cdot%
%EndExpansion
\right\Vert _{2}\right) $ satisfying (\textbf{Aslb}). By Lemma \ref%
{lemma_dev_phi1_reg_sup}, we get that there exists $L_{A,r_{m},\alpha
}^{(1)}>0$ such that, by setting%
\begin{equation*}
\Omega _{1}=\left\{ \max_{k\in \left\{ 1,\ldots,D_m\right\} }\left\vert \left(
P_{n}-P\right) \left( \psi _{m}%
%TCIMACRO{\TeXButton{.}{\cdot}}%
%BeginExpansion
\cdot%
%EndExpansion
\varphi _{k}\right) \right\vert \leq L_{A,r_{m},\alpha }^{(1)}\sqrt{\frac{%
\ln n}{n}}\right\} ,
\end{equation*}%
we have for all $n\geq n_{0}\left( A_{+}\right) $, $\mathbb{P}\left( \Omega
_{1}\right) \geq 1-n^{-\alpha }$. Moreover, we set%
\begin{equation*}
\Omega _{2}=\left\{ \max_{\left( k,l\right) \in \left\{ 1,\ldots,D_m\right\}
^{2}}\left\vert \left( P_{n}-P\right) \left( \varphi _{k}%
%TCIMACRO{\TeXButton{.}{\cdot}}%
%BeginExpansion
\cdot%
%EndExpansion
\varphi _{l}\right) \right\vert \leq L_{\alpha ,r_{m}}^{(2)}\min \left\{
\left\Vert \varphi _{k}\right\Vert _{\infty };\left\Vert \varphi
_{l}\right\Vert _{\infty }\right\} \sqrt{\frac{\ln n}{n}}\right\} ,
\end{equation*}%
where $L_{\alpha ,r_{m}}^{(2)}$ is defined in Lemma \ref%
{lemma_termes_croises}. By Lemma \ref{lemma_termes_croises}, we have that
for all $n\geq n_{0}\left( A_{+}\right) $, $\mathbb{P}\left( \Omega
_{2}\right) \geq 1-n^{-\alpha }$ and so, for all $n\geq n_{0}\left(
A_{+}\right) $,%
\begin{equation*}
\mathbb{P}\left( \Omega _{1}\bigcap \Omega _{2}\right) \geq 1-2n^{-\alpha }%
\text{ }.  %\label{proba_inter}
\end{equation*}%
We thus have for all $n\geq n_{0}\left( A_{+}\right) $,%
\begin{eqnarray}
&&\mathbb{P}\left( \left\Vert \widehat{s}_{m}-s_{m}\right\Vert _{\infty
}>C\right)  \notag \\
&\leq &\mathbb{P}\left( \inf_{s\in \mathcal{F}_{>C}^{\infty }}P_{n}\left(
\gamma \left( s\right) -\gamma \left( s_{m}\right) \right) \leq \inf_{s\in 
\mathcal{F}_{C}^{\infty }}P_{n}\left( \gamma \left( s\right) -\gamma \left(
s_{m}\right) \right) \right)  \notag \\
&=&\mathbb{P}\left( \sup_{s\in \mathcal{F}_{>C}^{\infty }}P_{n}\left( \gamma
\left( s_{m}\right) -\gamma \left( s\right) \right) \geq \sup_{s\in \mathcal{%
F}_{C}^{\infty }}P_{n}\left( \gamma \left( s_{m}\right) -\gamma \left(
s\right) \right) \right)  \notag \\
&\leq &\mathbb{P}\left( \left\{ \sup_{s\in \mathcal{F}_{>C}^{\infty
}}P_{n}\left( \gamma \left( s_{m}\right) -\gamma \left( s\right) \right)
\geq \sup_{s\in \mathcal{F}_{C/2}^{\infty }}P_{n}\left( \gamma \left(
s_{m}\right) -\gamma \left( s\right) \right) \right\} \bigcap \Omega
_{1}\bigcap \Omega _{2}\right) +2n^{-\alpha }.
\label{rewriting_proof_sup}
\end{eqnarray}%
Now, for any $s\in m$ such that%
\begin{equation*}
s-s_{m}=\sum_{k=1}^{D_m}\beta _{k}\varphi _{k}\text{, }\beta =\left( \beta
_{k}\right) _{k=1}^{D_m}\in \mathbb{R}^{D_m}\text{,}
\end{equation*}%
we have 
\begin{eqnarray*}
&&P_{n}\left( \gamma \left( s_{m}\right) -\gamma \left( s\right) \right) \\
&=&\left( P_{n}-P\right) \left( \psi _{m}%
%TCIMACRO{\TeXButton{.}{\cdot}}%
%BeginExpansion
\cdot%
%EndExpansion
\left( s_{m}-s\right) \right) -\left( P_{n}-P\right) \left( \left(
s-s_{m}\right) ^{2}\right) -P\left( \gamma \left( s\right) -\gamma \left(
s_{m}\right) \right) \\
&=&\sum_{k=1}^{D_m}\beta _{k}\left( P_{n}-P\right) \left( \psi _{m}%
%TCIMACRO{\TeXButton{.}{\cdot}}%
%BeginExpansion
\cdot%
%EndExpansion
\varphi _{k}\right) -\sum_{k,l=1}^{D_m}\beta _{k}\beta _{l}\left(
P_{n}-P\right) \left( \varphi _{k}%
%TCIMACRO{\TeXButton{.}{\cdot}}%
%BeginExpansion
\cdot%
%EndExpansion
\varphi _{l}\right) -\sum_{k=1}^{D_m}\beta _{k}^{2}.
\end{eqnarray*}%
We set for any $\left( k,l\right) \in \left\{ 1,\ldots,D_m\right\} ^{2}$, 
\begin{equation*}
R_{n,k}^{(1)}=\left( P_{n}-P\right) \left( \psi _{m}%
%TCIMACRO{\TeXButton{.}{\cdot}}%
%BeginExpansion
\cdot%
%EndExpansion
\varphi _{k}\right) \text{ \ and \ }R_{n,k,l}^{(2)}=\left( P_{n}-P\right)
\left( \varphi _{k}%
%TCIMACRO{\TeXButton{.}{\cdot}}%
%BeginExpansion
\cdot%
%EndExpansion
\varphi _{l}\right) .
\end{equation*}%
Moreover, we set a function $h_{n}$, defined as follows,%
\begin{equation*}
h_{n}:\beta =\left( \beta _{k}\right) _{k=1}^{D_m}\longmapsto
\sum_{k=1}^{D_m}\beta _{k}R_{n,k}^{(1)}-\sum_{k,l=1}^{D_m}\beta _{k}\beta
_{l}R_{n,k,l}^{(2)}-\sum_{k=1}^{D_m}\beta _{k}^{2}.
\end{equation*}%
We thus have for any $s\in m$ such that $s-s_{m}=\sum_{k=1}^{D_m}\beta
_{k}\varphi _{k}$, $\beta =\left( \beta _{k}\right) _{k=1}^{D_m}\in \mathbb{R}%
^{D_m}$,%
\begin{equation}
P_{n}\left( \gamma \left( s_{m}\right) -\gamma \left( s\right) \right)
=h_{n}\left( \beta \right) .  \label{identity_h_n_emp_excess_risk}
\end{equation}%
In addition we set for any $\beta =\left( \beta _{k}\right) _{k=1}^{D_m}\in 
\mathbb{R}^{D_m}$,%
\begin{equation*}
\left\vert \beta \right\vert _{m,\infty }=r_{m}\sum_{i=1}^{b_m}\sqrt{A_{i}}%
\max_{k\in \Pi _{i}}\left\vert \beta _{k}\right\vert .
%\label{def_norm_vect}
\end{equation*}%
It is straightforward to see that $\left\vert 
%TCIMACRO{\TeXButton{.}{\cdot}}%
%BeginExpansion
\cdot%
%EndExpansion
\right\vert _{m,\infty }$ is a norm on $\mathbb{R}^{D_{m}}$. We also set for a
real $D_m\times D_m$ matrix $B$, its operator norm $\left\Vert B\right\Vert _{m}$
associated to the norm $\left\vert 
%TCIMACRO{\TeXButton{.}{\cdot}}%
%BeginExpansion
\cdot%
%EndExpansion
\right\vert _{m,\infty }$ on the $D_m$-dimensional vectors. More explicitly,
we set for any $B\in \mathbb{R}^{D_{m}\times D_{m}}$,%
\begin{equation*}
\left\Vert B\right\Vert _{m}:=\sup_{\beta \in \mathbb{R}^{D},\text{ }\beta
\neq 0}\frac{\left\vert B\beta \right\vert _{m,\infty }}{\left\vert \beta
\right\vert _{m,\infty }}.
\end{equation*}%
We have, for any $B=\left( B_{k,l}\right) _{k,l=1,\ldots D_m}\in \mathbb{R}^{D_m\times
D_m}$,%
\begin{eqnarray*}
\left\Vert B\right\Vert _{m} &=&\sup_{\beta \in \mathbb{R}^{D_m},\text{ }%
\left\vert \beta \right\vert _{m,\infty }=1}\left\{ r_{m}\sum_{i=1}^{b_m}\sqrt{%
A_{i}}\max_{k\in \Pi _{i}}\left\vert \sum_{l=1}^{D_{m}}B_{k,l}\beta
_{l}\right\vert \right\}  \notag \\
&=&\sup_{\beta \in \mathbb{R}^{D_{m}},\text{ }\left\vert \beta \right\vert
_{m,\infty }=1}\left\{ r_{m}\sum_{i=1}^{b_m}\sqrt{A_{i}}\max_{k\in \Pi
_{i}}\left\vert \sum_{j=1}^{b_m}\sum_{l\in \Pi _{j}}B_{k,l}\beta
_{l}\right\vert \right\}  \notag \\
&=&\sup_{\beta \in \mathbb{R}^{D_{m}},\text{ }\left\vert \beta \right\vert
_{m,\infty }=1}\left\{ \sum_{i=1}^{b_m}\sqrt{A_{i}}\max_{k\in \Pi _{i}}\left\{
r_{m}\sum_{j=1}^{b_m}\sqrt{A_{j}}\max_{l\in \Pi _{j}}\left\vert \beta
_{l}\right\vert \left( \sqrt{A_{j}^{-1}}\sum_{l\in \Pi _{j}}\left\vert
B_{k,l}\right\vert \right) \right\} \right\}  \notag \\
&=&\sum_{i=1}^{b_m}\sqrt{A_{i}}\max_{k\in \Pi _{i}}\left\{ \max_{j\in \left\{
1,\ldots,b_m\right\} }\left\{ \sqrt{A_{j}^{-1}}\sum_{l\in \Pi _{j}}\left\vert
B_{k,l}\right\vert \right\} \right\} .  %\label{formula_operator_norm}
\end{eqnarray*}%
Notice that by Inequality (\ref{loc_plus}) of (\textbf{Aslb}), it holds 
\begin{equation}
\mathcal{F}_{>C}^{\infty }\subset \left\{ s\in m\text{ };\text{ }%
s-s_{m}=\sum_{k=1}^{D_{m}}\beta _{k}\varphi _{k}\And \left\vert \beta
\right\vert _{m,\infty }\geq C\right\}  \label{inclu_slice_>C}
\end{equation}%
and 
\begin{equation}
\mathcal{F}_{C/2}^{\infty }\supset \left\{ s\in m\text{ };\text{ }%
s-s_{m}=\sum_{k=1}^{D_{m}}\beta _{k}\varphi _{k}\And \left\vert \beta
\right\vert _{m,\infty }\leq C/2\right\} .  \label{inclu_slice_C}
\end{equation}%
Hence, from (\ref{rewriting_proof_sup}), (\ref{identity_h_n_emp_excess_risk}%
) (\ref{inclu_slice_>C}) and (\ref{inclu_slice_C}) we deduce that if we find
on $\Omega _{1}\bigcap \Omega _{2}$ a value of $C$ such that 
\begin{equation*}
\sup_{\beta \in \mathbb{R}^{D_{m}},\text{ }\left\vert \beta \right\vert
_{m,\infty }\geq C}h_{n}\left( \beta \right) <\sup_{\beta \in \mathbb{R}^{D_{m}},%
\text{ }\left\vert \beta \right\vert _{m,\infty }\leq C/2}h_{n}\left( \beta
\right) ,  %\label{but_h_n}
\end{equation*}%
then Inequality (\ref{conv_sup_norm_reg_sup}) follows and Theorem \ref%
{theorem_general_sup}\ is proved. Taking the partial derivatives of $h_{n}$
with respect to the coordinates of its arguments, it then holds for any $%
\left( k,l\right) \in \left\{ 1,\ldots,D_m\right\} ^{2}$ and $\beta =\left( \beta
_{i}\right) _{i=1}^{D_{m}}\in \mathbb{R}^{D_{m}}$,%
\begin{equation}
\frac{\partial h_{n}}{\partial \beta _{k}}\left( \beta \right)
=R_{n,k}^{(1)}-2\sum_{i=1}^{D_{m}}\beta _{i}R_{n,k,i}^{(2)}-2\beta _{k}
\label{derivative_h_n}
\end{equation}%
We look now at the set of solutions $\beta$ of the following system,%
\begin{equation}
\frac{\partial h_{n}}{\partial \beta _{k}}\left( \beta \right) =0\text{ , }%
\forall k\in \left\{ 1,\ldots,D_m\right\} .  \label{def_system_1}
\end{equation}%
We define the $D_m\times D_m$ matrix $R_{n}^{(2)}$ to be%
\begin{equation*}
R_{n}^{(2)}:=\left( R_{n,k,l}^{(2)}\right) _{k,l=1,\ldots,D_m}
\end{equation*}%
and by (\ref{derivative_h_n}), the system given in (\ref{def_system_1}) can
be written%
\begin{equation}
2\left( I_{D_{m}}+R_{n}^{(2)}\right) \beta =R_{n}^{(1)}, 
\tag{\textbf{S}}  \label{def_system_matrix}
\end{equation}%
where $R_{n}^{(1)}$ is a $D$-dimensional vector defined by%
\begin{equation*}
R_{n}^{(1)}=\left( R_{n,k}^{(1)}\right) _{k=1,\ldots,D_m}.
\end{equation*}%
Let us give an upper bound of the norm $\left\Vert R_{n}^{(2)}\right\Vert
_{m}$, in order to show that the matrix $I_{D_{m}}+R_{n}^{(2)}$ is nonsingular.
On $\Omega _{2}$ we have, using (\ref{def_card_Pi_i_k}),%
\begin{eqnarray}
\left\Vert R_{n}^{(2)}\right\Vert _{m} &=&\sum_{i=1}^{b_m}\sqrt{A_{i}}%
\max_{k\in \Pi _{i}}\left\{ \max_{j\in \left\{ 1,\ldots,b_m\right\} }\left\{ 
\sqrt{A_{j}^{-1}}\sum_{l\in \Pi _{j}}\left\vert R_{n,k,l}^{(2)}\right\vert
\right\} \right\}  \notag \\
&=&\sum_{i=1}^{b_m}\sqrt{A_{i}}\max_{k\in \Pi _{i}}\left\{ \max_{j\in \left\{
1,\ldots,b_m\right\} }\left\{ \sqrt{A_{j}^{-1}}\sum_{l\in \Pi _{j\left\vert
k\right. }}\left\vert R_{n,k,l}^{(2)}\right\vert \right\} \right\}  \notag \\
&\leq &\sum_{i=1}^{b_m}\sqrt{A_{i}}\max_{k\in \Pi _{i}}\left\{ \max_{j\in
\left\{ 1,\ldots,b_m\right\} }\left\{ \sqrt{A_{j}^{-1}}\text{Card}\left( \Pi
_{j\left\vert k\right. }\right) \max_{l\in \Pi _{j}}\left\vert \left(
P_{n}-P\right) \left( \varphi _{k}%
%TCIMACRO{\TeXButton{.}{\cdot}}%
%BeginExpansion
\cdot%
%EndExpansion
\varphi _{l}\right) \right\vert \right\} \right\}  \notag \\
&\leq &A_{c}L_{\alpha ,r_{m}}^{(2)}\sqrt{\frac{\ln n}{n}}\sum_{i=1}^{b_m}%
\max_{j\in \left\{ 1,\ldots,b_m\right\} }\left\{ \sqrt{\frac{A_{i}}{A_{j}}}\left( 
\frac{A_{j}}{A_{i}}\vee 1\right) \sqrt{\min \left\{ A_{i};A_{j}\right\} }%
\right\}  \label{calcul_norm_R_2}
\end{eqnarray}%
We deduce from (\ref{def_Ai}) and (\ref{calcul_norm_R_2}) that on $\Omega
_{2}$,%
\begin{equation}
\left\Vert R_{n}^{(2)}\right\Vert _{m}\leq L_{A_{c},\alpha ,r_{m}}%
%TCIMACRO{\TeXButton{.}{\cdot}}%
%BeginExpansion
\cdot%
%EndExpansion
b_m\sqrt{\frac{A_{b_m}\ln n}{n}}.  \label{upper_R_2_norme}
\end{equation}%
Hence, from (\ref{upper_R_2_norme}) and the fact that $b_m^{2}A_{b_m}\leq A_{+}%
\frac{n}{\left( \ln n\right) ^{2}}$, we get that for all $n\geq n_{0}\left(
A_{+},A_{c},r_{m},\alpha \right) $, it holds on $\Omega _{2}$,%
\begin{equation*}
\left\Vert R_{n}^{(2)}\right\Vert _{m}\leq \frac{1}{2}
\end{equation*}%
and the matrix $\left( I_{D_{m}}+R_{n}^{(2)}\right) $ is nonsingular, of inverse 
$\left( I_{D_{m}}+R_{n}^{(2)}\right) ^{-1}=\sum_{u=0}^{+\infty }\left(
-R_{n}^{(2)}\right) ^{u}$. Hence, the system (\ref{def_system_matrix})
admits a unique solution $\beta ^{(n)}$, given by%
\begin{equation*}
\beta ^{(n)}=\frac{1}{2}\left( I_{D_{m}}+R_{n}^{(2)}\right) ^{-1}R_{n}^{(1)}%
.
\end{equation*}%
Now, on $\Omega _{1}$ we have by (\ref{def_Ai}),%
\begin{equation*}
\left\vert R_{n}^{(1)}\right\vert _{m,\infty }\leq r_{m}\left( \sum_{i=1}^{b_m}%
\sqrt{A_{i}}\right) \max_{k\in \left\{ 1,\ldots,D_m\right\} }\left\vert \left(
P_{n}-P\right) \left( \psi _{m}%
%TCIMACRO{\TeXButton{.}{\cdot}}%
%BeginExpansion
\cdot%
%EndExpansion
\varphi _{k}\right) \right\vert \leq r_{m}L_{A_{m},r_{m},\alpha }^{(1)}\sqrt{%
\frac{D_{m}\ln n}{n}}  %\label{upper_norme_R_1}
\end{equation*}%
and we deduce that for all $n_{0}\left( A_{+},A_{c},r_{m},\alpha \right) $,
it holds on $\Omega _{2}\bigcap \Omega _{1}$,%
\begin{equation}
\left\vert \beta ^{(n)}\right\vert _{m,\infty }\leq \frac{1}{2}\left\Vert
\left( I_{d}+R_{n}^{(2)}\right) ^{-1}\right\Vert _{m}\left\vert
R_{n}^{(1)}\right\vert _{m,\infty }\leq r_{m}L_{A,r_{m},\alpha }^{(1)}\sqrt{%
\frac{D_m\ln n}{n}}.  \label{upper_betha_n}
\end{equation}%
Moreover, by the formula (\ref{identity_h_n_emp_excess_risk}) we have%
\begin{equation*}
h_{n}\left( \beta \right) =P_{n}\left( \gamma \left( s_{m}\right) \right)
-P_{n}\left( Y-\sum_{k=1}^{D_{m}}\beta _{k}\varphi _{k}\right) ^{2}
\end{equation*}%
and we thus see that $h_{n}$ is concave. Hence, for all $n_{0}\left(
A_{+},A_{c},r_{m},\alpha \right) $, we get that on $\Omega _{2}$, $\beta
^{(n)}$ is the unique maximum of $h_{n}$ and on $\Omega _{2}\bigcap \Omega
_{1}$, by (\ref{upper_betha_n}), concavity of $h_{n}$ and uniqueness of $%
\beta ^{(n)}$, we get%
\begin{equation*}
h_{n}\left( \beta ^{(n)}\right) =\sup_{\beta \in \mathbb{R}^{D_{m}},\text{ }%
\left\vert \beta \right\vert _{m,\infty }\leq C/2}h_{n}\left( \beta \right)
>\sup_{\beta \in \mathbb{R}^{D_{m}},\text{ }\left\vert \beta \right\vert
_{m,\infty }\geq C}h_{n}\left( \beta \right) ,
\end{equation*}%
with $C=2r_{m}L_{A,r_{m},\alpha }^{(1)}\sqrt{\frac{D_{m}\ln n}{n}}$, which
concludes the proof. 
%TCIMACRO{\TeXButton{blacksquare}{$\blacksquare$}}%
%BeginExpansion
$\blacksquare$%
%EndExpansion

\begin{rmk} \label{remark_proof}
The proof of Theorem \ref{theorem_general_sup}
 can be adapted for models endowed with a localized basis structure. Indeed, if we set for any $B\in \mathbb{R}^{D_{m}\times D_{m}}$,%
\begin{equation*}
\left\Vert B\right\Vert _{m}:=\sup_{\beta \in \mathbb{R}^{D_{m}},\text{ }\beta
\neq 0}\frac{\left\vert B\beta \right\vert _{m,\infty }}{\left\vert \beta
\right\vert _{m,\infty }}=\sup_{\beta \in \mathbb{R}^{D_{m}},\text{ }\beta \neq
0}\frac{\left\vert B\beta \right\vert _{\infty }}{\left\vert \beta
\right\vert _{\infty }},
\end{equation*}%
then we have, the following classical formula%
\begin{equation*}
\left\Vert B\right\Vert _{m}=\max_{k\in \left\{ 1,\ldots,D_m\right\} }\left\{
\left\{ \sum_{l\in \left\{ 1,\ldots,D_m\right\} }\left\vert B_{k,l}\right\vert
\right\} \right\} .
\end{equation*}%
Now, it holds,
\begin{align*}
\left\Vert R_{n}^{(2)}\right\Vert _{m} &=\max_{k\in \left\{ 1,\ldots,D_m\right\}
}\left\{ \left\{ \sum_{l\in \left\{ 1,\ldots,D_m\right\} }\left\vert \left(
P_{n}-P\right) \left( \varphi _{k}%
%TCIMACRO{\TeXButton{.}{\cdot}}%
%BeginExpansion
\cdot%
%EndExpansion
\varphi _{l}\right) \right\vert \right\} \right\}  \notag \\
&\leq L_{\alpha ,r_{m}}^{(2)}\max_{k\in \left\{ 1,\ldots,D_m\right\} }\left\{
\left\{ \sum_{l\in \left\{ 1,\ldots,D_m\right\} }\min \left\{ \left\Vert \varphi
_{k}\right\Vert _{\infty };\left\Vert \varphi _{l}\right\Vert _{\infty
}\right\} \sqrt{\frac{\ln n}{n}}\right\} \right\}  \notag \\
&\leq r_{m}L_{\alpha ,r_{m}}^{(2)}\sqrt{\frac{D_{m}^{3}\ln n}{n}}
\end{align*}%
The previous bound tends to zero if $D_m \leq n^{1/3}/ln^2(n)$ and this is the essential reason why results for localized bases are restricted to models with dimension lower that $n^{1/3}/ln^2(n)$ while for strongly localized bases we can go as far as $D_m \leq n/ln^2(n)$  (see also Remark \ref{remark_core_text}).
\end{rmk}

\subsubsection{Proofs related to excess risks' representations\label%
{ssection_proofs_excess_risk_rep}}

\begin{proof}[Proof of Proposition \protect\ref{prop_func_rep}]
Let us write $C_{\ast }:=\mathcal{F}\left( \widehat{s}_{m}\right) $. It
holds 
\begin{align*}
\inf_{s\in d_{C_{\ast }}}P_{n}\left( \gamma \left( s\right) \right)
&=\inf_{s\in m}P_{n}\left( \gamma \left( s\right) \right) \\
&\leq \min_{C\geq 0}\left\{ \inf_{s\in d_{C}}P_{n}\left( \gamma \left(
s\right) \right) \right\} ,
\end{align*}%
which readily proves Formula (\ref{formula_funct_rep}). Formula (\ref%
{formula_funct_rep_ball}) is a direct consequence of (\ref{formula_funct_rep}%
), since $m_{C}=\bigcup _{R\leq C}d_{C}$.
\end{proof}

\begin{proof}[Proof of Proposition \protect\ref{prop_excess_rep_loc}]
We will only prove the case where $R_{0}=+\infty $. Then the situation where 
$R_{0}\in \mathbb{R}_{+}$ can be deduced easily by noticing that the subset $%
\left\{ s\in m\text{ };\text{ }\mathcal{G}\left( s\right) \leq R_{0}\right\} 
$ of $m$ actually plays the role of $m$ in this latter case.

When $R_{0}=+\infty $, we have with the notations of Proposition \ref%
{prop_func_rep} and by taking $\mathcal{F}=P\left( \gamma \left( \cdot
\right) -\gamma \left( s_{m}\right) \right) $, $\widetilde{d}_{C}=d_{C}$ and 
$\widetilde{m}_{C}=m_{C}$. From formula (\ref{formula_funct_rep}) we thus get%
\begin{align*}
P\left( \gamma \left( \widehat{s}_{m}\right) -\gamma \left( s_{m}\right)
\right) &\in \arg \min_{C\geq 0}\left\{ \inf_{s\in \widetilde{d}%
_{C}}P_{n}\left( \gamma \left( s\right) \right) \right\} \\
&=\arg \max_{C\geq 0}\left\{ \sup_{s\in \widetilde{d}_{C}}P_{n}\left(
\gamma \left( s_{m}\right) -\gamma \left( s\right) \right) \right\} \\
&=\arg \max_{C\geq 0}\left\{ \sup_{s\in \widetilde{d}_{C}}\left(
P_{n}-P\right) \left( \gamma \left( s_{m}\right) -\gamma \left( s\right)
\right) -C\right\} .
\end{align*}%
Hence, Formula (\ref{formula_excess_risk}) is proved. Now, for (\ref%
{formula_excess_risk_ball}), take any $C\geq 0$ and notice that there exists
a random variable $C_{1}\mathbb{\in }\left[ 0,C\right] $ such that 
\begin{align}
\sup_{s\in \widetilde{m}_{C}}\left\{ \left( P_{n}-P\right) \left( \gamma
\left( s_{m}\right) -\gamma \left( s\right) \right) \right\} -C
&=\sup_{s\in \widetilde{d}_{C_{1}}}\left( P_{n}-P\right) \left( \gamma
\left( s_{m}\right) -\gamma \left( s\right) \right) -C  \notag \\
&\leq \sup_{s\in \widetilde{d}_{C_{1}}}\left( P_{n}-P\right) \left( \gamma
\left( s_{m}\right) -\gamma \left( s\right) \right) -C_{1}  \notag \\
&\leq \sup_{s\in \widetilde{d}_{C_{\ast }}}\left( P_{n}-P\right) \left(
\gamma \left( s_{m}\right) -\gamma \left( s\right) \right) -C_{\ast }\text{ ,%
}  \label{majo_ball_sph}
\end{align}%
where $C_{\ast }:=P\left( \gamma \left( \widehat{s}_{m}\right) -\gamma
\left( s_{m}\right) \right) $. Taking $C=C_{\ast }$, we get%
\begin{equation*}
\sup_{s\in \widetilde{m}_{C_{\ast }}}\left\{ \left( P_{n}-P\right) \left(
\gamma \left( s_{m}\right) -\gamma \left( s\right) \right) \right\} -C_{\ast
}\leq \sup_{s\in \widetilde{d}_{C_{\ast }}}\left( P_{n}-P\right) \left(
\gamma \left( s_{m}\right) -\gamma \left( s\right) \right) -C_{\ast }
\end{equation*}%
and since $\widetilde{d}_{C_{\ast }}\subset \widetilde{m}_{C_{\ast }}$, this
implies%
\begin{equation*}
\sup_{s\in \widetilde{m}_{C_{\ast }}}\left\{ \left( P_{n}-P\right) \left(
\gamma \left( s_{m}\right) -\gamma \left( s\right) \right) \right\} -C_{\ast
}=\sup_{s\in \widetilde{d}_{C_{\ast }}}\left( P_{n}-P\right) \left( \gamma
\left( s_{m}\right) -\gamma \left( s\right) \right) -C_{\ast }.
\end{equation*}%
Together with (\ref{majo_ball_sph}), the latter equality gives that for any $%
C\geq 0$,%
\begin{equation*}
\sup_{s\in \widetilde{m}_{C}}\left\{ \left( P_{n}-P\right) \left( \gamma
\left( s_{m}\right) -\gamma \left( s\right) \right) \right\} -C\leq
\sup_{s\in \widetilde{m}_{C_{\ast }}}\left\{ \left( P_{n}-P\right) \left(
\gamma \left( s_{m}\right) -\gamma \left( s\right) \right) \right\} -C_{\ast
},
\end{equation*}%
which is another way to write (\ref{formula_excess_risk_ball}).

Now, considering the case of the empirical excess risk, we could again apply
Proposition \ref{prop_func_rep}, but we will follow a more direct proof. We
have, by definition of $\widehat{s}_{m}$,%
\begin{align*}
\ell _{%
%TCIMACRO{\TeXButton{emp}{\emp}}%
%BeginExpansion
\emp%
%EndExpansion
}\left( s_{m},\widehat{s}_{m}\right) &=P_{n}\left( \gamma \left(
s_{m}\right) -\gamma \left( \widehat{s}_{m}\right) \right) \\
&=\max_{s\in m}\left\{ P_{n}\left( \gamma \left( s_{m}\right) -\gamma
\left( s\right) \right) \right\} .
\end{align*}%
Now, as $\left\{ \mathcal{G}\left( \widehat{s}_{m}\right) \leq R_{0}\right\}
=\bigcup_{C\geq 0}\widetilde{d}_{C}=\bigcup_{C\geq 0}\widetilde{m}_{C}$,
we get%
\begin{align*}
\ell _{%
%TCIMACRO{\TeXButton{emp}{\emp}}%
%BeginExpansion
\emp%
%EndExpansion
}\left( s_{m},\widehat{s}_{m}\right) &=\max_{s\in m}\left\{ P_{n}\left(
\gamma \left( s_{m}\right) -\gamma \left( s\right) \right) \right\} \\
&=\max_{C\geq 0}\sup_{s\in \widetilde{d}_{C}}\left\{ P_{n}\left( \gamma
\left( s_{m}\right) -\gamma \left( s\right) \right) \right\} \\
&=\max_{C\geq 0}\left\{ \sup_{s\in \widetilde{d}_{C}}\left\{ \left(
P_{n}-P\right) \left( \gamma \left( s_{m}\right) -\gamma \left( s\right)
\right) \right\} -C\right\} ,
\end{align*}%
that is (\ref{formula_emp_excess_risk}). Now formula (\ref%
{formula_emp_excess_risk_ball}) follows from the kind of arguments that
allow to prove (\ref{formula_excess_risk_ball}) based on (\ref%
{formula_emp_excess_risk}).
\end{proof}

\subsection*{Acknowledgements}
\noindent The authors are very grateful to the Associate Editor and the anonymous referees for their careful reading and valuable comments that have led to several improvements and new developments of the paper.

\bibliographystyle{abbrv}
\bibliography{Slope_heuristics_regression_13}

\section*{Appendix}\label{section_appendix}

\subsection{Some lemmas instrumental in the proofs} \label{ssection_appen_proof}

We gather here the lemmas that are used in the proofs of Section \ref{section_proof_slope_reg}.

In the next lemma, we apply Lemma 9 of \cite{saum:13}, with $n_{2}=n/V$, $%
n_{1}=n-n_{2}=n\left( 1-1/V\right) $, $\tau =2$ and we set $r=c^{-1}\in
\left( 1,+\infty \right) $. Furthermore, the notations $P_{n_{2}}$ and $%
s_{n_{1}}\left( m\right) $ used in \cite{saum:13} correspond respectively to
the quantities $P_{n}^{(j)}$ and $\widehat{s}_{m}^{(-j)}$.

\begin{lmm}
\label{lemma_hold_out}Assume that (\textbf{GSA}) holds. Let $r\in \left(
2,+\infty \right) $ and $V\in \left\{ 2,\ldots,n-1\right\} $ satisfying $%
1<V\leq r$. Then there exists $L=L_{\left( \text{\textbf{GSA}}\right) \text{,%
}r}>0$ such that for all $m\in \mathcal{M}_{n}$ satisfying $D_{m}\geq A_{%
\mathcal{M},+}\left( \ln n\right) ^{3}$, by setting%
\begin{equation*}
\varepsilon _{n}^{\left( 1\right) }\left( m\right) =L\sqrt{\frac{\ln n}{D_{m}%
}}\leq \frac{L}{\ln n},
\end{equation*}

it holds for all $n\geq n_{0}\left( \left( \text{\textbf{GSA}}\right) \text{,%
}r\right) $ and for all $j\in \left\{ 1,\ldots,V\right\} $, 
\begin{equation*}
\mathbb{P}\left( \left\vert P_{n}^{(j)}\left( \gamma \left( \widehat{s}%
_{m}^{(-j)}\right) -\gamma \left( s_{m}\right) \right) -P\left( \gamma
\left( \widehat{s}_{m}^{(-j)}\right) -\gamma \left( s_{m}\right) \right)
\right\vert \geq \varepsilon _{n}^{\left( 1\right) }\left( m\right) \mathbb{E%
}\left[ 
%TCIMACRO{\TeXButton{p}{\p}}%
%BeginExpansion
\p%
%EndExpansion
_{2}^{(-1)}\left( m\right) \right] \right) \leq 12n^{-2-\alpha _{\mathcal{M}%
}},
\end{equation*}%
where $%
%TCIMACRO{\TeXButton{p}{\p}}%
%BeginExpansion
\p%
%EndExpansion
_{2}^{(-1)}\left( m\right) =P_{n}^{\left( -1\right) }\left( \gamma \left(
s_{m}\right) -\gamma \left( \widehat{s}_{m}^{(-1)}\right) \right) $. If $%
D_{m}\leq A_{\mathcal{M},+}\left( \ln n\right) ^{3}$, then for all $n\geq
n_{0}\left( \left( \text{\textbf{GSA}}\right) \text{,}r\right) $,%
\begin{equation*}
\mathbb{P}\left( \left\vert P_{n}^{(j)}\left( \gamma \left( \widehat{s}%
_{m}^{(-j)}\right) -\gamma \left( s_{m}\right) \right) -P\left( \gamma
\left( \widehat{s}_{m}^{(-j)}\right) -\gamma \left( s_{m}\right) \right)
\right\vert \geq L\frac{\left( \ln n\right) ^{2}}{n}\right) \leq
12n^{-2-\alpha _{\mathcal{M}}}.
\end{equation*}
\end{lmm}

Taking into account the averaging between the blocks of the $V$-fold, we get
from Lemma \ref{lemma_hold_out}\ the following corollary.

\begin{crllr}
\label{corollary_V_fold}Assume that (\textbf{GSA}) holds. Let $r\in \left(
0,1\right) $ and $V\in \left\{ 2,\ldots,n-1\right\} $ satisfying $1<V\leq r$.
Then there exists $L=L_{\left( \text{\textbf{GSA}}\right) \text{,}r}>0$ such
that for all $m\in \mathcal{M}_{n}$ satisfying $D_{m}\geq A_{\mathcal{M}%
,+}\left( \ln n\right) ^{3}$, it holds for all $n\geq n_{0}\left( \left( 
\text{\textbf{GSA}}\right) \text{,}r\right) $, 
\begin{equation}
\mathbb{P}\left( \left\vert P\left( \gamma \left( \widehat{s}%
_{m}^{(-1)}\right) -\gamma \left( s_{m}\right) \right) -\frac{1}{V}%
\sum_{j=1}^{V}P_{n}^{(j)}\left( \gamma \left( \widehat{s}_{m}^{(-j)}\right)
-\gamma \left( s_{m}\right) \right) \right\vert \geq L\varepsilon _{n}\left( m\right) \mathbb{E}\left[ 
%TCIMACRO{\TeXButton{p}{\p}}%
%BeginExpansion
\p%
%EndExpansion
_{2}^{(-1)}\left( m\right) \right] \right) \leq 22rn^{-2-\alpha _{\mathcal{M}%
}},  \label{ineq_excess_grand_V_fold}
\end{equation}%
where $%
%TCIMACRO{\TeXButton{p}{\p}}%
%BeginExpansion
\p%
%EndExpansion
_{2}^{(-1)}\left( m\right) =P_{n}^{\left( -1\right) }\left( \gamma \left(
s_{m}\right) -\gamma \left( \widehat{s}_{m}^{(-1)}\right) \right) $ and $\varepsilon _{n}\left( m\right)$ is defined in Theorem \ref{theorem_excess_risk_strong_loc}. If $%
D_{m}\leq A_{\mathcal{M},+}\left( \ln n\right) ^{3}$, then for all $n\geq
n_{0}\left( \left( \text{\textbf{GSA}}\right) \text{,}r\right) $,%
\begin{equation}
\mathbb{P}\left( \left\vert P\left( \gamma \left( \widehat{s}%
_{m}^{(-1)}\right) -\gamma \left( s_{m}\right) \right) -\frac{1}{V}%
\sum_{j=1}^{V}P_{n}^{(j)}\left( \gamma \left( \widehat{s}_{m}^{(-j)}\right)
-\gamma \left( s_{m}\right) \right) \right\vert \geq L\frac{\left( \ln n\right) ^{2}}{n}\right) \leq 22rn^{-2-\alpha _{\mathcal{%
M}}}.  \label{ineq_excess_petit_V_fold}
\end{equation}
\end{crllr}

\begin{proof}
First we prove the following inequality,%
\begin{equation}\resizebox{.99\hsize}{!}{$
\mathbb{P}\left( \left\vert \frac{1}{V}\sum_{j=1}^{V}P\left( \gamma \left( 
\widehat{s}_{m}^{(-j)}\right) -\gamma \left( s_{m}\right) \right) -\frac{1}{V%
}\sum_{j=1}^{V}P_{n}^{(j)}\left( \gamma \left( \widehat{s}_{m}^{(-j)}\right)
-\gamma \left( s_{m}\right) \right) \right\vert \geq \varepsilon
_{n}^{\left( 1\right) }\left( m\right) \mathbb{E}\left[ 
%TCIMACRO{\TeXButton{p}{\p}}%
%BeginExpansion
\p%
%EndExpansion
_{2}^{(-1)}\left( m\right) \right] \right) \leq 12Vn^{-2-\alpha _{\mathcal{M}%
}}.$}  \label{ineq_intermed_V_fold}
\end{equation}%
Indeed, it easily derives from Lemma \ref{lemma_hold_out} together with a
union bound along the $V$ blocks, taking advantage of the following formula%
\begin{eqnarray*}
&&\left\vert \frac{1}{V}\sum_{j=1}^{V}P\left( \gamma \left( \widehat{s}%
_{m}^{(-j)}\right) -\gamma \left( s_{m}\right) \right) -\frac{1}{V}%
\sum_{j=1}^{V}P_{n}^{(j)}\left( \gamma \left( \widehat{s}_{m}^{(-j)}\right)
-\gamma \left( s_{m}\right) \right) \right\vert  \\
&\leq &\max_{j\in \left\{ 1,\ldots,V\right\} }\left\vert P_{n}^{(j)}\left(
\gamma \left( \widehat{s}_{m}^{(-j)}\right) -\gamma \left( s_{m}\right)
\right) -P\left( \gamma \left( \widehat{s}_{m}^{(-j)}\right) -\gamma \left(
s_{m}\right) \right) \right\vert .
\end{eqnarray*}%
Then, we show that the quantity $\frac{1}{V}\sum_{j=1}^{V}P\left( \gamma
\left( \widehat{s}_{m}^{(-j)}\right) -\gamma \left( s_{m}\right) \right) $
is close enough to $P\left( \gamma \left( \widehat{s}_{m}^{(-1)}\right)
-\gamma \left( s_{m}\right) \right) $ with probability close to one. Indeed,
it holds for any $C\geq 0$,%
\begin{eqnarray*}
&&\mathbb{P}\left( \left\vert P\left( \gamma \left( \widehat{s}%
_{m}^{(-1)}\right) -\gamma \left( s_{m}\right) \right) -\frac{1}{V}%
\sum_{j=1}^{V}P\left( \gamma \left( \widehat{s}_{m}^{(-j)}\right) -\gamma
\left( s_{m}\right) \right) \right\vert \geq C\right)  \\
&=&\mathbb{P}\left( \left\vert \frac{1}{V}\sum_{j=2}^{V}\left[ P\left(
\gamma \left( \widehat{s}_{m}^{(-1)}\right) -\gamma \left( s_{m}\right)
\right) -P\left( \gamma \left( \widehat{s}_{m}^{(-j)}\right) -\gamma \left(
s_{m}\right) \right) \right] \right\vert \geq C\right)  \\
&\leq &\mathbb{P}\left( \max_{j\in \left\{ 2,\ldots,V\right\} }\left\vert
P\left( \gamma \left( \widehat{s}_{m}^{(-1)}\right) -\gamma \left(
s_{m}\right) \right) -P\left( \gamma \left( \widehat{s}_{m}^{(-j)}\right)
-\gamma \left( s_{m}\right) \right) \right\vert \geq C\right)  \\
&\leq &\sum_{j=2}^{V}\mathbb{P}\left( \left\vert P\left( \gamma \left( 
\widehat{s}_{m}^{(-1)}\right) -\gamma \left( s_{m}\right) \right) -P\left(
\gamma \left( \widehat{s}_{m}^{(-j)}\right) -\gamma \left( s_{m}\right)
\right) \right\vert \geq C\right)  \\
&=&\sum_{j=2}^{V}\mathbb{P}\left( \left\vert \left[ P\left( \gamma \left( 
\widehat{s}_{m}^{(-1)}\right) -\gamma \left( s_{m}\right) \right) -\frac{%
\mathcal{C}_{m}}{n}\right] -\left[ P\left( \gamma \left( \widehat{s}%
_{m}^{(-j)}\right) -\gamma \left( s_{m}\right) \right) -\frac{\mathcal{C}_{m}%
}{n}\right] \right\vert \geq C\right)  \\
&\leq &2V\mathbb{P}\left( \left\vert P\left( \gamma \left( \widehat{s}%
_{m}^{(-1)}\right) -\gamma \left( s_{m}\right) \right) -\frac{\mathcal{C}_{m}%
}{n}\right\vert \geq \frac{C}{2}\right) .
\end{eqnarray*}%
Hence, from Theorem 2 of \cite{saum:12}\ applied with $\alpha =2+\alpha _{%
\mathcal{M}}$ and sample size equal to $n_{V}=nV/\left( V-1\right) $, we get
that by taking 
\begin{equation*}
C=2\varepsilon _{n_{V}}\left( m\right) \frac{\mathcal{C}_{m}}{n_{V}}\leq L_{%
\mathrm{(\mathbf{GSA}),}r}\varepsilon _{n}\left( m\right) \mathbb{E}\left[ 
%TCIMACRO{\TeXButton{p}{\p}}%
%BeginExpansion
\p%
%EndExpansion
_{2}^{(-1)}\left( m\right) \right] ,
\end{equation*}%
it holds%
\begin{equation}
\mathbb{P}\left( \left\vert P\left( \gamma \left( \widehat{s}%
_{m}^{(-1)}\right) -\gamma \left( s_{m}\right) \right) -\frac{1}{V}%
\sum_{j=1}^{V}P\left( \gamma \left( \widehat{s}_{m}^{(-j)}\right) -\gamma
\left( s_{m}\right) \right) \right\vert \geq C\right) \leq 10Vn^{-2-\alpha _{%
\mathcal{M}}}.  \label{ineq_P_V_fold}
\end{equation}%
Inequality (\ref{ineq_excess_grand_V_fold}) now follows from combining (\ref%
{ineq_intermed_V_fold}) with (\ref{ineq_P_V_fold}) and noticing that $%
\varepsilon _{n}^{(1)}\left( m\right) \leq L_{\mathrm{(\textbf{GSA})}%
,r}\varepsilon _{n}\left( m\right) $. Inequality (\ref%
{ineq_excess_petit_V_fold}) also derives from Lemma \ref{lemma_hold_out}\
with the same type of reasoning and further details are left to the reader.
\end{proof}

\begin{lmm}
\label{lemma_delta_VFCV}Let $\alpha >0$. Assume that (\textbf{GSA}) is
satisfied and that $1< V\leq r$. Then by setting 
\begin{equation*}
\bar{\delta}\left( m\right) =\frac{1}{V}\sum_{j=1}^{V}\left( P_{n}^{\left(
j\right) }-P\right) \left( \gamma \left( s_{m}\right) -\gamma \left( s_{\ast
}\right) \right) \text{ \ \ and \ \ }\bar{\delta}^{\prime }\left( m\right) =%
\frac{1}{V}\sum_{j=1}^{V}\left( P_{n}^{\left( -j\right) }-P\right) \left(
\gamma \left( s_{m}\right) -\gamma \left( s_{\ast }\right) \right) ,
\end{equation*}%
we have for all $m\in \mathcal{M}_{n}$, 
\begin{equation}
\mathbb{P}\left( \max \left\{ \left\vert \bar{\delta}\left( m\right)
\right\vert ,\left\vert \bar{\delta}^{\prime }\left( m\right) \right\vert
\right\} \geq L_{(\mathrm{\mathbf{GSA}}),r}\left( \sqrt{\frac{\ell \left( s_{\ast
},s_{m}\right) \ln n}{n}}+\frac{\ln n}{n}\right) \right) \leq 2rn^{-\alpha }%
\text{ }.  \label{delta_bar_small}
\end{equation}%
Furthermore, for all $m\in \mathcal{M}_{n}$ such that $A_{\mathcal{M}%
,+}\left( \ln n\right) ^{2}\leq D_{m}$ and for all $n\geq n_{0}\left( (%
\mathbf{GSA}),\alpha \right) $, we have%
\begin{equation*}
\mathbb{P}\left( \max \left\{ \left\vert \bar{\delta}\left( m\right)
\right\vert ,\left\vert \bar{\delta}^{\prime }\left( m\right) \right\vert
\right\} \geq \frac{\ell \left( s_{\ast },s_{m}\right) }{\sqrt{D_{m}}}+L_{(%
\mathbf{GSA}),r}\frac{\ln n}{\sqrt{D_{m}}}\mathbb{E}\left[ 
%TCIMACRO{\TeXButton{p}{\p}}%
%BeginExpansion
\p%
%EndExpansion
_{2}^{(-1)}\left( m\right) \right] \right) \leq 2rn^{-\alpha },
%\label{delta_bar_reasonable}
\end{equation*}%
where $%
%TCIMACRO{\TeXButton{p}{\p}}%
%BeginExpansion
\p%
%EndExpansion
_{2}^{(-1)}\left( m\right) :=P_{n}^{\left( -1\right) }\left( \gamma \left(
s_{m}\right) -\gamma \left( \widehat{s}_{m}^{(-1)}\right) \right) \geq 0$.
\end{lmm}

\begin{proof}
Notice that for any $C\geq 0$,%
\begin{eqnarray*}
\mathbb{P}\left( \left\vert \bar{\delta}\left( m\right) \right\vert \geq
C\right) &\leq &\mathbb{P}\left( \max_{j\in \left\{ 1,\ldots,V\right\}
}\left\vert \left( P_{n}^{\left( j\right) }-P\right) \left( \gamma \left(
s_{m}\right) -\gamma \left( s_{\ast }\right) \right) \right\vert \geq
C\right) \\
&\leq &\sum_{j=1}^{V}\mathbb{P}\left( \left\vert \left( P_{n}^{\left(
j\right) }-P\right) \left( \gamma \left( s_{m}\right) -\gamma \left( s_{\ast
}\right) \right) \right\vert \geq C\right) .
\end{eqnarray*}%
Then use $j$ times Lemma 5 of \cite{saum:13} with a sample size equal to $%
n/V $ in order to control the summands at the right-hand side of the
inequality in the last display. The same reasoning holds for $\left\vert 
\bar{\delta}^{\prime }\left( m\right) \right\vert $. Further details are
left to the reader.
\end{proof}

\smallskip

\begin{lmm}
\label{lemma_termes_croises}Let $\alpha >0$. Consider a finite-dimensional
linear model $m$ of linear dimension $D$ and assume that $\left( \varphi
_{k}\right) _{k=1}^{D_{m}}$ is a localized orthonormal basis of $\left(
m,\left\Vert 
%TCIMACRO{\TeXButton{.}{\cdot}}%
%BeginExpansion
\cdot%
%EndExpansion
\right\Vert _{2}\right) $ with index of localization $r_{m}>0$. More
explicitly, we thus assume that for all $\beta =\left( \beta _{k}\right)
_{k=1}^{D_{m}}\in \mathbb{R}^{D_{m}}$,%
\begin{equation*}
\left\Vert \sum_{k=1}^{D_{m}}\beta _{k}\varphi _{k}\right\Vert _{\infty }\leq
r_{m}\sqrt{D_{m}}\left\vert \beta \right\vert _{\infty }.
\end{equation*}%
If (\textbf{Ab}$\left( m\right) $) given in Theorem \ref{theorem_general_sup}%
\ holds and if for some positive constant $A_{+}$,%
\begin{equation*}
D_m\leq A_{+}\frac{n}{\left( \ln n\right) ^{2}},
\end{equation*}%
then there exists a positive constant $L_{\alpha ,r_{m}}^{(2)}$ such that
for all $n\geq n_{0}\left( A_{+}\right) $, we have%
\begin{equation}
\mathbb{P}\left( \max_{\left( k,l\right) \in \left\{ 1,\ldots,D_m\right\}
^{2}}\left\vert \left( P_{n}-P\right) \left( \varphi _{k}%
%TCIMACRO{\TeXButton{.}{\cdot}}%
%BeginExpansion
\cdot%
%EndExpansion
\varphi _{l}\right) \right\vert \geq L_{\alpha ,r_{m}}^{(2)}\min \left\{
\left\Vert \varphi _{k}\right\Vert _{\infty };\left\Vert \varphi
_{l}\right\Vert _{\infty }\right\} \sqrt{\frac{\ln n}{n}}\right) \leq
n^{-\alpha }.  \label{dev_croises}
\end{equation}
\end{lmm}

\begin{proof}
For any $%
\left( k,l\right) \in \left\{ 1,\ldots,D_m\right\} ^{2}$, we have 
\begin{equation*}
\mathbb{E}\left[ \left( \varphi _{k}%
%TCIMACRO{\TeXButton{.}{\cdot}}%
%BeginExpansion
\cdot%
%EndExpansion
\varphi _{l}\right) ^{2}\left( X\right) \right] \leq \min \left\{ \left\Vert
\varphi _{k}\right\Vert _{\infty }^{2};\left\Vert \varphi _{l}\right\Vert
_{\infty }^{2}\right\}
\end{equation*}%
and%
\begin{eqnarray*}
\left\Vert \varphi _{k}%
%TCIMACRO{\TeXButton{.}{\cdot}}%
%BeginExpansion
\cdot%
%EndExpansion
\varphi _{l}\right\Vert _{\infty } &\leq &\min \left\{ \left\Vert \varphi
_{k}\right\Vert _{\infty };\left\Vert \varphi _{l}\right\Vert _{\infty
}\right\} \times \max \left\{ \left\Vert \varphi _{k}\right\Vert _{\infty
};\left\Vert \varphi _{l}\right\Vert _{\infty }\right\} \\
&\leq &\min \left\{ \left\Vert \varphi _{k}\right\Vert _{\infty };\left\Vert
\varphi _{l}\right\Vert _{\infty }\right\} \times r_{m}\sqrt{D_{m}}.
\end{eqnarray*}%
Hence, we apply Bernstein's inequality (see Proposition 2.9 in \cite{Massart:07}) and we get, for all $\gamma >0$,%
\begin{equation}
\mathbb{P}\left( \left\vert \left( P_{n}-P\right) \left( \varphi _{k}%
%TCIMACRO{\TeXButton{.}{\cdot}}%
%BeginExpansion
\cdot%
%EndExpansion
\varphi _{l}\right) \right\vert \geq \min \left\{ \left\Vert \varphi
_{k}\right\Vert _{\infty };\left\Vert \varphi _{l}\right\Vert _{\infty
}\right\} \left( \sqrt{\frac{2\gamma \ln n}{n}}+\frac{r_{m}\sqrt{D_{m}}\gamma
\ln n}{3n}\right) \right) \leq 2n^{-\gamma }.
\label{bernstein_dev_reg_sup}
\end{equation}%
Since, for all $n\geq n_{0}\left( A_{+}\right) $, 
\begin{equation*}
\frac{r_{m}\sqrt{D_{m}}\ln n}{n}\leq \frac{r_{m}\sqrt{A_{+}}}{\sqrt{\ln n}}%
%TCIMACRO{\TeXButton{.}{\cdot}}%
%BeginExpansion
\cdot%
%EndExpansion
\sqrt{\frac{\ln n}{n}}\leq r_{m}\sqrt{\frac{\ln n}{n}},
\end{equation*}%
we get from (\ref{bernstein_dev_reg_sup}) that for all $n\geq n_{0}\left(
A_{+}\right) $,%
\begin{eqnarray}
&&\mathbb{P}\left( \max_{\left( k,l\right) \in \left\{ 1,\ldots,D_m\right\}
^{2}}\left\vert \left( P_{n}-P\right) \left( \varphi _{k}%
%TCIMACRO{\TeXButton{.}{\cdot}}%
%BeginExpansion
\cdot%
%EndExpansion
\varphi _{l}\right) \right\vert \geq \left( \sqrt{2\gamma }+\frac{\gamma
r_{m}}{3}\right) \min \left\{ \left\Vert \varphi _{k}\right\Vert _{\infty
};\left\Vert \varphi _{l}\right\Vert _{\infty }\right\} \sqrt{\frac{\ln n}{n}%
}\right)  \notag \\
&\leq &\sum_{\left( k,l\right) \in \left\{ 1,\ldots,D_m\right\} ^{2}}\mathbb{P}%
\left( \left\vert \left( P_{n}-P\right) \left( \varphi _{k}%
%TCIMACRO{\TeXButton{.}{\cdot}}%
%BeginExpansion
\cdot%
%EndExpansion
\varphi _{l}\right) \right\vert \geq \left( \sqrt{2\gamma }+\frac{\gamma
r_{m}}{3}\right) \min \left\{ \left\Vert \varphi _{k}\right\Vert _{\infty
};\left\Vert \varphi _{l}\right\Vert _{\infty }\right\} \sqrt{\frac{\ln n}{n}%
}\right)  \notag \\
&\leq &\sum_{\left( k,l\right) \in \left\{ 1,\ldots,D_m\right\} ^{2}}\mathbb{P}%
\left( \left\vert \left( P_{n}-P\right) \left( \varphi _{k}%
%TCIMACRO{\TeXButton{.}{\cdot}}%
%BeginExpansion
\cdot%
%EndExpansion
\varphi _{l}\right) \right\vert \geq \min \left\{ \left\Vert \varphi
_{k}\right\Vert _{\infty };\left\Vert \varphi _{l}\right\Vert _{\infty
}\right\} \sqrt{\frac{2\gamma \ln n}{n}}+\frac{r_{m}\sqrt{D_{m}}\gamma \ln n}{3n}%
\right)  \notag \\
&\leq &2D^{2}n^{-\gamma }\leq n^{-\gamma +2}.  \label{dev_2_reg_sup}
\end{eqnarray}%
We deduce from (\ref{dev_2_reg_sup}) that (\ref{dev_croises}) holds with $%
L_{\alpha,r_m}^{(2)}=\sqrt{2\alpha +4}+\left( \alpha +2\right) r_{m}/3>0$.\ 
\end{proof}

\begin{lmm}
\label{lemma_dev_phi1_reg_sup}Under the assumptions of Lemma \ref{lemma_termes_croises} there exists a positive constant $L_{A,r_{m},\alpha }^{(1)}$ such that
for all $n\geq n_{0}\left( A_{+}\right) $, we have%
\begin{equation*}
\mathbb{P}\left( \max_{k\in \left\{ 1,\ldots,D_m\right\} }\left\vert \left(
P_{n}-P\right) \left( \psi _{m}%
%TCIMACRO{\TeXButton{.}{\cdot}}%
%BeginExpansion
\cdot%
%EndExpansion
\varphi _{k}\right) \right\vert \geq L_{A,r_{m},\alpha }^{(1)}\sqrt{\frac{%
\ln n}{n}}\right) \leq n^{-\alpha },  %\label{dev_phi1_reg_sup}
\end{equation*}%
where $\psi _{m}\left( x,y\right) =-2\left( y-s_{m}\left( x\right) \right) $.
\end{lmm}

\begin{proof}
Let $\beta
>0 $. Notice that by (\textbf{Ab}$\left( m\right) $), 
\begin{equation*}
\left\vert \psi _{m}\left( X,Y\right) \right\vert \leq 4A\text{ \ \ }a.s.
\end{equation*}%
Then by Bernstein's inequality, we get by straightforward computations (in
the spirit of the proof of Lemma \ref{lemma_termes_croises}) that there
exists $L {(1)}_{A,r_{m},\beta }>0$ such that, for all $k\in \left\{
1,\ldots,D_m\right\} $,%
\begin{equation*}
\mathbb{P}\left( \left\vert \left( P_{n}-P\right) \left( \psi _{m}%
%TCIMACRO{\TeXButton{.}{\cdot}}%
%BeginExpansion
\cdot%
%EndExpansion
\varphi _{k}\right) \right\vert \geq L_{A,r_{m},\beta }^{(1)}\sqrt{\frac{\ln
n}{n}}\right) \leq n^{-\beta }.
\end{equation*}%
Now the result follows from a simple union bound with $\beta =\alpha +1$. 
\end{proof}

\subsection{Additional simulation results}

This section provides additional simulation results to those in Section \ref{section_experiments}. Figure~\ref{fig:singleWave} is an analogy 	to Figure~\ref{fig:singleSpikes} which illustrates the difference between the test functions \textit{Spikes} and \textit{Wave} for a smaller sample size (i.e. $n=1024$).

%Two typical examples of estimation from a single simulation for the \textit{Wave} and \textit{Spikes} functions are depicted in Figures~\ref{fig:singleWave}(e)  \ref{fig:singleSpikes}(e) for $n$-samples of sizes $1024$ and $4096$ respectively.
% Indeed, for virtually all cases, SH consistently selects the best model with the exception of the (very smooth) \textit{Wave} function for which a smaller model is selected in the \textit{Low Homoscedastic Noise}.

\begin{figure}
\centering
\subfigure[$\sigma_{l1}(x)=0.01$]{\includegraphics[width=0.23\textwidth]{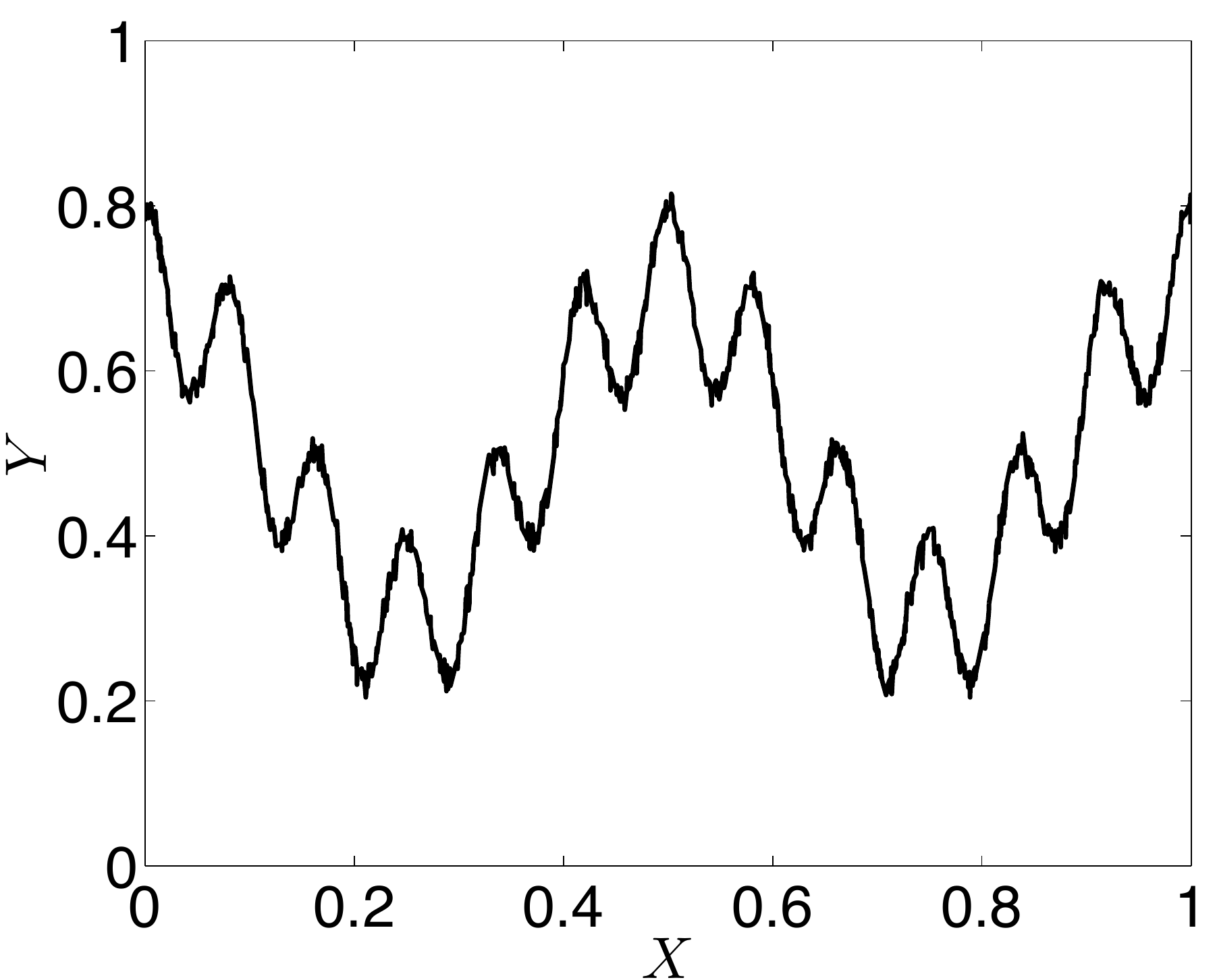}}
\subfigure[$\sigma_{l2}(x)=0.02x$]{\includegraphics[width=0.23\textwidth]{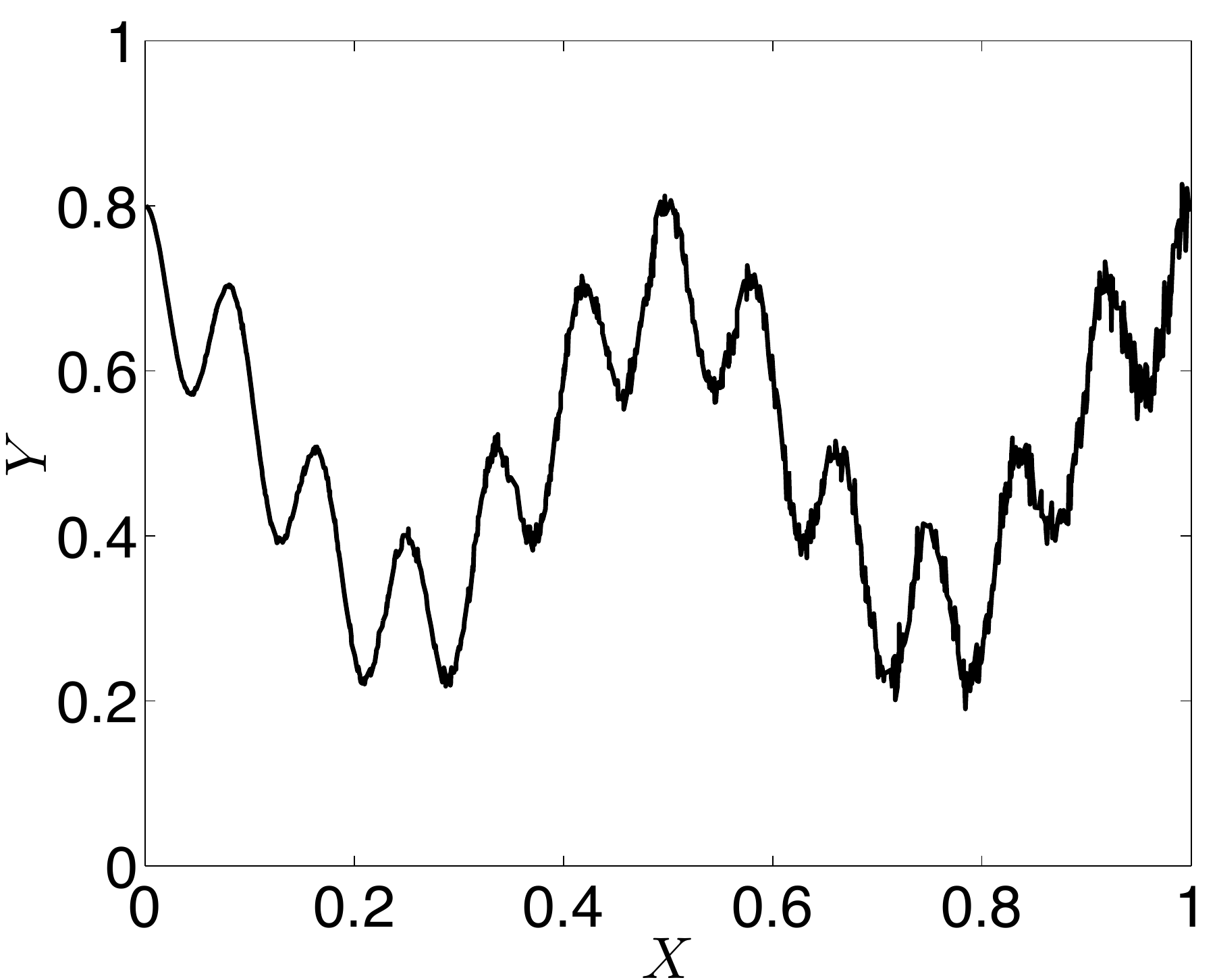}}
\subfigure[$\sigma_{h1}(x)=0.05$]{\includegraphics[width=0.23\textwidth]{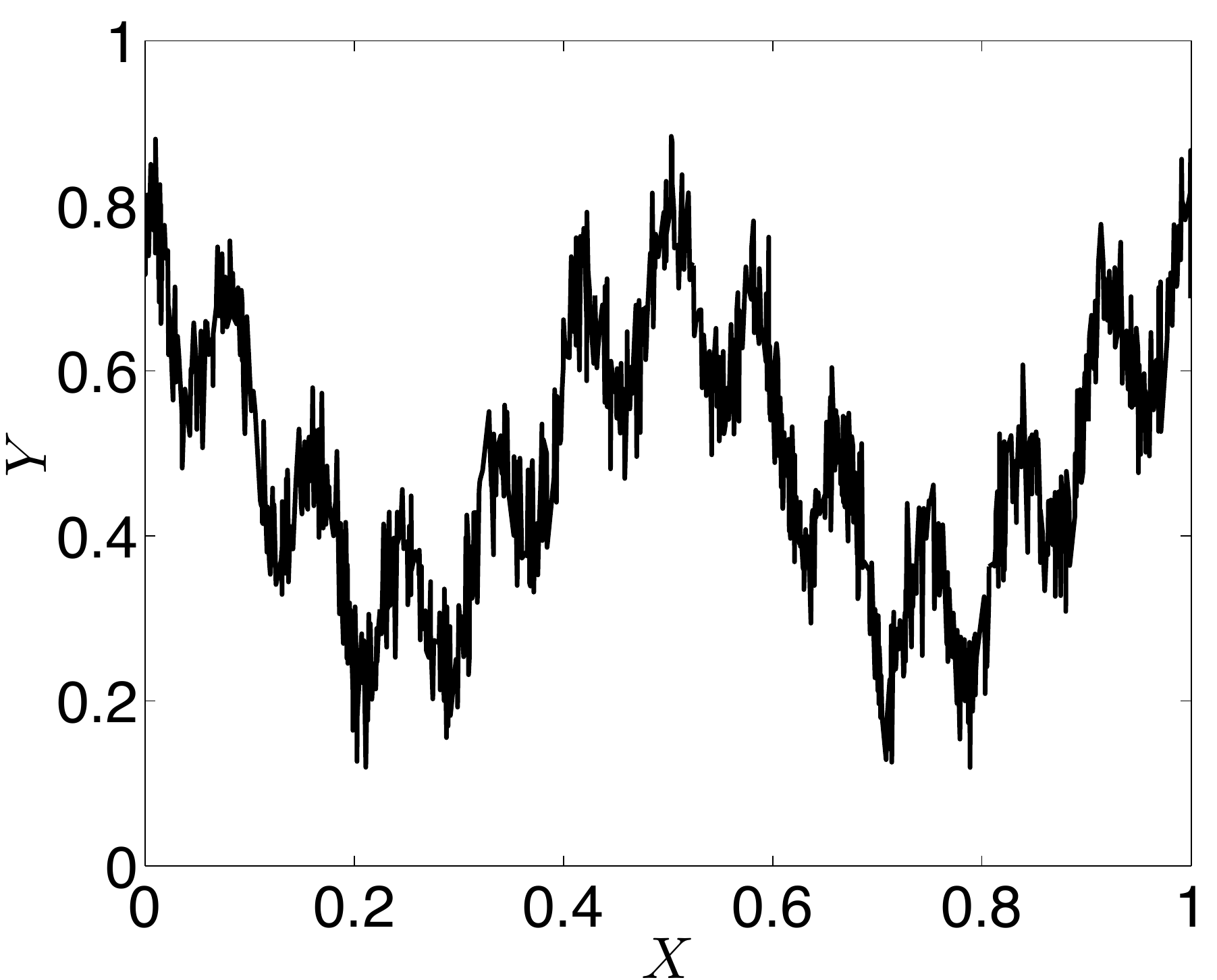}}
\subfigure[$\sigma_{h2}(x)=0.1x$]{\includegraphics[width=0.23\textwidth]{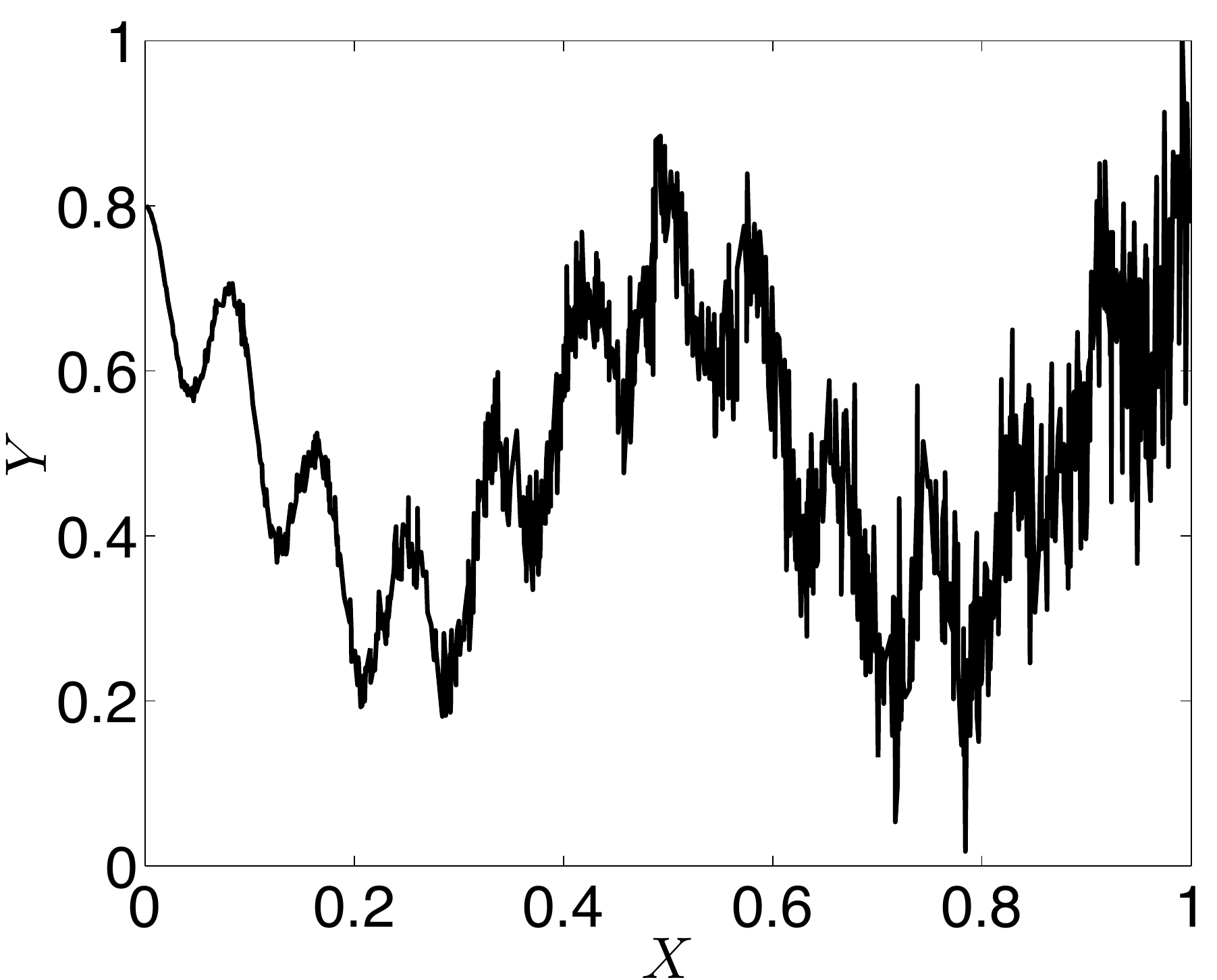}}
\subfigure[]{
\includegraphics[width=0.23\textwidth]{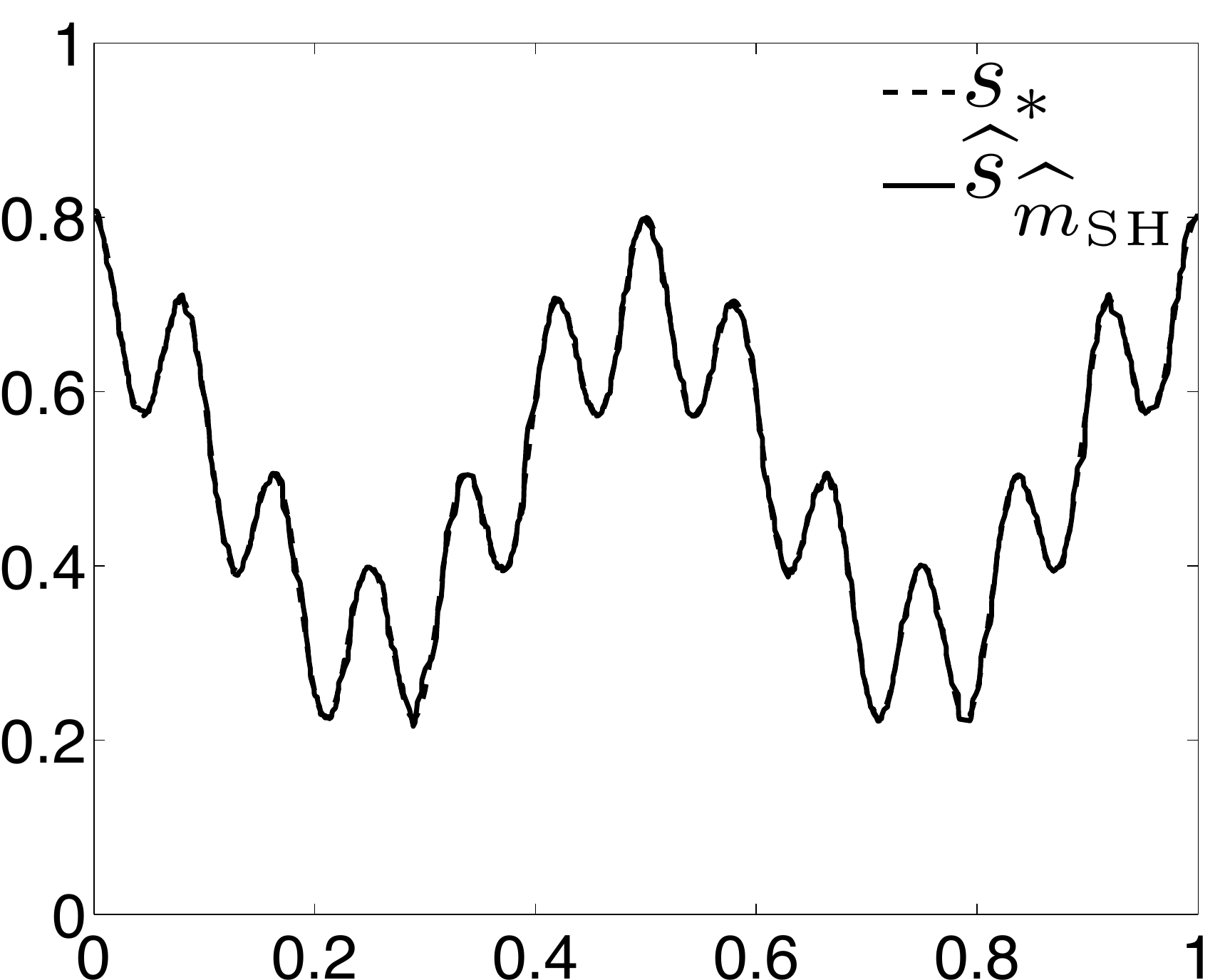}
\includegraphics[width=0.23\textwidth]{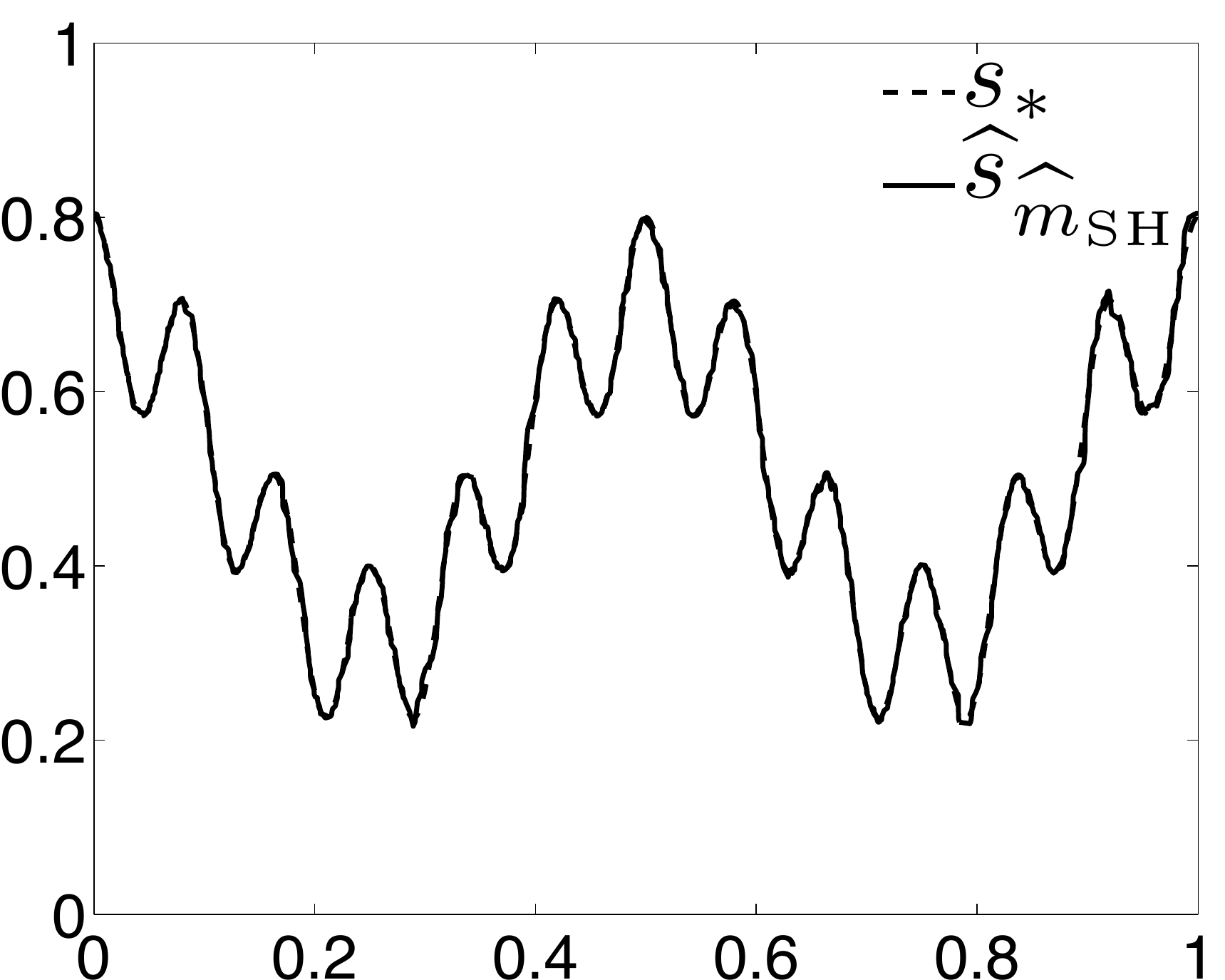}
\includegraphics[width=0.23\textwidth]{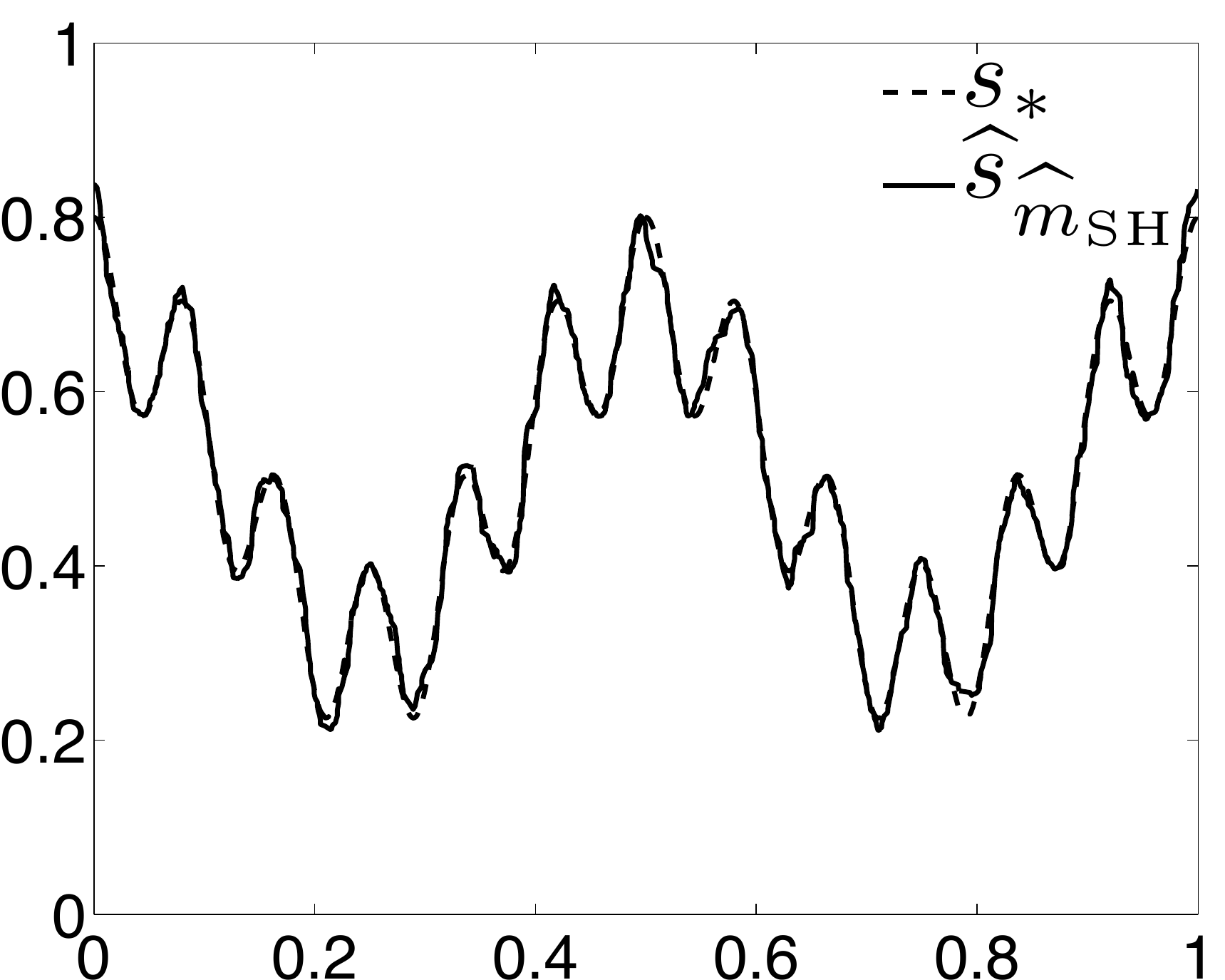}
\includegraphics[width=0.23\textwidth]{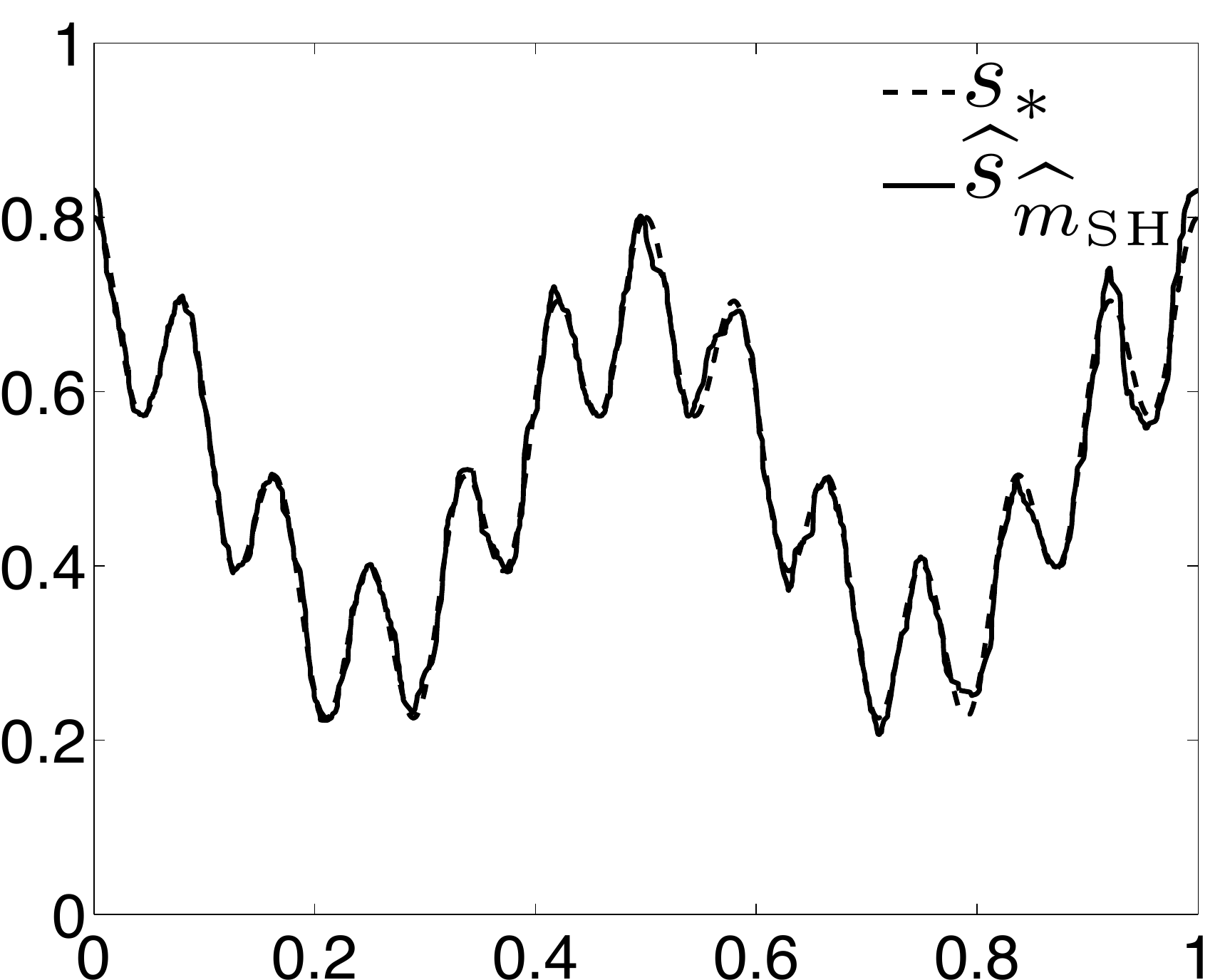}}
\subfigure[]{
\includegraphics[width=0.23\textwidth]{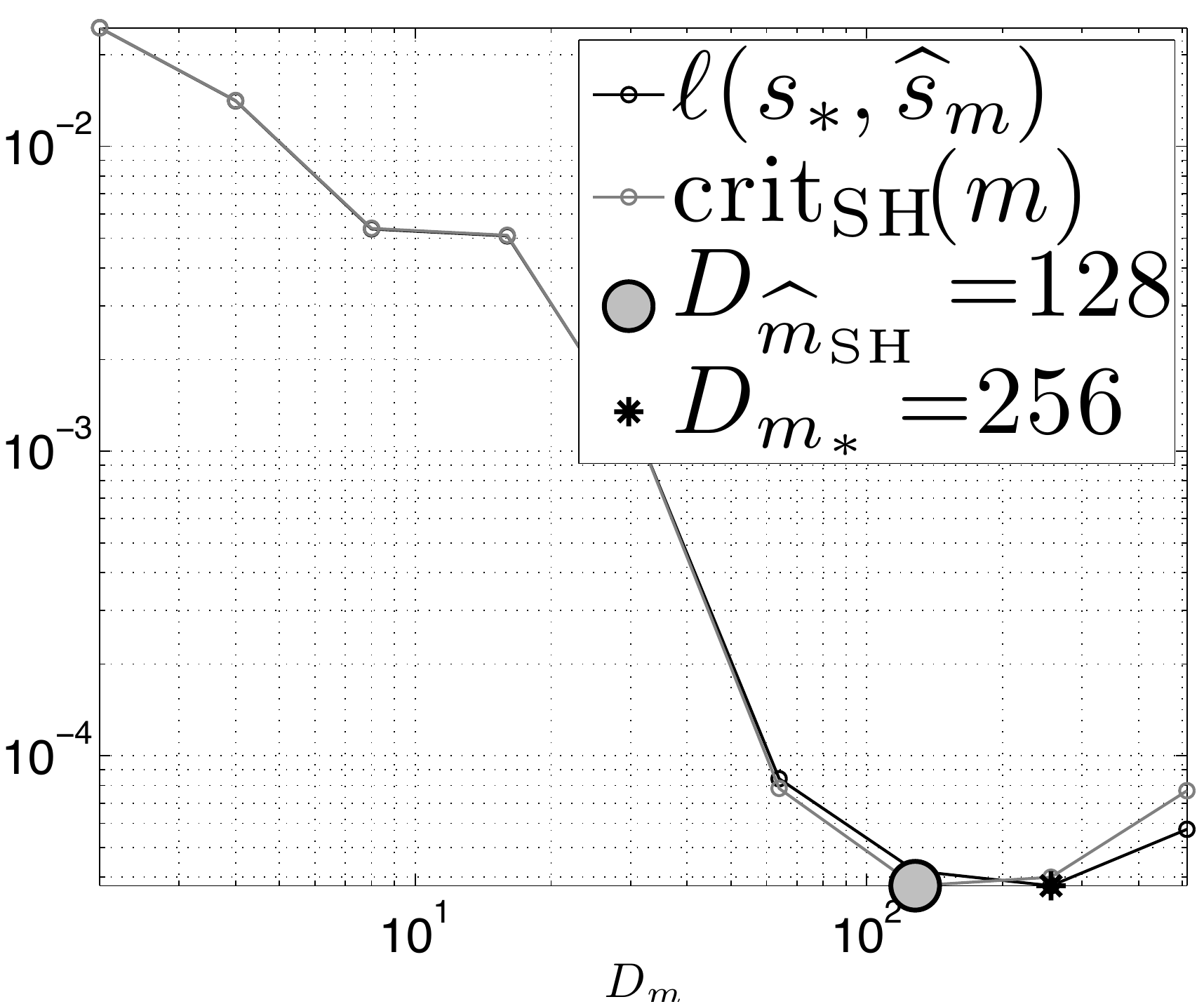}
\includegraphics[width=0.23\textwidth]{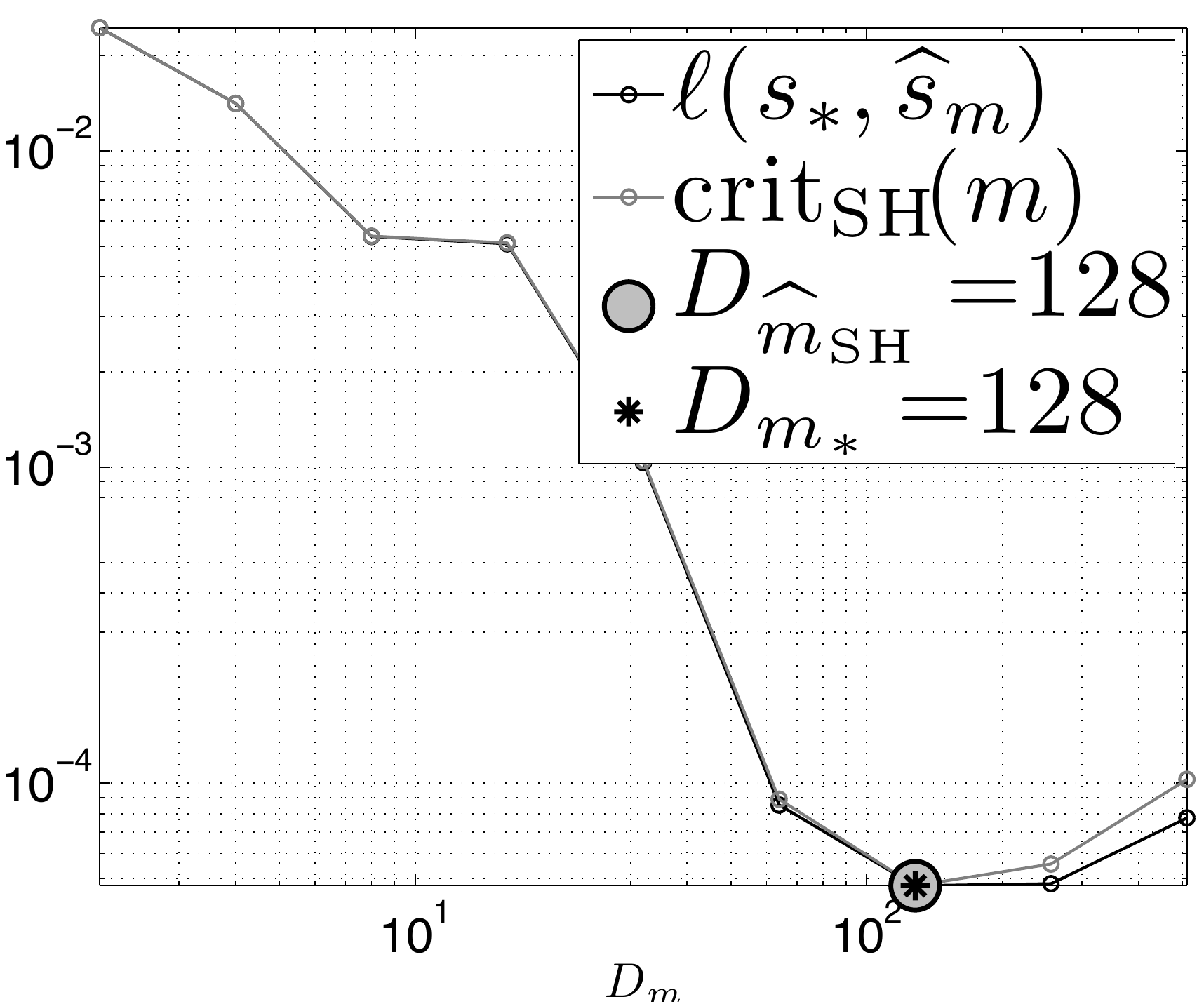}
\includegraphics[width=0.23\textwidth]{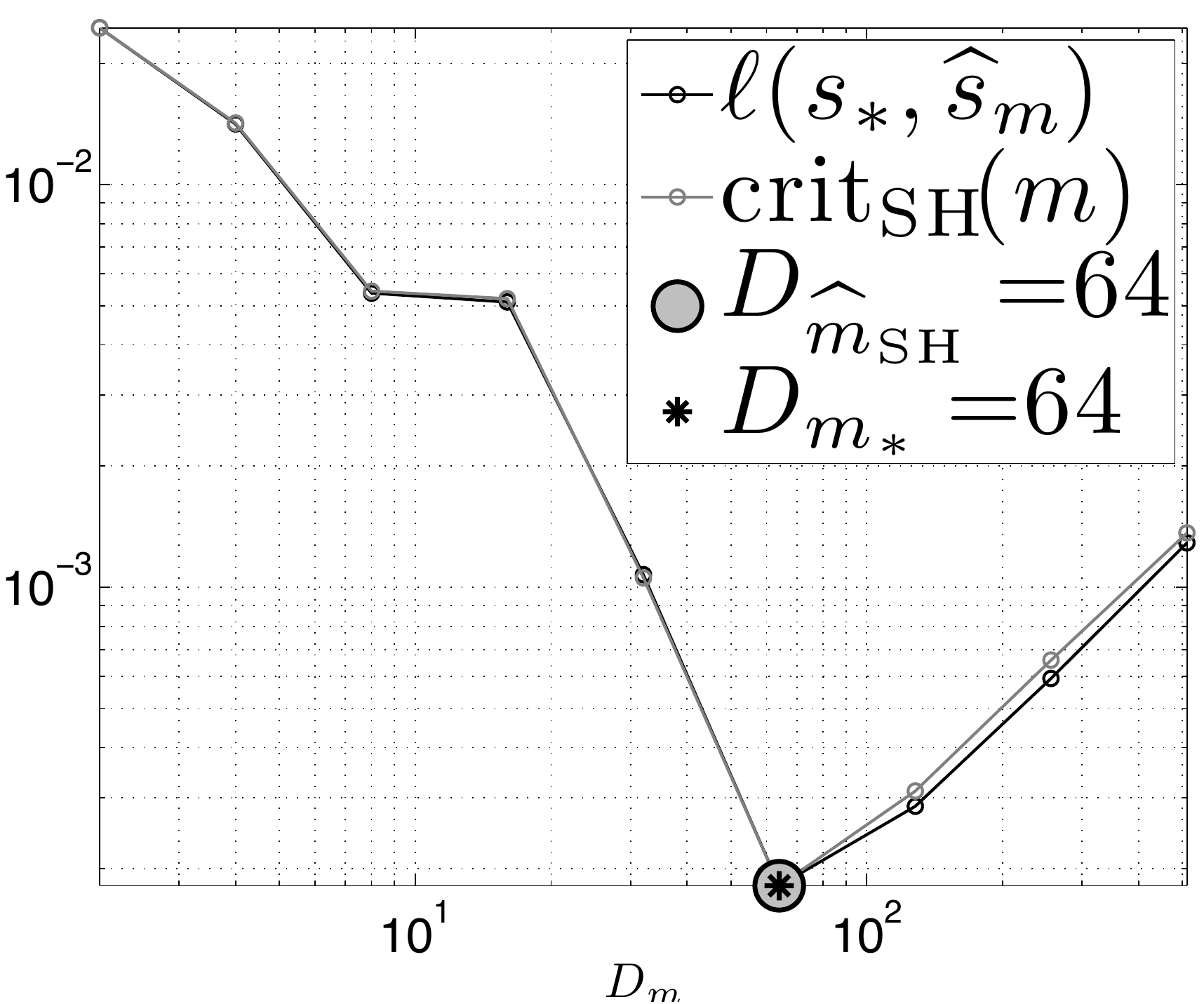}
\includegraphics[width=0.23\textwidth]{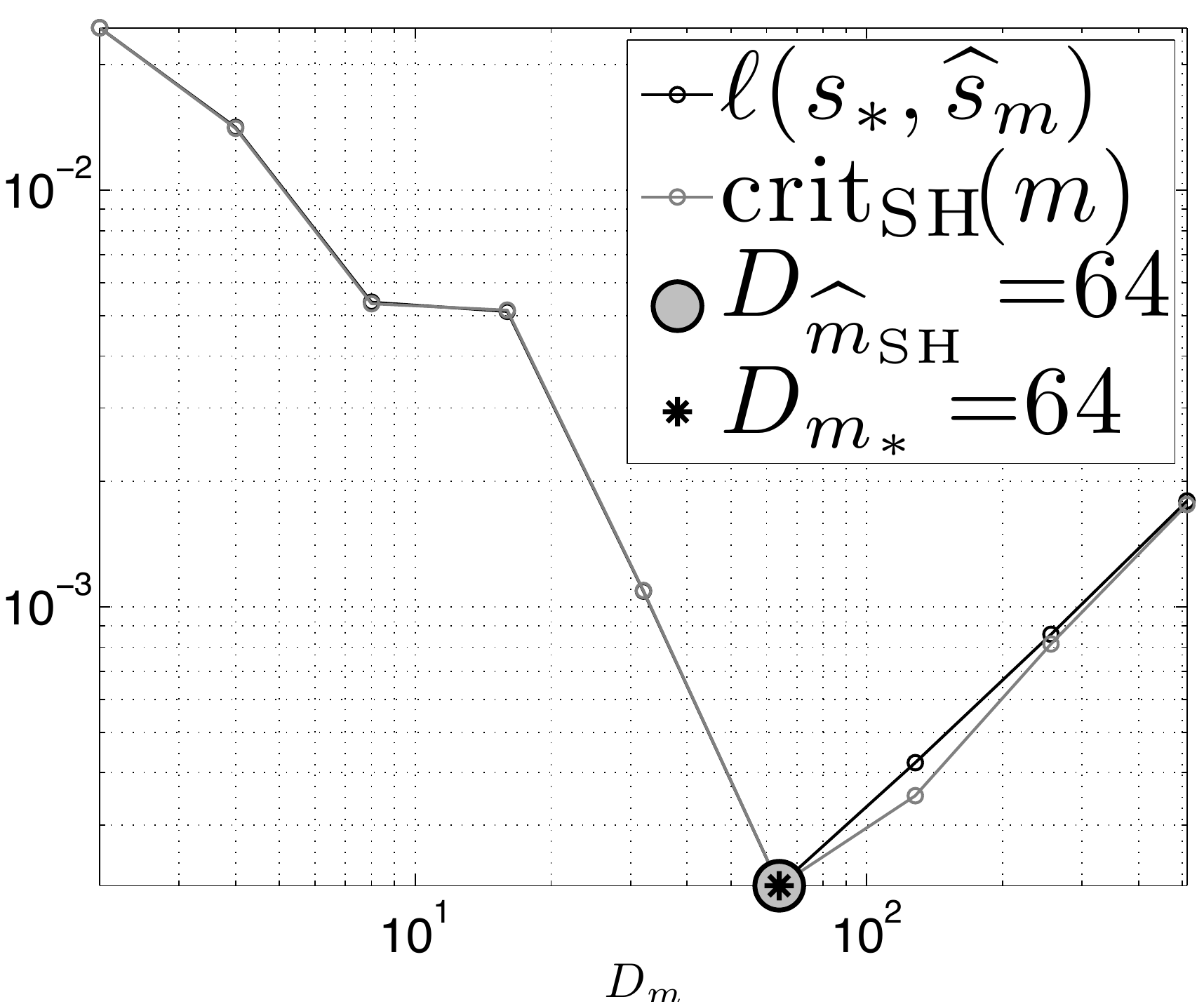}}
\subfigure[]{
\includegraphics[width=0.23\textwidth]{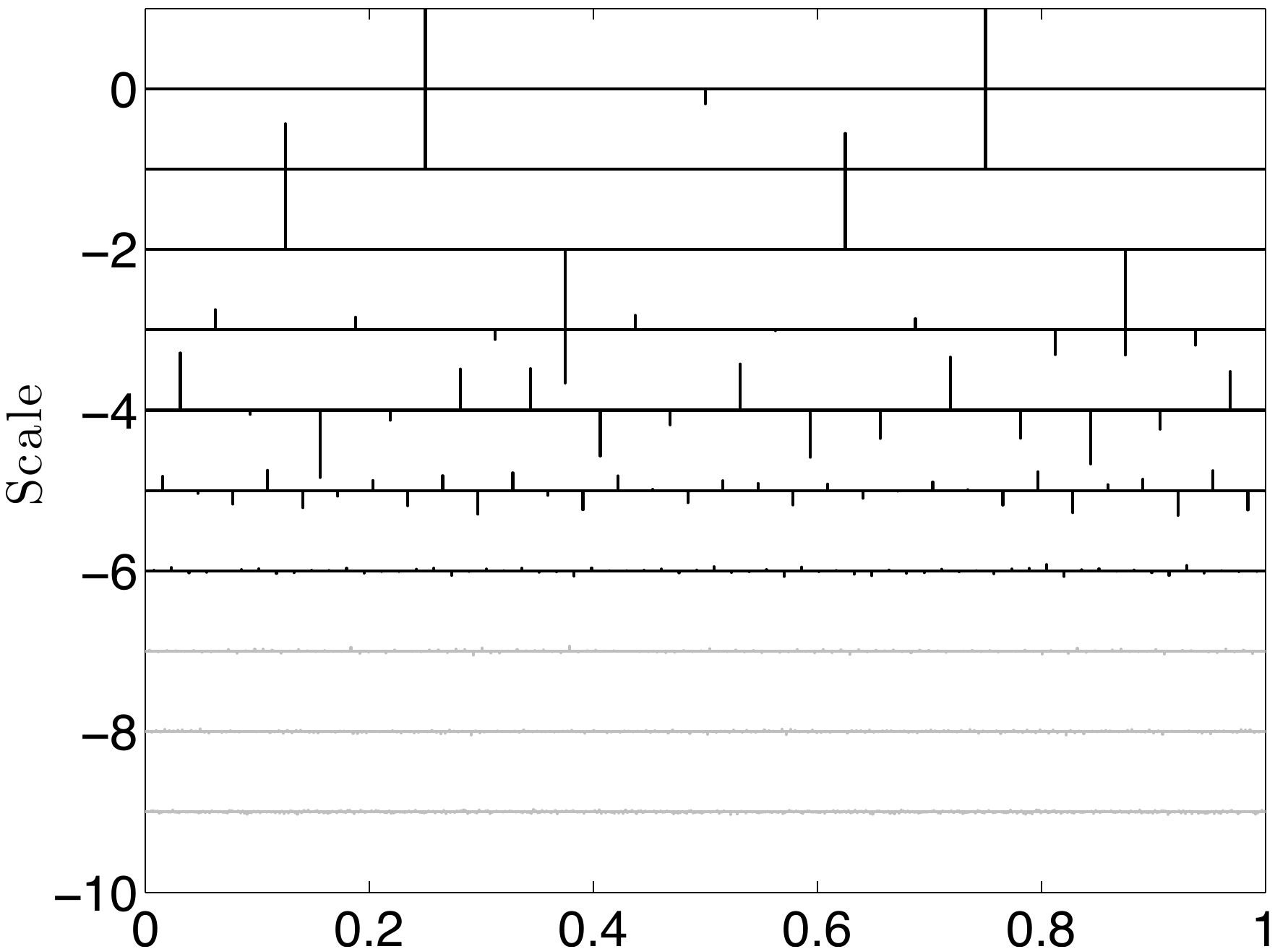}
\includegraphics[width=0.23\textwidth]{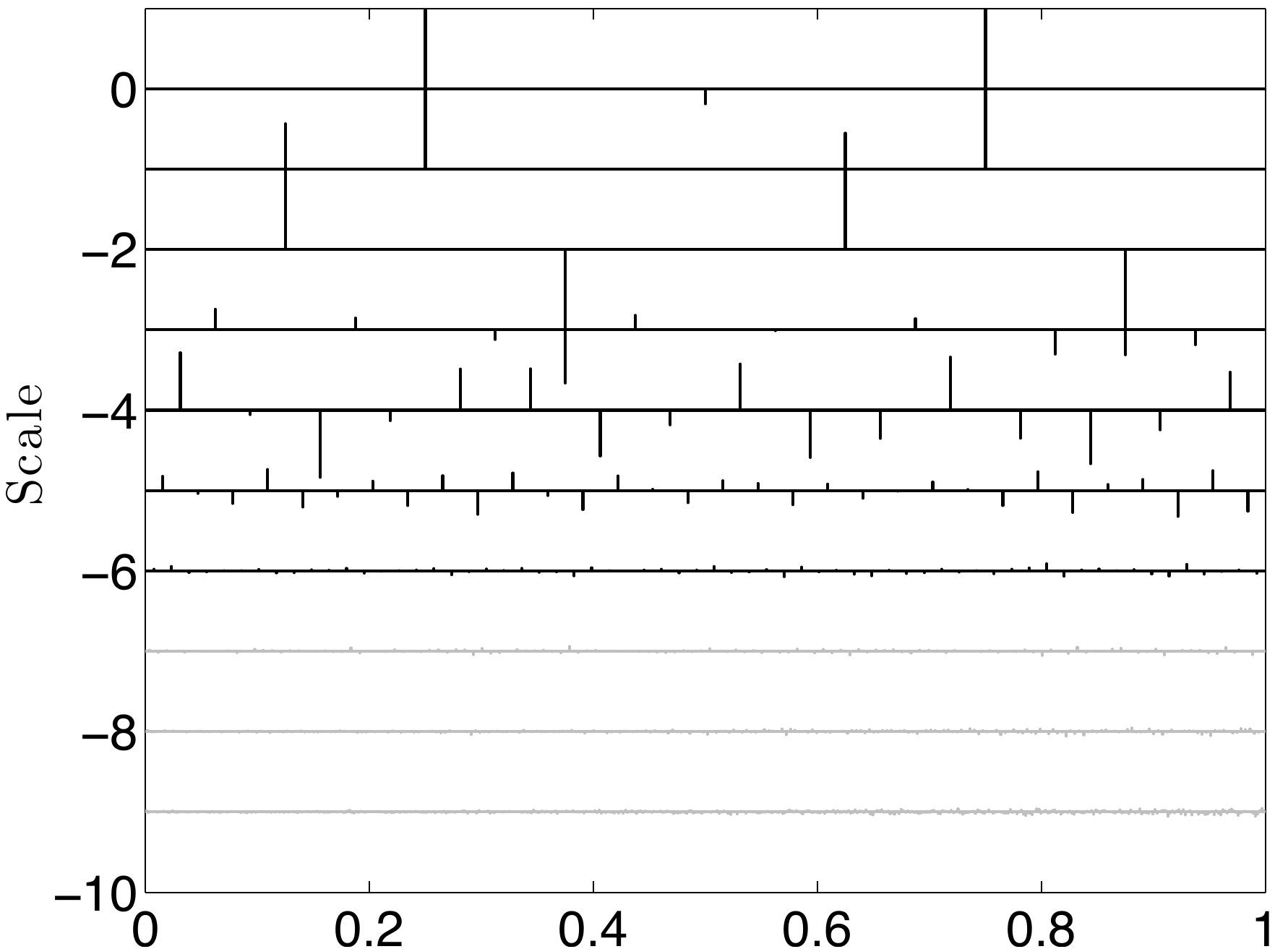}
\includegraphics[width=0.23\textwidth]{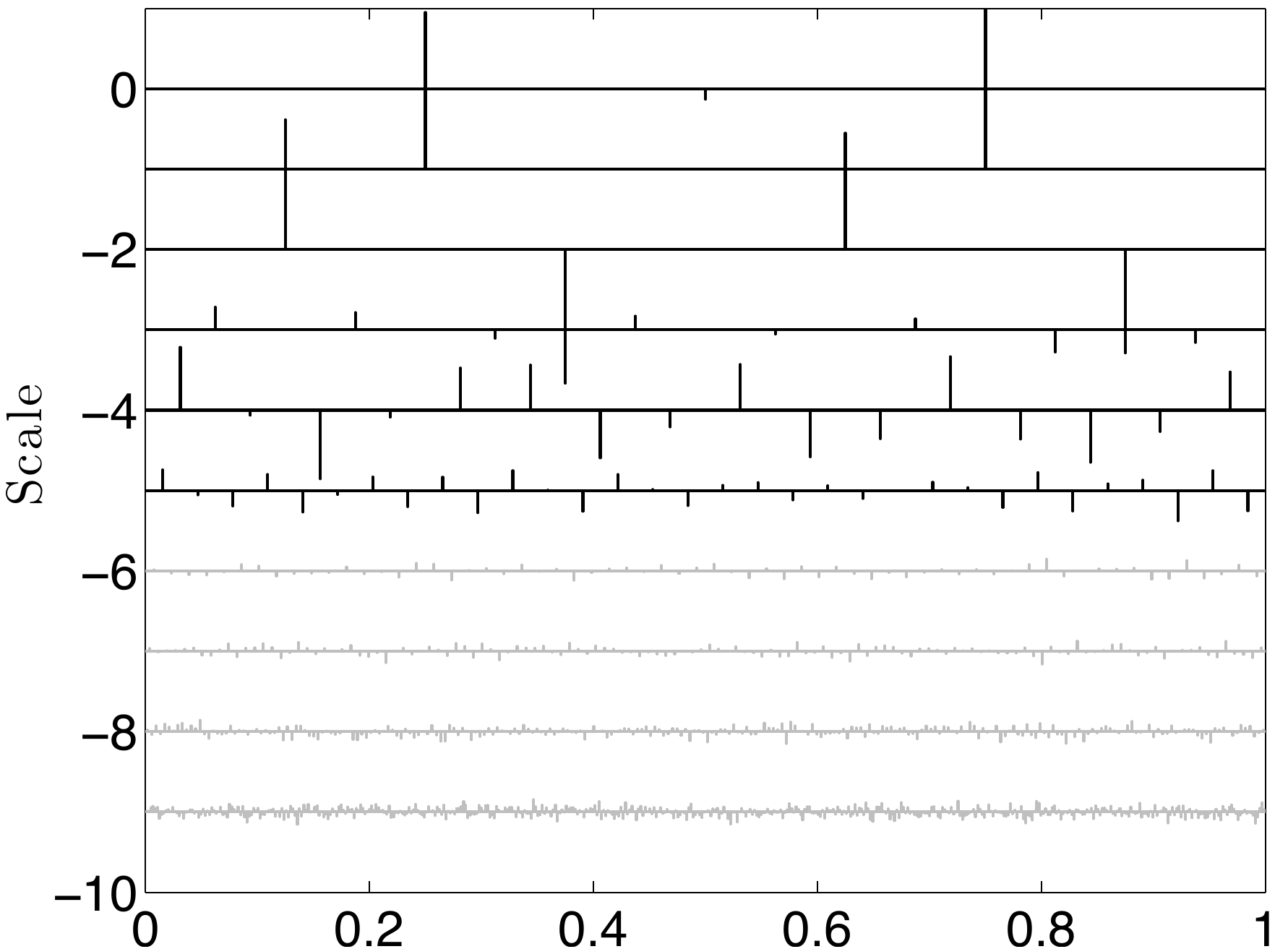}
\includegraphics[width=0.23\textwidth]{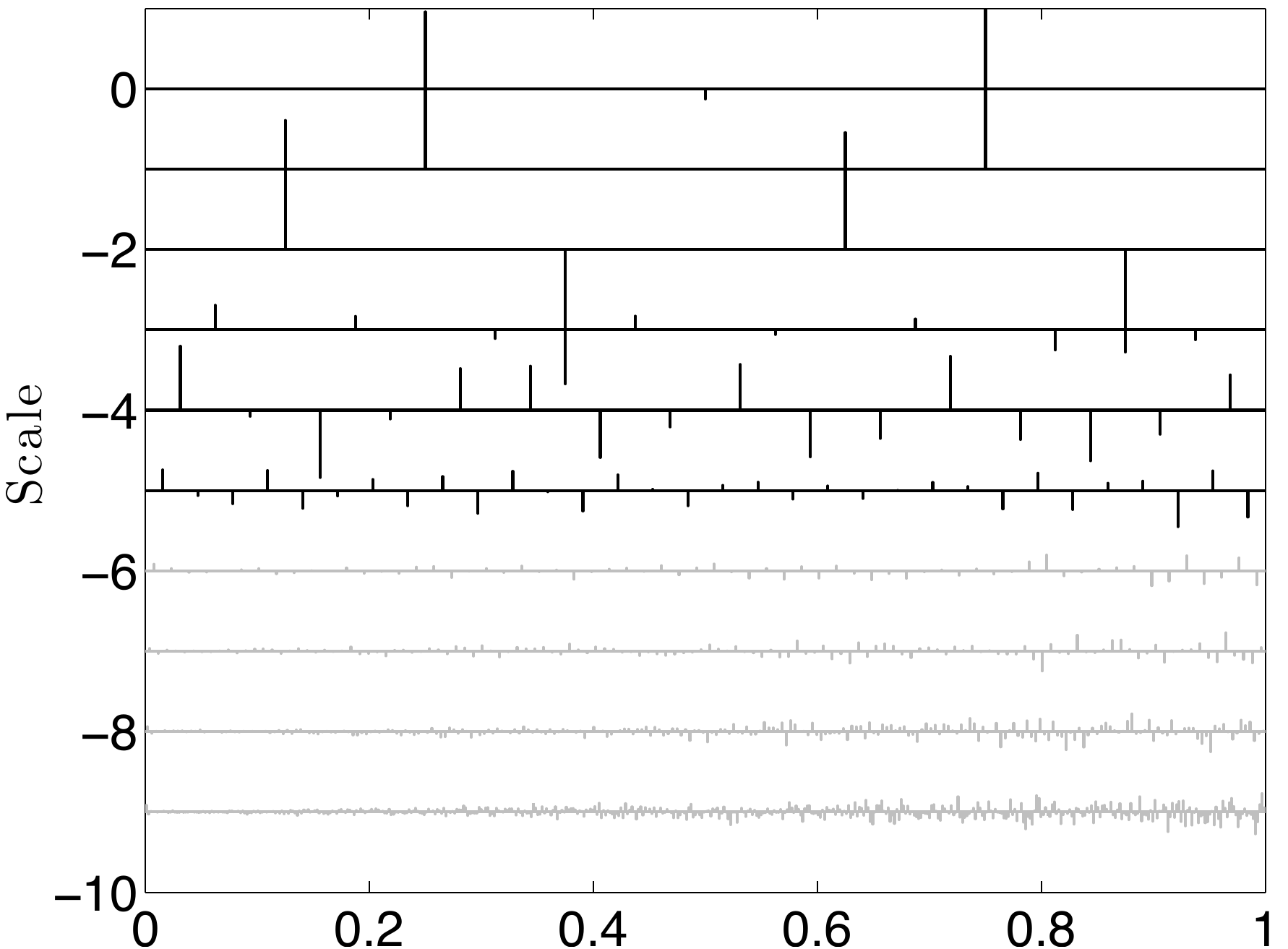}}
\caption{(a)-(d): Noisy version of \textit{Wave} for each $\sigma(\cdot)$ scenarios. (e): Typical reconstructions from a single simulation with $n=1024$. The dotted line is the true signal and the solid one depicts the estimates $\widehat{s}_{\widehat{m}_{\mathrm{SH}}}$. (f): Graph of the excess risk $\ell(s_\ast, \widehat{s}_m)$ against the dimension $D_m$ and (shifted) $\mathrm{crit}_{\mathrm{SH}}(m)$ (in a log-log scale). The gray circle represents the global minimizer $\widehat{m}$ of $\mathrm{crit_{\mathrm{SH}}}(m)$ and the black star the oracle model $m_\ast$. (g): Noisy and selected (black) wavelet coefficients (see Figure~\ref{fig:target}(e) for a visual comparison with the original wavelet coefficients).}
\label{fig:singleWave}
\end{figure} 

\end{document}